\documentclass[10pt,reqno]{amsart}
\allowdisplaybreaks
\usepackage{amscd,amssymb,amsopn,amsmath,amsthm,graphics,amsfonts,accents,enumerate,verbatim,calc}
\usepackage{tikz-cd}
\usepackage{graphicx}
\usepackage{mathrsfs}
\usepackage{xcolor}
\calclayout
\makeatletter
\renewcommand{\@seccntformat}[1]{\bf\csname the#1\endcsname.}
\renewcommand{\section}{\@startsection{section}{1}
	\z@{.7\linespacing\@plus\linespacing}{.5\linespacing}
	{\normalfont\upshape\bfseries\centering}}
\renewcommand{\@biblabel}[1]{\@ifnotempty{#1}{#1.}}
\makeatother
\theoremstyle{plain}
\newtheorem{thm}{Theorem}[section]
\newtheorem{lem}[thm]{Lemma}
\newtheorem{prop}[thm]{Proposition}
\newtheorem{cor}[thm]{Corollary}

\theoremstyle{definition}
\newtheorem{ex}[thm]{Example}
\newtheorem{defn}[thm]{Definition}
\newtheorem{rem}{Remark}[section]

\usepackage{tikz-cd}
\usepackage[german]{varioref}
\usepackage[all]{xy}
\usepackage{cancel}

\def\>{\succ}
\def\<{\prec}

\def\O{\Omega}

\begin{document}

\title[Sania Asif \textsuperscript{1}, Zhixiang Wu\textsuperscript{2}
 ]{Quasi-Twilled Lie Pseudolgebras and Their Deformation Maps}
%%%%%%%%%%%%%%%%%%%%%%%%%%%%%%%%%%%%%%%%%%%%%%%%%%%%%%%%%%%%%%%%%%%%%%%%%%%%%%%%
\author{Sania Asif\textsuperscript{1}, Zhixiang Wu\textsuperscript{2}
}
\address{\textsuperscript{1} Institute of Mathematics, Henan Academy of Sciences, Zhengzhou, 450046, P.R. China.}
\address{\textsuperscript{2}School of Mathematical Sciences, Zhejiang University, Hangzhou, Zhejiang Province, 310058, P.R. China} 
\email{\textsuperscript{1}11835037@zju.edu.cn}
\email{\textsuperscript{2}wzx@zju.edu.cn}
	\keywords{Quasi twilled Lie pseudoalgebra; controlling algebra; $\mathfrak{L}_\infty$-pseudoalgebra; cohomology; deformation}
	\subjclass[2020]{Primary 17B65, 17B10, 17B38, 17B69, Secondary 14F35, 18N10, 18N40.}%%%%%%%%%
	\date{\today}
	\thanks{This work is funded by the Second batch of the Provincial project of the Henan Academy of Sciences (No. 241819105). This paper is also sponsored
by NNSFC (No.12471038, No.12171129) and ZJNNF(No.Z25A010006).}
%%%%%%%%%%%%%%%%%%%%%%%%%%%%%%%%%%%%
	\begin{abstract}
In this paper, we present a unified framework for studying cohomology theories of various operators in the context of pseudoalgebras. The central tool in our approach is the notion of a quasi-twilled Lie pseudoalgebra. We introduce two types of deformation maps. Type I unifies modified $r$ matrices, crossed homomorphisms, derivations, and homomorphisms; and Type II provides a uniform treatment of relative Rota-Baxter operators, twisted Rota-Baxter operators, Reynolds operators, and deformation maps of matched pairs of Lie conformal algebras.
We construct the corresponding controlling algebras and define cohomology theories for both types of deformation maps. These results recover existing cohomological results for known operators and yield new results, including the cohomology theory for modified $r$-matrices and deformation maps of matched pairs of Lie pseudoalgebras.
\end{abstract}

 \footnote{*The corresponding authors' emails: wzx@zju.edu.cn; 11835037@zju.edu.cn.}
	\maketitle \section{ Introduction}
	
 \par The study of algebraic structures through their cohomology and deformation has long been a central theme in modern pure mathematics and mathematical physics. The origins of deformation theory can be traced back to the foundational work of Gerstenhaber \cite{G1,G2} on associative algebras and Nijenhuis-Richardson \cite{N1, N2} on Lie algebras. Their insights revealed that the deformation problem for many algebraic structures could be encoded within a differential graded (dg) Lie algebra or its homotopy analog. This principle was often attributed to Deligne, Drinfeld, and Kontsevich. This philosophy was later made precise by Lurie \cite{Lurie} and Pridham \cite{Pridham}, who showed that under appropriate conditions, every reasonable deformation theory corresponds to an $\mathcal{L}_\infty$-algebra whose Maurer-Cartan elements encode the structures being deformed. This theory is extended to operator-type structures (i.e., linear maps compatible with algebraic operations), including morphisms \cite{ALWH, Figueroa, Fregier}, derivations \cite{TFS}, $\mathcal{O}$ operator (relative Rota-Baxter operators) \cite{AWW, AWMB, Das, Das2, TBGS}, crossed homomorphisms \cite{HongSu, Lue}, twisted Rota-Baxter operators \cite{AYW, YL}, Reynolds-type operators \cite{Das3}, averaging operators \cite{AW}, and Nijenhuis operators \cite{A, Asif, AWY}. A common strategy in these studies involves the construction of a controlling algebra. That is typically a $\mathcal{L}_\infty$-algebra, whose Maurer-Cartan elements correspond precisely to the given operators. By twisting this algebra with a known solution, one can derive the cochain complex governing infinitesimal deformations and obstructions.
\par In particular, the concept of a twilled Lie algebra, introduced originally in the context of Lie bialgebras and Poisson geometry, has provided a powerful framework for understanding various operator-type structures such as modified $r$ matrices, relative Rota-Baxter operators, crossed homomorphisms, and deformation maps of matched pairs of Lie algebras \cite{AM, BD, JST}. This framework has since been extended to other algebraic settings, including pre-Lie algebras and quasi-pre-Lie bialgebras \cite{Liu2020}, associative algebras via quasi-twilled associative algebras and their governing $\mathcal{L}_\infty$-algebras \cite{DM}, and higher algebraic structures such as $3$-Lie algebras \cite{HSZ}. These structures are governed by a rich interplay between homotopy algebras \cite{LS}, operads \cite{Vallette}, and cohomology theories, often encoded in differential graded Lie algebras or curved $\mathcal{L}_\infty$-algebras whose Maurer-Cartan elements represent the structures under investigation.
\par Despite substantial progress in the classical algebra setting, the development of analogous results in infinite-dimensional contexts, particularly for Lie pseudoalgebras, has remained relatively limited. In this paper, we extend the notion of quasi-twilled Lie algebras to the Lie pseudoalgebras, which provides an algebraic formalism for studying vertex algebras, Lie conformal algebra, and chiral symmetry algebras arising in two-dimensional conformal field theory. We introduce the concept of a quasi-twilled Lie pseudoalgebra, which includes direct sums, semidirect products, action algebras, and matched pairs of Lie pseudoalgebras as special cases. Within this framework, we define two types of deformation maps:
\begin{enumerate}
\item[i.] \textbf{Type I} deformation maps unify modified $r$-matrices, crossed homomorphisms, derivations and homomorphisms between Lie pseudoalgebras.
\item[ii.] \textbf{Type II} deformation maps unify relative Rota-Baxter operators, twisted Rota-Baxter operators, Reynolds-type operators and deformation maps of pairs of matched Lie pseudoalgebras.
\end{enumerate}For each type, we construct a corresponding curved $\mathcal{L}_\infty$- pseudoalgebra, whose Maurer-Cartan elements precisely encode the respective deformation maps. This construction provides the first known controlling algebra for modified $r$ matrices on Lie pseudoalgebras. The modified $r$ matrices are solutions to the modified classical Yang-Baxter equation. 
Furthermore, we develop the associated Chevalley-Eilenberg cohomology theories for these deformation maps, thereby establishing a rigorous foundation for their infinitesimal deformations and obstruction theory. Our approach recovers known results in the context of pseudoalgebras, and offers a unified perspective on the deformation and cohomology of a wide class of operator algebras.
 \par This paper is structured as follows. In Section $2$, we define the notion of a quasi-twilled Lie pseudoalgebra, which generalizes the concept of matched pairs and provides a unified framework for various algebraic structures such as direct sums, semidirect products, and action algebras in the context of Lie pseudoalgebras. Several examples are provided to illustrate the generality and applicability of this concept. Moreover, we classify quasi-twilled Lie $H$-pseudoalgebras that are free $H$-modules of rank $2$. In Section $3$, we introduce type I deformation maps and construct their corresponding curved algebra $\mathcal{L}_\infty$-algebra, whose Maurer-Cartan elements encode these maps. We then derive the associated cohomology theory, which governs infinitesimal deformations of type I maps. Section $4$ extends this analysis to type II deformation maps, showing that they unify relative Rota-Baxter operators, twisted Rota-Baxter operators, Reynolds-type operators, and deformation maps of matched pairs of Lie pseudoalgebras. For type II map, we construct its controlling curved $\mathcal{L}_\infty$-pseudoalgebra and develop the corresponding cohomology theory. 
 \par Throughout the paper, all vector spaces, tensor products, and modules are on a field {\bf k} of zero characteristic and $H$-modules denote left $H$-modules, where $H$ is a cocommutative Hopf algebra on the field $\bf k$. The tensor product over $ \mathbf{k}$ will be denoted by $ \otimes $.  Moreover, we distinguish between two types of direct sums. The symbol $ \boxplus $ denotes the direct sum of $ H $-modules. For two $ H $-modules $ \mathfrak{g} $ and $ \mathfrak{h} $, $ \mathfrak{g} \boxplus \mathfrak{h} $ is their direct sum as left $ H $-modules. The symbol $ \oplus $ denotes the direct sum of Lie $ H $-pseudoalgebras. For two Lie $ H $-pseudoalgebras $ (\mathfrak{g}, [\cdot * \cdot]_\mathfrak{g}) $ and $ (\mathfrak{h}, [\cdot * \cdot]_\mathfrak{h}) $, their direct sum $ \mathfrak{g} \oplus \mathfrak{h} $ is the Lie $ H $-pseudoalgebra on $ \mathfrak{g} \boxplus \mathfrak{h} $.
\section{ Quasi-twilled Lie $H$-Pseuodalgebras}
In this section, we introduce the notion of a quasi-twilled Lie $H$-pseudoalgebra and illustrate it with several examples. Before proceeding to the core content, we recall foundational ideas that underpin the construction and properties of such structures.
\subsection{ Preliminaries on Pseudoalgebras}
For any natural number $ n \in \mathbb{N} $, the group of permutations of a finite set with $ n $ elements is denoted by $ S_n $, $S_{p, n-p} $ denotes the set of $(p,n-p)$-shuffles in the permutation group $ S_n $.

\noindent We adopt the following convention throughout the paper, i.e., 
if $ f : \bigotimes E_i \to F $ is a map, then we sometimes write
$$
f(x_1, x_2, \dots, x_n) := f(x_1 \otimes x_2 \otimes \cdots \otimes x_n),
$$
where $ x_1, \dots, x_n $ are elements of the corresponding spaces.

\noindent Let $ \alpha \in S_n $ be an element of the symmetric group on $ n $ letters. Then the action of $ \alpha $ on a tensor $ x_1 \otimes \cdots \otimes x_n $ is defined by
$$
\alpha(x_1 \otimes \cdots \otimes x_n) := x_{\alpha(1)} \otimes \cdots \otimes x_{\alpha(n)}.
$$
This action will be used to define equivariance conditions on multilinear operations and cochains.
\begin{defn}
A \emph{$\mathbf{k}$-linear operad} $\mathcal{P}$ is a collection $\{\mathcal{P}(n) \mid n \geq 1\}$ of vector spaces over $\mathbf{k}$, equipped with the following structure:\\
For each $n \geq 1$, there is an action of the symmetric group $S_n$ on $\mathcal{P}(n)$. There are linear maps, called compositions,
\[
\gamma_{m_1, \dots, m_t} : \mathcal{P}(t) \otimes \mathcal{P}(m_1) \otimes \cdots \otimes \mathcal{P}(m_t) \to \mathcal{P}(m_1 + \cdots + m_t),
\]
defined for all $m_1, \dots, m_t \geq 1$. Additionally, there exists a unit element $1 \in \mathcal{P}(1)$ such that
\[
\mu(1, \dots, 1) = \mu
\]
for any $t \geq 1$ and any $\mu \in \mathcal{P}(t)$. Moreover, these structure maps must satisfy the conditions of associativity
 and equivariance. We denote the operads associated to associative algebras and Lie algebras by $\mathcal{A}ss$ and $\mathcal{L}ie$, respectively.%: The composition maps are associative in the sense of operads. %: The compositions are equivariant with respect to the natural actions of symmetric groups.
\end{defn}
\noindent Assume that $ H $ is a cocommutative Hopf algebra over $ \mathbf{k} $, with a coassociative coproduct $ \Delta $, a counit $ \varepsilon $, and an antipode $ S $.
We use the notation (cf. \cite{Sweedler}),
\begin{align*}
 \Delta^p(h) &= h_{(1)} \otimes \cdots \otimes h_{(p+1)},\\
 (S^{\otimes p} \otimes 1^{\otimes q})\Delta^{p+q-1}(h) &= h_{(-1)} \otimes \cdots \otimes h_{(-p)} \otimes h_{(p+1)} \otimes \cdots \otimes h_{(p+q)},
\end{align*} for any $ h \in H $, where $ \Delta^p $ denotes the $ p $-iterated coproduct.\\
Let $\mathfrak{g}$ be a left $ H $-module in the pseudotensor category $\mathcal{M}^*(H) $. A morphism of the form $$\mathrm{Hom}_{H^{\otimes(p+1)}}\left(\mathfrak{g}^{\otimes(p+1)}, H^{\otimes(p+1)} \otimes_H \mathfrak{g}\right)$$ is referred to as an $ H $-\emph{pseudoproduct} of $\mathfrak{g}$ of degree $ p $. We denote the space of $ H $-pseudoproduct of $\mathfrak{g}$ of degree $ p\geq 0 $ by $ M^p(\mathfrak{g})$. Additionally, we set $ M^{p}(\mathfrak{g}) = \mathfrak{g}$, for $p=-1$ and $ M^p(\mathfrak{g}) = 0 $ for $ p < -1 $.\\
For any permutation $ \sigma \in S_{p+1} $ and $ f \in M^p(\mathfrak{g}) $, the action of the symmetric group on $f$ is defined by
\begin{align}\label{eq1}\sigma(f)(a_1, \dots, a_{p+1}) := (\sigma \otimes_H 1)\, f(a_{\sigma^{-1}(1)}, \dots, a_{\sigma^{-1}(p+1)}), \quad\text{ for all } a_1, \dots, a_{p+1} \in \mathfrak{g}.
\end{align} This equips $ M^p(\mathfrak{g}) $ with the structure of a $ \mathbf{k}[S_{p+1}] $-module.
\begin{defn}
We refer to $\mathrm{End}(\mathfrak{g})$ as the \emph{pseudo-operad} associated with $\mathfrak{g}$.
\end{defn} 
\begin{defn}
 Let $ f $ be a $ \mathbf{k} $-linear map from $ \mathfrak{g}^{\otimes n} $ to $ H^{\otimes(n-1)} \otimes \mathfrak{g}$. We say that $ f $ is \emph{conformal} if it satisfies the following two conditions:
\begin{align*}
 f(h_1 a_1 \otimes \cdots \otimes h_{n-1} a_{n-1} \otimes a_n) &= (h_1 \otimes \cdots \otimes h_{n-1} \otimes 1)\,f(a_1 \otimes \cdots \otimes a_n),\\
f(a_1 \otimes \cdots \otimes a_{n-1} \otimes h a_n) &= (1^{\otimes(n-1)} \otimes h_{(n)})\,f(a_1 \otimes \cdots \otimes a_n)\,(h_{(-1)} \otimes \cdots \otimes h_{(-(n-1))} \otimes 1),
\end{align*}
for all $ h_i, h \in H $, $ a_i \in \mathfrak{g}$.
\end{defn}
 \noindent Define
$$
\mathrm{End}^c_n(\mathfrak{g}) = \left\{ f \in \mathrm{Hom}_\mathbf{k}(\mathfrak{g}^{\otimes n}, H^{\otimes(n-1)} \otimes \mathfrak{g}) \,\Big|\, f \text{ is conformal} \right\}.$$
Similar to Eq. \eqref{eq1}, the action of symmetric group $ S_n $ on $ \mathrm{End}^c_n(\mathfrak{g}) $ is given by
$$
\sigma(f)(a_1 \otimes \cdots \otimes a_n) = (-1)^\sigma (\sigma \otimes 1)\,f(a_{\sigma^{-1}(1)} \otimes \cdots \otimes a_{\sigma^{-1}(n-1)} \otimes a_n),
$$
for $ f \in \mathrm{End}^c_n(\mathfrak{g}) $, $ a_i \in A $, and $ \sigma \in S_{n-1} $. For transpositions $ (n-1,n) $, the action is defined by
$$
(n-1,n)(f)(a_1 \otimes \cdots \otimes a_n) = -g^i_1 h_i^{(-1)} \otimes \cdots \otimes g^i_{n-2} h_i^{(-(n-2))} \otimes h_i^{(n-1)} e_i,
$$
where
$ f(a_1 \otimes \cdots \otimes a_n) = g^i_1 \otimes \cdots \otimes g^i_{n-2} \otimes h_i \otimes e_i.
$ We denote the collection of all such spaces by $\mathrm{End}^c(\mathfrak{g})= \bigoplus_{n \in \mathbb{N}^*}\mathrm{End}^c_n(\mathfrak{g}) $ where $\mathbb{N}^* = \mathbb{N} \setminus \{0\}$.
 \begin{defn}
 For any $ \sigma \in S_p $, let $ (-1)^\sigma $ to denote its signature. A map $f\in M^{p-1}(\mathfrak{g})$ is said to be skew-symmetric if it satisfies \begin{align}\label{eq11}
f(x_1,x_2,\cdots,x_p)
= (-1)^\sigma (\sigma \otimes_H \mathrm{id}) f(x_{\sigma(1)}, x_{\sigma(2)}, \cdots, x_{\sigma(p)}), \quad\text{ for all } x_1, x_2, \cdots, x_p\in \mathfrak{g}. 
\end{align}
 \end{defn}\noindent Let $C^p(\mathfrak{g},\mathfrak{g}) $ denote the subspace of $M^{\otimes p}(\mathfrak{g})$ consisting of all skew-symmetric maps. We define the graded vector space by $ C^*(\mathfrak{g},\mathfrak{g}) = \bigoplus_{p \geq 0} C^p (\mathfrak{g},\mathfrak{g}). $ From \cite{AWY}, we recall the circle product $ f \circledcirc g: C^{p} (\mathfrak{g}, \mathfrak{g}) \otimes C^{q} (\mathfrak{g}, \mathfrak{g}) \to C^{p+q-1} (\mathfrak{g}, \mathfrak{g}) $ by 
\begin{equation}\label{circleproduct}
(f \circledcirc g)(x_1,\cdots,x_{p+q-1}) = \sum_{\sigma \in S_{q,p-1}} (-1)^{\sigma} (\sigma \otimes_H \mathrm{id}) f\big( g(x_{\sigma(1)}, \cdots, x_{\sigma(q)}), x_{\sigma(q+1)}, \cdots, x_{\sigma(p+q-1)} \big),
\end{equation}
 for any $ f \in C^{p} (\mathfrak{g}, \mathfrak{g}) $ and $ g \in C^{q} (\mathfrak{g}, \mathfrak{g}) $ , and $ x_1,\cdots,x_{p+q-1}\in \mathfrak{g} $. Using Eq. \eqref{circleproduct}, we define the Nijenhuis--Richardson bracket (also called the $ NR $-bracket) on the graded space $ C^* (\mathfrak{g}, \mathfrak{g}) $ by
$$
[f,g]_{NR} := f \circledcirc g - (-1)^{(p-1)(q-1)} g \circledcirc f.
$$
The $ NR $-bracket is an element of $ C^{p+q-1}(\mathfrak{g},\mathfrak{g}) $. The graded space $ C^{*}(\mathfrak{g}, \mathfrak{g}) $, equipped with the Nijenhuis--Richardson bracket $ [\cdot,\cdot]_{NR} $, forms a graded Lie algebra, expressed as $ \big(C^{*}(\mathfrak{g}, \mathfrak{g}), [\cdot, \cdot]_{NR}\big) .$ \\
\noindent As is well known, any graded Lie algebra (also called \emph{controlling Lie algebra}) admits the structure of differential graded Lie algebra (dgLa), when a differential operator is introduced. The differential operator on the graded Lie algebra $ (C^{*}(\mathfrak{g}, \mathfrak{g}),[\cdot,\cdot]_{NR}) $ can be defined using a Maurer-Cartan element. A \emph{Maurer--Cartan element}, denoted by $ \mathfrak{M} $, is an element of degree $1$, i.e., $ \mathfrak{M} \in C^2(\mathfrak{g},\mathfrak{g}) $, that satisfies the Maurer-Cartan condition 
$$
[\mathfrak{M}, \mathfrak{M}]_{NR} = 0.
$$
Furthermore, the differential of the dgLa, expressed in terms of the Maurer-Cartan element, is given by
$$
d_\mathfrak{M} := [\mathfrak{M}, -]_{NR}.
$$
\par Let $ \mathfrak{g} $ and $ \mathfrak{h} $ be two left $ H $-modules, with elements denoted by $ x_i $'s, $ u_i$'s  respectively. For any skew-symmetric $ H^{\otimes (k+l)} $-linear map $\kappa: \mathfrak{g}^{\otimes k} \otimes \mathfrak{h}^{\otimes l} \to H^{\otimes(k+l-1)} \otimes_H \mathfrak{g},$ we define its \textbf{lift} $ \hat{\kappa} \in C^{k+l}(\mathfrak{g} \boxplus \mathfrak{h}, \mathfrak{g} \boxplus \mathfrak{h}) $ as
\begin{align*}
\hat{\kappa}\big((x_1, u_1), \dots, (x_{k+l}, u_{k+l})\big) = \left( \sum_{\sigma \in S(k,l)} (-1)^\sigma (\sigma \otimes_H \mathrm{id})\,\kappa(x_{\sigma(1)}, \dots, x_{\sigma(k)}, u_{\sigma(k+1)}, \dots, u_{\sigma(k+l)}),\quad 0 \right).
\end{align*}
Similarly, for any skew-symmetric $ H^{\otimes (k+l)} $-linear map
$\kappa : \mathfrak{g}^{\otimes k} \otimes \mathfrak{h}^{\otimes l} \to H^{\otimes(k+l-1)} \otimes_H \mathfrak{h},$ its lift $ \hat{\kappa} \in C^{k+l}(\mathfrak{g} \boxplus \mathfrak{h}, \mathfrak{g} \boxplus \mathfrak{h}) $ is defined by
\begin{align*}
\hat{\kappa}\big((x_1, u_1), \dots, (x_{k+l}, u_{k+l})\big) = \left( 0,\quad \sum_{\sigma \in S(k,l)} (-1)^\sigma (\sigma \otimes_H \mathrm{id})\,\kappa(x_{\sigma(1)}, \dots, x_{\sigma(k)}, u_{\sigma(k+1)}, \dots, u_{\sigma(k+l)}) \right).
\end{align*}
The maps $ \hat{\kappa} $ are referred to as the \emph{lift} of $ \kappa $. For example, consider the maps $ \alpha: \mathfrak{g} \otimes \mathfrak{h} \to H^{\otimes 2} \otimes_H \mathfrak{g} $ and $ \beta: \mathfrak{g} \otimes \mathfrak{h} \to H^{\otimes 2} \otimes_H \mathfrak{h} $. The corresponding lift maps $\hat{\alpha}$ and $\hat{\beta}$ are given by
\begin{align}
\label{mapa}
\hat{\alpha}\big((x_1, u_1), (x_{2}, u_{2})\big) &= \Big( \alpha(x_{1},u_{2}) - ((12) \otimes_H \mathrm{id})\,\alpha(x_{2}, u_{1}),\quad 0 \Big), \\
\label{mapb}
\hat{\beta}\big((x_1, u_1), (x_{2}, u_{2})\big) &= \Big( 0,\quad \beta(x_{1}, u_{2}) - ((12) \otimes_H \mathrm{id})\,\beta (x_2,u_{1}) \Big).
\end{align}
\begin{defn}\label{def:bidegree}
Let $ \mathfrak{g} $ and $ \mathfrak{h} $ be two left $ H $-modules. A cochain $ f \in C^{p}(\mathfrak{g} \boxplus \mathfrak{h}, \mathfrak{g} \boxplus \mathfrak{h}) $, with $ p = k+l+1 $, is said to have \emph{bidegree} $ k|l $ if 
\begin{enumerate}
 \item[(i)] for any $ Y \in \mathfrak{g}^{\otimes (k+1)} \otimes \mathfrak{h}^{\otimes l} $,
 $$
 f(Y) \in H^{\otimes (k+l+1)} \otimes_H \mathfrak{g};
 $$
 \item[(ii)] for any $ Y \in \mathfrak{g}^{\otimes k} \otimes \mathfrak{h}^{\otimes (l+1)} $,
 $$
 f(Y) \in H^{\otimes (k+l+1)} \otimes_H \mathfrak{h};
 $$
 \item[(iii)] for all $ Y \in (\mathfrak{g} \boxplus \mathfrak{h})^{\otimes p} $ not of the form described in (i) or (ii),
 $$
 f(Y) = 0.
 $$
\end{enumerate}
A cochain $ f \in C^p(\mathfrak{g} \boxplus \mathfrak{h}, \mathfrak{g} \boxplus \mathfrak{h}) $ is called \emph{homogeneous} if it has a bidegree, denoted by $ \|f\| = k|l $.
\end{defn}

\noindent The lifted maps $ \hat{\alpha} $ and $ \hat{\beta} $, defined in Eqs.~\eqref{mapa} and \eqref{mapb}, belong to $ C^2(\mathfrak{g} \boxplus \mathfrak{h},\ \mathfrak{g} \boxplus \mathfrak{h}) $ and have bidegree $ \| \hat{\alpha} \| = \| \hat{\beta} \| = 1|0 $. This naturally gives rise to a homogeneous linear map of bidegree $ 1|0 $, defined as 
\begin{align*}
\hat{\mu}:= \hat{\alpha} + \hat{\beta}.
\end{align*}One verifies that $ \hat{\mu} $ defines a Lie pseudoalgebra bracket corresponding to a semi-direct product of $\mathfrak{g} $ and $ \mathfrak{h}$
\begin{align*}
\hat{\mu} \big((x_1, u_1), (x_2, u_2) \big)
= \Big( \alpha(x_{1},u_{2})-((12)\otimes_H\mathrm{id})\alpha(x_{2}, u_{1}),\quad \beta(x_{1}, u_{2})-((12)\otimes_H\mathrm{id})\beta (x_2,u_{1}) \Big),
\end{align*}
for all $ x_1, x_2 \in \mathfrak{g} $, $ u_1, u_2 \in \mathfrak{h} $. Although $ \hat{\mu} $ is not technically a lift of a single map $ \mu $ (as there is no underlying $ \mu $ of the correct type), we retain this notation for convenience.\\We define $ \mathfrak{g}^{k,l} = \mathfrak{g}^{\otimes k} \otimes \mathfrak{h}^{\otimes l} $. Then $$ (\mathfrak{g} \boxplus \mathfrak{h})^{\otimes n} \cong \bigoplus_{k+l=n} \mathfrak{g}^{k,l} $$ For the space of cochains $ C^p(\mathfrak{g} \boxplus \mathfrak{h}, \mathfrak{g} \boxplus \mathfrak{h})$, we have the isomorphism 
$$
C^p( \mathfrak{g} \boxplus \mathfrak{h}, \mathfrak{g} \boxplus \mathfrak{h}) \cong \left(\bigoplus_{k+l=p} C(\mathfrak{g}^{k,l}, \mathfrak{g})\right) \boxplus \left(\bigoplus_{k+l=p} C (\mathfrak{g}^{k,l}, \mathfrak{h})\right),
$$
where $C(\mathfrak{g}^{k,l}, \mathfrak{g})$ and $C(\mathfrak{g}^{k,l}, \mathfrak{h})$ denote the spaces of cochains mapping $\mathfrak{g}^{k,l}$
 to $\mathfrak{g}$ and $\mathfrak{h}$, respectively. The isomorphism is induced by the lift map, which assigns to each cochain its corresponding component in the direct sum decomposition.
 \\
 The following lemma shows that the $NR$-bracket on $ C^*(\mathfrak{g} \boxplus \mathfrak{h}, \mathfrak{g} \boxplus \mathfrak{h}) $ is compatible with the bigrading.
\begin{lem}
If $ \|f\| = k_f | l_f $ and $ \|g\| = k_g | l_g $, then the $NR$-bracket $[f,g]_{NR}$ has the bidegree $ k_f + k_g | l_f + l_g $.
\end{lem}
\noindent It is straightforward to check that:
\begin{lem}
If $ \|f\| = -1|l$ and $ \|g\| = -1|k $, then $ [f,g]_{NR} = 0 $. Similarly, if $ \|f\| =l|-1 $ and $ \|g\| = k|-1 $, then $ [f,g]_{NR} = 0 $.
\end{lem}\noindent This result reflects the fact that cochains with negative degrees in either component cannot interact nontrivially under the $NR$-bracket.

 We now introduce the main concepts that will be studied in this paper. 
 \begin{defn}Let $(\mathfrak{G},[\cdot *\cdot ]_{\mathfrak{G}})$ be a Lie pseudoalgebra that decomposes as a direct sum of two left $H$-modules, i.e., $\mathfrak{G}=\mathfrak{g}\boxplus \mathfrak{h}$. The triple $(\mathfrak{G}, \mathfrak{g}, \mathfrak{h})$ is called a twilled Lie pseudoalgebra if both $\mathfrak{g}$ and $\mathfrak{h}$ are subpseudoalgebras of $(\mathfrak{G}, [\cdot*\cdot]_\mathfrak{G})$.\end{defn}
\begin{defn}Let $(\mathfrak{G}, [\cdot* \cdot]_{\mathfrak{G}})$ be a Lie pseudoalgebra that decomposes as a direct sum of two left $H$-modules, i.e., $\mathfrak{G} = \mathfrak{g} \boxplus \mathfrak{h}$. The triple $(\mathfrak{G}, \mathfrak{g}, \mathfrak{h})$ is called a quasi-twilled Lie pseudoalgebra if $\mathfrak{h}$ is a Lie subpseudoalgebra of $(\mathfrak{G}, [\cdot*\cdot]_\mathfrak{G})$.
\end{defn}\begin{rem}
In a twilled Lie pseudoalgebra, both $\mathfrak{g}$ and $\mathfrak{h}$ retain their Lie pseudoalgebra structures within $\mathfrak{G}$. This imposes stricter conditions compared to the quasi-twilled Lie pseudoalgebra, that only requires $\mathfrak{h}$ to be a subpseudoalgebra, allowing $\mathfrak{g}$ to interact more flexibly with $\mathfrak{h}$.
\end{rem}
\noindent For any Lie algebra $\mathfrak{b}$, if the dimension of $\mathfrak{b}$ is at least two, then the current pseudoalgebra $Cur(\mathfrak{b}):=H\otimes \mathfrak{b}$ is a quasi-twilled Lie pseudoalgebra. Let $\mathfrak{r}$ denote the radical of $\mathfrak{b}$, then by the Levi decomposition theorem, there exists a simsimple subalgebra $\mathfrak{l}$ such that$$\mathfrak{b}=\mathfrak{l}\oplus \mathfrak{r}.$$
Consequently, the current pseudoalgebra decomposes as $$Cur(\mathfrak{b})=Cur(\mathfrak{l})\oplus Cur(\mathfrak{r}),$$ where $Cur(\mathfrak{l})$ is semsimple and $Cur(\mathfrak{r})$ is solvable. This decomposition ensures that $Cur(\mathfrak{b})$ is a twilled Lie pseudoalgebra. 
Moreover, rank two quasi-twilled Lie pseudoalgebras over $\mathbf{k}[x]$ are twilled Lie pseudoalgebras, as shown in \cite[Theorem 3.4]{Wu1}.
\\Next, we generalize this result to Lie pseudoalgebras over the enveloping algebra $U(\mathfrak{g})$ of a finite-dimensional Lie algebra $\mathfrak{g}$. To facilitate this generalization, we use the notation $((12)\otimes_H id)$ to denote the linear map $((12)\otimes_H id)(f\otimes g\otimes_Hc)=g\otimes f\otimes_H c$ for any $f\otimes g\otimes_H c\in H^{\otimes 2} \otimes_H \mathfrak{g}.$ To prove certain rank two quasi-twilled Lie pseudoalgebras are indeed twilled Lie pseudoalgebra, we establish the following result.
\begin{thm}\label{QTPC}
A quasi-twilled Lie pseudoalgebra corresponds to a tuple $(\mathfrak{g}, \mathfrak{h},\pi, \mu, \rho, \eta,\theta)$, where $\mathfrak{g}$ is a left $H$-module, $\mathfrak{h}$ is a Lie pseudoalgebra with a pseudobracket $\mu: \mathfrak{h} \otimes \mathfrak{h} \to H^{\otimes 2}\otimes_H\mathfrak{h},$ defined by $ (u,v) \mapsto [u* v]_\mathfrak{h},$ and the direct sum of left $H$-modules $\mathfrak{G}=\mathfrak{g} \boxplus \mathfrak{h} $ forms a Lie pseudoalgeba with pseudobracket $\O=\pi+\eta+\rho+\mu+\theta$. The maps $\pi,\eta,\rho,\theta$ are $H^{\otimes 2}$-linear maps and are defined by $$\pi: \mathfrak{g} \otimes \mathfrak{g} \to H^{\otimes 2}\otimes_H\mathfrak{g}, \quad
 \rho: \mathfrak{g} \otimes \mathfrak{h} \to H^{\otimes 2}\otimes_H \mathfrak{h}, \quad
 \eta: \mathfrak{g} \otimes \mathfrak{h} \to H^{\otimes 2}\otimes_H\mathfrak{g}, \quad
 \theta: \mathfrak{g} \otimes \mathfrak{g} \to H^{\otimes 2}\otimes_H\mathfrak{h}.$$ The pseudobracket $\Omega$ on $\mathfrak{g} \boxplus \mathfrak{h}$ is explicitly given by: 
 \begin{align}\label{eq:quasi-mult}
\Omega\big((x, u),(y,v)\big) = (\pi(x\otimes y) + \eta(x\otimes v) -((12)\otimes_Hid) \eta(y\otimes u),\quad\qquad\qquad \\ \nonumber \qquad\qquad [u*v]_\mathfrak{h} + \rho(x\otimes v) -((12)\otimes_Hid)\rho (y\otimes u) + \theta(x\otimes y) ), 
\end{align}for all $x,y,z\in \mathfrak{g}$, $u,v,w \in \mathfrak{h}$. \\
Furthermore, the following compatibility conditions hold for all $x,y,z\in \mathfrak{g}$, $u,v,w \in \mathfrak{h}$, \begin{align*}
\text{(PC1):}\quad& [u*v]_\mathfrak{h}= -((12)\otimes_Hid)[v*u]_\mathfrak{h},\\
\text{(PC2):}\quad &
\pi(x\otimes\pi(y\otimes z)) - \pi(\pi(x\otimes y)\otimes z ) - ((12)\otimes_Hid)\pi(y\otimes \pi(x\otimes z)) \\
& \qquad =((12)\otimes_Hid) \eta(y\otimes \theta(x\otimes z)) -\eta(x\otimes \theta(y\otimes z))-((123)\otimes_Hid) \eta(z \otimes \theta(x\otimes y)),
\\ \text{(PC3):}\quad&
 \rho(x\otimes \theta(y\otimes z) )
 +((123)\otimes_Hid)\rho(z \otimes \theta(x\otimes y) )
-((12)\otimes_Hid)\rho(y\otimes \theta(x\otimes z) )\\
& \qquad=((12)\otimes_Hid) \theta(y\otimes \pi(x\otimes z)) + \theta(\pi(x\otimes y)\otimes z)- \theta(x\otimes \pi(y\otimes z)),\\
\text{(PC4):}\quad& \pi(x \otimes \eta(y \otimes w)) + \eta(x \otimes \rho(y \otimes w)) - \eta(\pi(x \otimes y) \otimes w) \\
&\qquad - ((12) \otimes_H \mathrm{id})\pi(y \otimes \eta(x \otimes w)) -((12) \otimes_H \mathrm{id}) \eta(y \otimes \rho(x \otimes w)) = 0,\\
\text{(PC5):}\quad&\rho(x \otimes \rho(y \otimes w)) + \theta(x \otimes \eta(y \otimes w))- \rho(\pi(x \otimes y) \otimes w))-[ \theta(x \otimes y) * w]_\mathfrak{h} \\
&\qquad - ((12) \otimes_H \mathrm{id})\rho(y \otimes \rho(x \otimes w)) - ((12) \otimes_H \mathrm{id})\theta(y \otimes \eta(x \otimes w)) = 0,\\
\text{(PC6):}\quad& \eta(x \otimes [v * w]_\mathfrak{h}) - \eta(\eta(x \otimes v) \otimes w) + ((23) \otimes_H \mathrm{id})\eta(\eta(x \otimes w) \otimes v) = 0,\\
 \text{(PC7):}\quad&\rho(x \otimes [v * w]_\mathfrak{h}) - \rho(\eta(x \otimes v) \otimes w) - [\rho(x \otimes v) * w]_\mathfrak{h} \\&\qquad+((23) \otimes_H \mathrm{id}) \rho(\eta(x \otimes w) \otimes v) - ((12) \otimes_H \mathrm{id})[v * \rho(x \otimes w)]_\mathfrak{h} = 0,\\
\text{(PC8):}\quad&[u * [v * w]_\mathfrak{h}]_\mathfrak{h} - [[u * v]_\mathfrak{h} * w]_\mathfrak{h} -((12) *_H \mathrm{id}) [v * [u * w]_\mathfrak{h}]_\mathfrak{h} = 0.
\end{align*}
\end{thm}
\begin{proof}
Let $(\mathfrak{G}, \O)$ be a quasi-twilled Lie pseudoalgebra with a decomposition $\mathfrak{G} = \mathfrak{g} \boxplus \mathfrak{h}$ as left $H$-modules. The pseudobracket $\O$ on $\mathfrak{G}$ is defined by Eq. \eqref{eq:quasi-mult}. Since $\pi, \theta, \rho, \eta, \mu={[\cdot*\cdot]}_\mathfrak{h}$ are $H^{\otimes 2}$-linear maps, $\O$ is $H^{\otimes2}$-linear. We now verify skew-symmetry and the Jacobi identity.\\
 \noindent \textbf{Skew-symmetry:} The pseudobracket $\O$ is given explicitly by:
\begin{align*}%\label{eq:quasi-mult1}
 \O((x, 0),(y, 0)) &= \Big(\pi(x\otimes y),\quad \theta(x\otimes y) \Big), \\
 %\label{eq:quasi-mult2}
 \O( (x, 0), (0, v)) &= \Big( \eta(x\otimes v),\quad \rho(x\otimes v) \Big),\\ 
 %\label{eq:quasi-mult3}
 \O( (0, u), (y, 0)) &= \Big( -((12)\otimes_Hid) \eta(y\otimes u), -((12)\otimes_Hid)\rho (y\otimes u )\Big), \\
 %\label{eq:quasi-mult4}
 \O( (0, u),(0, v)) &= \Big( 0,\quad [u*v]_\mathfrak{h} \Big).
\end{align*} From these identities, it follows that
\begin{align*} 
 \Omega ((x, u), (y, v)) &= -((12)\otimes_Hid)\Omega \Big((y, v), (x, u)\Big),
\end{align*} provided each component satisfies the skew-symmetry individually :
\begin{align*}
 \Omega ( (x, 0), (y, 0))&= -((12)\otimes_Hid)\Omega\Big( (y, 0), (x, 0)\Big),\\
 \Omega ( (x, 0), (0, v))&= -((12)\otimes_Hid)\Omega\Big( (0, v), (x, 0)\Big),\\
 \Omega ( (0, u), (y, 0))&= -((12)\otimes_Hid)\Omega \Big( (y, 0), (0, u)\Big),\\
 \Omega( (0, u), (0, v))&=-((12)\otimes_Hid)\Omega\Big( (0, v), (0, u)\Big). \end{align*}Thus, $\Omega$ is skew-symmetric. \\
\noindent \textbf{Jacobi identity:} The Jacobi identity for the pseudobracket $\Omega$ is given by:
\begin{align*}
J\Big((x, u), (y, v), (z, w)\Big) 
&= \Omega\big((x, u),\Omega \big((y, v), (z, w)\big)\big) 
- \Omega \big(\Omega \big((x,u), (y, v)\big), (z, w)\big) \\&\quad- ((12)\otimes_Hid)\Omega\big((y, v), \Omega\big((x, u), (z, w)\big)\big) =0,
\end{align*}
for all $ x, y, z \in \mathfrak{g} $ and $ u, v, w \in \mathfrak{h} $. Where $J$ is called the Jacobiator. To better understand the structure of the $\Omega$ on $\mathfrak{G}$, we begin by analyzing the Jacobi identity under specific input configurations. \\
\noindent
 \textbf{Case 1:} Let us consider the case where all inputs lie in $ \mathfrak{g} $. Then 
$
\Omega\big((x, 0) , (y, 0)\big) = \big( \pi(x \otimes y), \theta(x \otimes y) \big),
$
The Jacobi identity becomes:
\begin{align*}
 &J\Big((x, 0) , (y, 0) , (z, 0)\Big)\\ &= \Big(\pi(x \otimes \pi(y \otimes z)) + \eta(x \otimes \theta(y \otimes z)), \quad \rho(x \otimes \theta(y \otimes z)) + \theta(x \otimes \pi(y \otimes z))\Big)\\ 
& \quad- \Big((\pi(\pi(x \otimes y) \otimes z)) +((123)\otimes_Hid) \eta(z \otimes \theta(x \otimes y)), \quad ((123)\otimes_Hid)\rho(z \otimes \theta(x \otimes y)) + \theta(\pi(x \otimes y) \otimes z) \Big) \\
& \quad- ((12)\otimes_Hid)\Big(\pi(y \otimes \pi(x \otimes z)) 
+ \eta(y \otimes \theta(x \otimes z)), \quad \rho\big(y \otimes \theta(x \otimes z)) + \theta(y \otimes \pi(x \otimes z)) \Big) = 0.
\end{align*}
Equating components leads to:
\begin{align*}
\text{(PC2)} &\quad \pi(x \otimes \pi(y \otimes z)) - \pi(\pi(x \otimes y) \otimes z) - ((12)\otimes_Hid)\pi(y \otimes \pi(x \otimes z)) \\&\qquad\qquad=((12)\otimes_Hid) \eta(y \otimes \theta(x \otimes z)) - \eta(x \otimes \theta(y \otimes z)) -((123)\otimes_Hid) \eta(z \otimes \theta(x \otimes y)), \\
\text{(PC3)} &\quad\rho(x \otimes \theta(y \otimes z))+ ((123)\otimes_Hid)\rho(z \otimes \theta(x \otimes y)) - ((12)\otimes_Hid)\rho(y \otimes \theta(x \otimes z)) \\ &\qquad\qquad=((12)\otimes_Hid) \theta(y \otimes \pi(x \otimes z)) + \theta(\pi(x \otimes y) \otimes z) - \theta(x \otimes \pi(y \otimes z)).
\end{align*}
\noindent \textbf{Case 2:} Let $x,y\in \mathfrak{g}$ and $w\in \mathfrak{h}$ . Expanding the Jacobi identity yields:
%$ J\big((x, 0) \otimes (y, 0) \otimes (0, w)\big). $Now let two elements come from $ \mathfrak{g} $ and one from $ \mathfrak{h} $. We compute:$ \Omega\big((y, 0) \otimes (0, w)\big) = \big(\eta(y \otimes w),\quad \rho(y \otimes w)\big),$ and $\Omega((x, 0) \otimes (\eta(y \otimes w),\rho(y \otimes w))) = ( \pi(x \otimes \eta(y \otimes w)) + \eta(x \otimes \rho(y \otimes w)),\quad \rho(x \otimes \rho(y \otimes w)) + \theta(x \otimes \eta(y \otimes w)) ).$Then
\begin{align*}
& J\Big((x, 0) , (y, 0), (0, w)\Big) \\
&= \Omega\big((x, 0) , \Omega\big((y, 0) , (0, w)\big)\big)
- \Omega\big(\Omega\big((x, 0) , (y, 0)\big) , (0, w)\big) \\
&\quad - ((12) \otimes_H \mathrm{id}) \Omega\big((y, 0), \Omega\big((x, 0) , (0, w)\big)\big)
\\
&= \Big(\pi(x \otimes \eta(y \otimes w)) + \eta(x \otimes \rho(y \otimes w)),\quad \rho(x \otimes \rho(y \otimes w)) + \theta(x \otimes \eta(y \otimes w)) \Big) \\
&\quad - \Big(\eta(\pi(x \otimes y) \otimes w),\quad \rho(\pi(x \otimes y) \otimes w))+ [ \theta(x \otimes y) * w]_\mathfrak{h} \Big) \\
&\quad -((12) \otimes_H \mathrm{id}) \Big(\pi(y \otimes \eta(x \otimes w)) + \eta(y \otimes \rho(x \otimes w)),\quad\rho(y \otimes \rho(x \otimes w)) + \theta(y \otimes \eta(x \otimes w)) \Big)=0.
\end{align*}
This leads to the compatibility conditions:
\begin{align*}
 \text{(PC4)}\quad & \pi(x \otimes \eta(y \otimes w)) + \eta(x \otimes \rho(y \otimes w))- \eta(\pi(x \otimes y) \otimes w) \\
&\qquad\qquad - ((12) \otimes_H \mathrm{id})\pi(y \otimes \eta(x \otimes w)) -((12) \otimes_H \mathrm{id}) \eta(y \otimes \rho(x \otimes w)) = 0,\\
 \text{(PC5)}
 \quad & \rho(x \otimes \rho(y \otimes w)) + \theta(x \otimes \eta(y \otimes w)) - \rho(\pi(x \otimes y) \otimes w))-[ \theta(x \otimes y) * w]_\mathfrak{h} \\
&\qquad\qquad - ((12) \otimes_H \mathrm{id})\rho(y \otimes \rho(x \otimes w)) - ((12) \otimes_H \mathrm{id})\theta(y \otimes \eta(x \otimes w)) = 0.
\end{align*}
%This ensures the compatibleity of $ \rho $ and $ \eta $ with the pseudobrackets $ \pi $ and $ \theta $.\\
\noindent\textbf{Case 3:} %$ J\big((x, 0) \otimes (0, v) \otimes (0, w)\big). $
 Let $ x \in \mathfrak{g} $, $ v, w \in \mathfrak{h} $. Then
%$\Omega((0, v) \otimes (0, w)) = ( 0,\quad [v * w]_\mathfrak{h} ),\Omega((x, 0) \otimes (0, v)) = \big( \eta(x \otimes v),\quad \rho(x \otimes v) ).$ Thus
\begin{align*}
& J\Big((x, 0), (0, v) , (0, w)\Big) \\
&= \Omega\big((x, 0) , \Omega\big((0,v) , (0, w)\big)\big)
- \Omega\big(\Omega\big((x, 0) , (0, v)\big) , (0, w)\big) - ((12) \otimes_H \mathrm{id}) \Omega\big((0, v) , \Omega\big((x, 0) , (0, w)\big)\big).
\\
&= \Omega\big((x, 0) , (0, [v * w]_\mathfrak{h})\big)
- \Omega\big((\eta(x \otimes v), \rho(x \otimes v)) , (0, w)\big) - ((12) \otimes_H \mathrm{id}) \Omega\big((0, v), (\eta(x \otimes w), \rho(x \otimes w))\big) \\
&= \Big(\eta(x \otimes [v * w]_\mathfrak{h}), \rho(x \otimes [v * w]_\mathfrak{h})\Big)- \Big(\eta(\eta(x \otimes v) \otimes w), \rho(\eta(x \otimes v) \otimes w) + [\rho(x \otimes v) * w]_\mathfrak{h} \Big) \\
&\quad - ((12) \otimes_H \mathrm{id})\Big(-((23) \otimes_H \mathrm{id})\eta(\eta(x \otimes w) \otimes v), -((23) \otimes_H \mathrm{id})\rho(\eta(x \otimes w) \otimes v) + [v * \rho(x \otimes w)]_\mathfrak{h} \Big).
\end{align*} This gives:
\begin{align*}
 \text{(PC6)}&\quad \eta(x \otimes [v * w]_\mathfrak{h}) - \eta(\eta(x \otimes v) \otimes w) + ((123) \otimes_H \mathrm{id})\eta(\eta(x \otimes w) \otimes v) = 0,\\
 \text{(PC7)}&\quad \rho(x \otimes [v * w]_\mathfrak{h}) - \rho(\eta(x \otimes v) \otimes w) - [\rho(x \otimes v) * w]_\mathfrak{h} +((123) \otimes_H \mathrm{id}) \rho(\eta(x \otimes w) \otimes v) \\& \qquad \qquad\qquad\qquad\qquad \qquad\qquad\qquad\qquad\qquad\qquad\qquad- ((12) \otimes_H \mathrm{id})[v * \rho(x \otimes w)]_\mathfrak{h} = 0.
\end{align*}
%These ensure that $ \eta $ defines a representation of $ \mathfrak{h} $ on $ \mathfrak{g} $, and $ \rho $ satisfies a twisted Jacobi identity in the pseudoalgebraic setting.
\noindent \textbf{Case 4:}% $ J\big((0, u) \otimes (0, v) \otimes (0, w)\big) .$
 Now consider all inputs from $ \mathfrak{h} $. %Using$\Omega\big((0, u) \otimes (0, v)\big) = \big( 0,\quad [u *v]_\mathfrak{h} \big),$
We compute
\begin{align*}
& J\Big((0, u) ,(0, v) , (0, w)\Big) \\
&= \Omega\big((0, u) , \Omega\big((0,v) , (0, w)\big)\big)
- \Omega\big(\Omega\big((0,u), (0, v)\big) , (0, w)\big) \\
&\quad - ((12) \otimes_H \mathrm{id}) \Omega\big((0, v) ,\Omega\big((0,u) , (0, w)\big)\big) \\
&= \Big(0, [u *[v *w]_\mathfrak{h}]_\mathfrak{h} \Big)
- \Big(0, [[u *v]_\mathfrak{h} * w]_\mathfrak{h} \Big)
- ((12) \otimes_H \mathrm{id}) \Big(0, [v * [u * w]_\mathfrak{h}]_\mathfrak{h} \Big).
\end{align*}This yields the Jacobi identity in $ \mathfrak{h}$ 
\begin{align*}
 \text{(PC8)} \quad [u * [v * w]_\mathfrak{h}]_\mathfrak{h} - [[u * v]_\mathfrak{h} * w]_\mathfrak{h} -((12) *_H \mathrm{id}) [v * [u * w]_\mathfrak{h}]_\mathfrak{h} = 0,
\end{align*} confirming that $ \mathfrak{h} $ is a Lie pseudoalgebra over $ H $.\\
The compatibility conditions PC1-PC8 are satisfied. In particular, PC1 and PC8 ensure that $ \mathfrak{h} $ is a Lie pseudoalgebra over $H$ , and $ \mathfrak{G} $ is a quasi-twilled Lie pseudoalgebra. This completes the proof.
\end{proof}
\begin{rem}\label{rem:PC-interpretations-pseudo}
The identities provided in Theorem \ref{QTPC} admit the following interpretations within the framework of pseudoalgebras over a cocommutative Hopf algebra $ H $.
\begin{itemize}
 \item [\textbf{PC1}] \textbf{and PC8} ensure that $ \mathfrak{h} $ is a Lie pseudoalgebra over $ H $, and is a sub-pseudoalgebra of $ \mathfrak{G} $.
 \item [\textbf{PC2}] shows that $ (\mathfrak{g}, \pi) $ is not a Lie pseudoalgebra on its own, due to the presence of interaction terms $ \eta $ and $ \theta $.
 \item [\textbf{PC3}] encodes a twisted cocycle condition, which governs the deviation of $ \pi $ from being a Lie pseudoalgebra bracket. 
 \item [\textbf{PC5}]describes a twisted module condition, detailing how the action $ \rho $ of $ \mathfrak{g} $ on $ \mathfrak{h} $ fails to satisfy the usual Jacobi identity in the category $ \mathcal{M}^*(H) $. This failure arises due to the interaction term $ \theta $ and the pseudobracket on $ \mathfrak{h} $.
 \item [\textbf{PC6}] establishes that $ \eta $ defines a representation of $ \mathfrak{h} $ on $ \mathfrak{g} $ in the pseudotensor category $ \mathcal{M}^*(H) $,
 \item [\textbf{PC7}]reflects a twisted Jacobi identity, highlighting the deviation of $ \rho $ from being a strict Lie pseudoalgebra representation. 
\end{itemize}
\end{rem}
\begin{lem}
 Let $(\mathfrak{G}, \mathfrak{g}, \mathfrak{h})$ be a quasi-twilled Lie pseudoalgebra, with pseudobracket $\Omega \in C^2(\mathfrak{G},\mathfrak{G})$ as defined in Theorem \ref{QTPC}.
\end{lem} \noindent Further, using $\Omega $, the Nijenhuis-Richardson bracket $NR$-bracket $\big[\Omega, \Omega ]_{NR}$ expands as:
\begin{align*}
 \big[\Omega, \Omega \big]_{NR}
 =& \big[\pi+\rho+\mu +\eta+\theta,\ \pi+\rho+\mu +\eta+\theta\big]_{NR}
\\=& [\pi, \pi]_{NR} + [\pi, \rho]_{NR} + [\pi, \mu]_{NR} + [\pi, \eta]_{NR} + [\pi, \theta]_{NR} 
\\&+[\rho, \pi]_{NR}
+ [\rho, \rho]_{NR} + [\rho, \mu]_{NR} + [\rho, \eta]_{NR} + [\rho, \theta]_{NR} \\& +[\mu, \pi]_{NR}
+ [\mu, \rho]_{NR}
+ [\mu, \mu]_{NR} + [\mu, \eta]_{NR} + [\mu, \theta]_{NR} 
\\&
+ [\eta, \pi]_{NR}
+ [\eta, \rho]_{NR}
+ [\eta, \mu]_{NR}+[\eta, \eta]_{NR} + [\eta, \theta]_{NR} \\&+ [\theta, \pi]_{NR}
+ [\theta, \rho]_{NR}
+ [\theta, \mu]_{NR}+[\theta, \eta]_{NR}+ [\theta, \theta]_{NR}\\
=& [\pi, \pi]_{NR} + 2[\pi, \rho]_{NR} + 2[\pi, \mu]_{NR} + 2[\pi, \eta]_{NR} + 2[\pi, \theta]_{NR} 
\\&
+ [\rho, \rho]_{NR} + 2[\rho, \mu]_{NR} + 2[\rho, \eta]_{NR} + 2[\rho, \theta]_{NR} \\& 
+ [\mu, \mu]_{NR} + 2[\mu, \eta]_{NR} + 2[\mu, \theta]_{NR}. 
\end{align*}
Indeed, $\Omega $ is called a Maurer-Cartan element if $[\Omega, \Omega]_{ NR} = 0$, which holds if and only if the following conditions are satisfied:
\begin{align*}\begin{cases}
\text{(PC2)} \implies &[\pi, \pi]_{NR} + 2\eta \circledcirc \theta = 0,\\ 
\text{(PC3)}\implies &[\rho, \theta]_{NR} + [\pi, \theta]_{NR} = 0, \\
\text{(PC4)} \implies &{[\pi, \eta]}_{NR} + \eta \circledcirc \rho = 0,\\ \text{(PC5)} \implies &\frac{1}{2}[\rho, \rho]_{NR}+ \theta \circledcirc \eta+ [\mu, \theta]_{NR}+ [\pi, \rho]_{NR}= 0,\\
\text{(PC6)}\implies &2{[\mu, \eta]}_{NR} + [\eta, \eta]_{NR} = 0, \\ \text{(PC7)}\implies &{[\rho, \mu]}_{NR} + \rho \circledcirc \eta = 0,\\ \text{(PC8)}\implies &[\mu, \mu]_{NR} = 0.\\
\end{cases}\end{align*} 
Thus $\Omega $ is indeed a Maurer-Cartan element. This characterization will be used later to define the differential of a graded Lie algebra via $d_\Omega :=[\Omega,-]_{NR}$. 

In the following, we classify quasi-twilled Lie $H$-pseudoalgebras that are free $H$-modules of rank $2$, where $H= U(\mathfrak{b})$ denotes the enveloping algebra of a finite-dimensional Lie algebra $\mathfrak{b}$ over a field $ \mathbf{k}$ of characteristic zero. \\
Let $ \mathfrak{G}$ be a quasi-twilled Lie $ H $-pseudoalgebra such that $ \mathfrak{G}=Hu\oplus Hx $ has rank $2$, with basis $\{u,x\}$, and suppose that $Hx$ is a subpseudoalgebra. Then by \cite[Corollary 3.1]{Wu},
the pseudobracket on $x$ satisfies $$\text{either }[x*x]=0, \text{ or }[x*x]=\sum\limits_{i=1}^r(a_i\otimes a_{r+i}-a_{r+i}\otimes a_i)+s\otimes 1-1\otimes s$$ for some linearly independent set $\{a_i|1\leq i\leq 2r\}\subseteq \mathfrak{b}$ and some $s\in\mathfrak{b}$.
\\ If $[x*x]\neq 0$, then by \cite[Theorem 8.1]{BDK}, $Hx$ is either a pseudoalgebra of type $H$ or a pseudoalgebra of type $K$, depending on whether $s$ lies in the space spanned by $\{a_i|1\leq i\leq 2r\}$.

Let $\{a_i|1\leq i\leq n\}$ be a fixed basis of $\mathfrak{b}$. Then $H$ admits a Poincar\'e-Birkhoff-Witt basis consisting of monomials: $$a^{(K)}:=\{\frac{a_1^{k_1}}{k_1!}\frac{a_2^{k_2}}{k_2!}\cdots \frac{a_n^{k_n}}{k_n!}|k_i\in\mathbb{N}\},$$ where $K=(k_1,k_2,\cdots,k_n)\in \mathbb{N}^n$. We denote the degree of $a^{(K)}$ by $|K|:=k_1+k_2+\cdots+k_n.$\\
 The coproduct on $H$ acts on these basis elements as: $$\Delta(a^{(I)})=\sum\limits_{L}a^{(L)}\otimes a^{(I-L)},$$
 where the sum runs over all $ L = (l_1, \dots, l_n) \in \mathbb{N}^n $ such that $ l_i \leq i_i $ for all $ i $, and $ I - L = (i_1 - l_1, \dots, i_n - l_n) $.
 
\begin{lem}\label{alem12}Let $\alpha=\sum\limits_{i=1}^r(a_i\otimes a_{r+i}-a_{i+r}\otimes a_i)+s\otimes 1-1\otimes s,$ for some $a_j,s\in\mathfrak{g}$, where $a_1,a_2,\cdots,a_{2r}$ are linear independent. Assume that $C=\sum_{I,J}C_{I,J}a^{(I)}\otimes a^{(J)}$ for some $C_{I,J}\in\mathbf {k}$.
If \begin{eqnarray} \label{alem}(id\otimes \alpha\Delta)C-(C\Delta\otimes id)C
 + ((23) \otimes_H \mathrm{id}) (C\Delta\otimes id)C = 0,\end{eqnarray}
then $C$ must be of the form $C=\sum\limits_{i=1}^na_i\otimes v_i+s \otimes 1+1\otimes t_2+c_0\otimes 1$, where $a_i\ (1\leq i\leq n)$ is a basis of $\mathfrak{b}$, $v_i,t_2\in \mathfrak{b}$ and $c_0\in\mathbf{k}$.
Moreover, the elements $v_i,t_2, c_0$ satisfy the following conditions:
 \begin{enumerate}
 \item[1.] $[s,t_2]=[\sum\limits_{i=1}^n[a_i\otimes v_i,\Delta(s)]=0$, 
 \item[2.] $\sum\limits_{j=1}^r(a_j\otimes a_{j+r}t_2-a_{j+r}\otimes a_jt_2+a_jt_2\otimes a_{j+r}-a_{j+r}t_2\otimes a_j)+(s\otimes t_2-t_2\otimes s)+\sum\limits_{j=1}^rc_0(a_j\otimes a_{j+r}-a_{j+r}\otimes a_j)=\sum\limits_{i=1}^n(t_2a_i\otimes v_i-v_i\otimes t_2a_i)+\sum\limits_{i=1}^nc_0(a_i\otimes v_i-v_i\otimes a_i),$ 
 \item[3.] $\sum\limits_{i,=1}^{i=r, j=n}a_i\otimes (a_j\otimes a_{j+r}v_i-a_{j+r}\otimes a_jv_i+a_jv_i\otimes a_{j+r}-a_{j+r}v_i\otimes a_j)+\sum\limits_{i=1}^na_i\otimes (s\otimes v_i-v_i\otimes s)+\sum\limits_{j=1}^rs\otimes(a_j\otimes a_{j+r}-a_{j+r}\otimes a_j)=\sum\limits_{1\leq i<j\leq n}[a_i,a_j]\otimes(v_i\otimes v_j-v_j\otimes v_i)+\sum\limits_{i,j=1}^na_i\otimes(v_ia_j\otimes v_j-v_j\otimes v_ia_j)+\sum\limits_{i=1}^ns\otimes(a_i\otimes v_i-v_i\otimes a_i).$
 \end{enumerate}
 In particular, $\alpha=s\otimes 1-1\otimes s$ if and only if $C=s\otimes 1-\lambda \otimes s+c_0\otimes 1$ for some $\lambda, c_0\in\mathbf{k}$.
\end{lem}
\begin{proof}Assume that $C\neq 0$ and $\alpha\neq 0$. Define $m= max\{|I|| C_{I,J}\neq 0\}$ and $n=max\{|J|| C_{I,J}\neq 0\}$. From Eq. \eqref{alem}, comparing degrees on both sides yields $2m-1=m$ and $2n+1=n+2$, which implies $m=n=1$. Consequently, we may assume that $$C=\sum\limits_{i=1}^n a_i\otimes v_i +t_1\otimes 1+1\otimes t_2+c_0\otimes 1,$$ where $\{a_i\}$ is a basis of $\mathfrak{b}$, $v_i$ ($1\leq i\leq n)$, $t_1,t_2\in\mathfrak{b}$ and $c_0\in\mathbf{k}$. Substituting this ansatz into Eq. \eqref{alem}, we obtain 
\begin{enumerate}
 \item[1.] The degree-zero component gives $st_2-t_2t_1=c_0(t_1-s)$,
 \item[2.] The degree-one components yield $\sum\limits_{j=1}^r(a_j\otimes a_{j+r}t_2-a_{j+r}\otimes a_jt_2+a_jt_2\otimes a_{j+r}-a_{j+r}t_2\otimes a_j)+(s\otimes t_2-t_2\otimes s)+\sum\limits_{j=1}^rc_0(a_j\otimes a_{j+r}-a_{j+r}\otimes a_j)=\sum\limits_{i=1}^n(t_2a_i\otimes v_i-v_i\otimes t_2a_i)+\sum\limits_{i=1}^nc_0(a_i\otimes v_i-v_i\otimes a_i),$
 \item[3.]The mixed-degree terms give $\sum\limits_{i=1}^na_i\otimes(sv_i -v_i t_1)= \sum \limits_{i=1}^n [a_i,t_1] \otimes v_i+ t_1\otimes (t_1-s)$,
 \item[4.]The higher-degree terms (degree two) give \begin{align*}
 & \sum\limits_{i,=1}^{i=r, j=n}a_i\otimes (a_j\otimes a_{j+r}v_i-a_{j+r}\otimes a_jv_i+a_jv_i\otimes a_{j+r}-a_{j+r}v_i\otimes a_j)+\sum\limits_{i=1}^na_i\otimes (s\otimes v_i-v_i\otimes s)\\&\qquad \qquad \qquad \qquad +\sum\limits_{j=1}^rt_1\otimes(a_j\otimes a_{j+r}-a_{j+r}\otimes a_j)\\&=\sum\limits_{1\leq i<j\leq n}[a_i,a_j]\otimes(v_i\otimes v_j-v_j\otimes v_i)+\sum\limits_{i,j=1}^na_i\otimes(v_ia_j\otimes v_j-v_j\otimes v_ia_j)+\sum\limits_{i=1}^nt_1\otimes(a_i\otimes v_i-v_i\otimes a_i).
 \end{align*}
\end{enumerate}
 \noindent From identity (3) \begin{align*}
 \sum\limits_{i=1}^na_i\otimes [s,v_i]+\sum\limits_{i=1}^na_i\otimes v_i(s-t_1)=\sum\limits_{i=1}^n[a_i,t_1]\otimes v_i+t_1\otimes (t_1-s),
\end{align*} we deduce that either $t_1=s$ or $\sum\limits_{i=1}^na_i\otimes v_i= 0$. \\
If $\sum\limits_{i=1}^na_i\otimes v_i=0$, then identity (3) reduces to
 \begin{align*}
0=\sum\limits_{i=1}^n[a_i,t_1]\otimes v_i+t_1\otimes (t_1-s),
\end{align*}
providing $t_1=s$. If $t_1=s$, we get $[s,t_2]=0$, and $[\sum\limits_{i=1}^na_i\otimes v_i,\Delta(s)]=0$ from identity (1) and identity (3) respectively. If $\alpha=s\otimes 1-1\otimes s$, then $\sum\limits_{i=1}^na_i\otimes v_i=0$ by (2). Conversely, if $\sum\limits_{i=1}^na_i\otimes v_i=0$, then $\sum\limits_{i=1}^r(a_i\otimes a_{i+r}-a_{i+r}\otimes a_i)=0$ and $t_2=\lambda s$ for some $\lambda\in\mathbf{k}$. Therefore, we conclude $\alpha=s\otimes 1-1\otimes s$ and $C= s\otimes 1- \lambda \otimes s+ c_0\otimes 1$. \end{proof}
 
 From Lemma \ref{alem12}, we get that if $\alpha=s\otimes 1-1\otimes s$, then $Hx=Cur^H_{\mathbf{k}[s]}(Vir)=H\otimes_{\mathbf{k}[s]}Vir$, where $Vir$ is the Virosoro Lie conformal algebra. Thus we obtain twilled Lie pseudoalgebras $Cur^H_{\mathbf{k}[s]}L$, where $L$ is the rank two twilled Lie pseudoalgebra over $\mathbf{k}[s]$ classified in \cite[Theorem 3.4]{Wu1}. Next, we describe the rank two quasi-twilled Lie pseudoalgebras with an abelian subpseudoalgebra of rank one.
 
 \begin{thm}Let $H= U(\mathfrak{b})$ be the enveloping algebra of a finite-dimensional Lie algebra $\mathfrak{b}$ over the field $\mathbf{k}$ of characteristic zero, and let $ \mathfrak{G}$ be a quasi-twilled Lie $ H $-pseudoalgebra of rank $2$,i.e., $ \mathfrak{G}=Hu\oplus Hx $, where $\{u,x\}$ is an $H$-basis and $Hx$ is an abelian subpseudoalgebra. Then $\mathfrak{G}$ is isomorphic to one of the following types.
\begin{enumerate}
 \item [i.] Direct sum of two rank one Lie pseudoalgebras.
 \item [ii.] $Hx$ is an abelian Lie pseudoalgebra and the only nontrivial pseudobracket is $$\eta(u\otimes x)=(\sum\limits_{J\in \mathbb{N}}C_J\otimes a^{(J)})\otimes_Hu,\text{ for some } C_j\in \mathbf{k}.$$ 
 \item [iii.] $Hx$ is an abelian Lie pseudoalgebra, but $Hu$ is equipped with a simple Lie pseudoalgebra structure. The pseudobracket on $\mathfrak{G}$ is defined by 
 \begin{align*}
\pi(u\otimes u)&= \alpha\otimes_Hu,& \qquad \rho (u\otimes x)&=(\sum\limits_{J\in \mathbb{N}}bC_J\otimes a^{(J)})\otimes_Hx,& \\ 
\eta(u\otimes x)&=(\sum\limits_{J\in \mathbb{N}}C_J\otimes a^{(J)})\otimes_Hu,& \qquad \theta(u\otimes u)&=0.& 
 \end{align*} such that $\alpha=b(C-(12)C)$, where $\alpha=\sum\limits_{i=1}^r(a_i\otimes a_{i+r}-a_{i+r}\otimes a_i)+s\otimes 1-1\otimes s$ for some linearly independent $a_i, s\in\mathfrak{b}$, $b, C_j\in\mathbf{k}$. 
\end{enumerate}
\end{thm}
\begin{proof}
Assume that $\pi(u\otimes u)=A\otimes_Hu=\sum\limits_{I,J}A_{I,J}a^{(I)}\otimes a^{(J)}\otimes_Hu$, $\rho(u\otimes x)=B\otimes_Hx=\sum\limits_{I,J}B_{I,J}a^{(I)}\otimes a^{(J)}\otimes_Hx$, 
$\eta(u\otimes x)=C\otimes_Hu=\sum\limits_{i,J}C_{I,J}a^{(I)}\otimes a^{(J)}\otimes_Hu$ and $\theta(u\otimes u)=D\otimes x=\sum\limits_{I,J}D_{I,J}a^{(I)}\otimes a^{(J)}\otimes_Hx$ for some $A_{I,J}, B_{I,J},C_{I,J},D_{I,J}\in \mathbf{k}.$ To classify the rank two quasi-twilled Lie pseudoalgebras, we must determine all possible coefficients $A_{I,J},B_{I,J},C_{I,J},D_{I,J}\in\mathbf{k}$ such that the compatibility conditions $(PC2)-(PC7)$ hold.\\
\noindent \textbf{Case 1:}
Suppose that $B=0,C=0$ and $D=0$. Then $\mathfrak{G}=Hu\oplus Hx$ is a direct sum of two pseudoalgebras of rank one. This corresponds to Type (i).

 Next, we assume that either $B\neq 0$, $C\neq 0$, or $D\neq 0$.\\ \noindent
\noindent \textbf{Case2:} Assume that $C\neq 0$. Since $Hx$ is an abelian Lie pseudoalgebra of rank one, $[x*x]=0$. Condition (PC6) implies that 
\begin{eqnarray}\begin{array}{l}\label{eq13}\sum\limits_{K\neq I'}C_{I,J}C_{I'J'}a^{(I)}a^{(K)}\otimes (a^{(J)}a^{(I'-K)}\otimes a^{(J')}-a^{(J')}\otimes a^{(J)}a^{(I'-K)})=0.\end{array}\end{eqnarray} 
Let $n=max\{|J||C_{I,J}\neq 0\}$ and $m=max\{|I||C_{I,J}\neq 0\}$.
By comparing the third factors of tensor products in Eq. \eqref{eq13}, we find that $n+m=n$, which implies $m=0$. 
Thus, $C$ has no dependence on the first tensor factor and we say $C=\sum\limits_{J}C_J\otimes a^{(J)}$, where $C_J=C_{0,J}$. Now, from condition (PC7), we obtain $$\sum B_{I,J}C_{J'}a^{(K)}\otimes (a^{(J')}a^{(I-K)}\otimes a^{(J)}-a^{(J)}\otimes a^{(J')}a^{(I-K)})=0.$$
Consequently, we obtain $$\sum B_{I,J}C_{J'} (a^{(J')}a^{(I)}\otimes a^{(J)}-a^{(J)}\otimes a^{(J')}a^{(I)})=0.$$ 
Since $C\neq 0$, we get $B=t(\sum\limits_IB_{I,0}a^{(I)}\otimes 1)C$ for some $t\in\mathbf{k}$. In addition, $max\{|I||B_{I,J}\neq 0\}=0$. So $B=bC$ for some $b\in\mathbf{k}$.\\
 If $b=0$, then $B=0$. From Condition (PC4), w e get $$\sum A_{I,J}C_{J'}(a^{(I)}\otimes a^{(K)}-a^{(K)}\otimes a^{(I)})\otimes a^{(J')}a^{(J-K)}=\sum A_{I,J}C_{J'}a^{(I)}\otimes a^{(J)}\otimes a^{(J')}.$$ 
 Since $C\neq 0$, $$\sum A_{I,J}(a^{(I)}\otimes a^{(K)}-a^{(K)}\otimes a^{(I)})\otimes a^{(J-K)}=\sum A_{I,J}a^{(I)}\otimes a^{(J)}\otimes 1.$$ Hence $A=0$. then from this and (PC2), we derive that 
$$\sum D_{I,J}C_{J'}(1\otimes a^{(I)}a^{(K)}-a^{(I)}a^{(K)})\otimes a^{(J)}a^{(J'-K)}=-\sum D_{I,J}C_{J'}a^{(J)}a^{(J'-K)}\otimes 1\otimes a^{(I)}a^{(K)}.$$ Since the first two tensor factors are antisymmetric, $$\sum D_{I,J}C_{J'}a^{(I)}a^{(J'-K)}\otimes a^{(I)}a^{(J'-K)}=D(\sum C_{J'}\Delta(a^{(J)})=0.$$ This implies $D=0$. Thus, $\mathfrak{G}$ is isomorphic to Type (ii).\\
\noindent \textbf{Case 3:}
Now we assume that $\alpha=0$, $B\neq 0$ and $C\neq 0$. From (PC5) we have \begin{align*}
 \sum \limits_{K\neq 0}b^2C_JC_{J'}(1\otimes a^{(K)}&-a^{(K)}\otimes 1)\otimes a^{(J')}a^{(J-K)}+\sum D_{I,J}C_{J'}(a^{(I)}\otimes a^{(K)}-a^{(K)}\otimes a^{(I)})\otimes a^{(J')}a^{(J-K)}\\&=\sum bA_{I,J}C_{J'}a^{(I)}\otimes a^{(J)}\otimes a^{(J')}.
\end{align*}
Thus \begin{align*}
 \sum \limits_{K\neq 0}b^2C_J(1\otimes a^{(K)}&-a^{(K)}\otimes 1)\otimes a^{(J-K)}+\sum D_{I,J}(a^{(I)}\otimes a^{(K)}-a^{(K)}\otimes a^{(I)})\otimes a^{(J-K)}\\&=\sum bA_{I,J}a^{(I)}\otimes a^{(J)}\otimes 1.
\end{align*} Applying the functor $id\otimes id\otimes \varepsilon$ to the previous equation, we get that $$\sum b^2C_J(1\otimes a^{(J)}-a^{(J)}\otimes 1)+\sum D_{I,J}(a^{(I)}\otimes a^{(J)}-a^{(J)}\otimes a^{(I)})=\sum bA_{I,J}a^{(I)}\otimes a^{(J)}.$$ Hence $A=B-(12)B+\frac1b(D-(12)D)$.

\noindent From (PC3) we get that \begin{align*}
 &\sum bD_{I,J}C_{J'}(1\otimes a^{(I)}a^{(K)}
-a^{(I)}a^{(K)}\otimes 1)\otimes a^{(J)}a^{(J'-K)}+\sum b D_{I,J}C_{J'}a^{(J)}a^{(K)}\otimes 1\otimes a^{(I)}a^{(J'-K)}\\=&\sum A_{I,J}D_{I',J'}(a^{(I)}a^{(K)}\otimes a^{(I')}-a^{(I')}\otimes a^{(I)}a^{(K)})\otimes a^{(J)}a^{(J'-K)}\\&+\sum A_{I,J}D_{I',J'}a^{(I)}a^{(K)}\otimes a^{(J)}a^{(I'-K)}\otimes a^{(J')}.
\end{align*}Thus $$\sum D_{I,J}C_{J'}a^{(J)}\otimes 1\otimes a^{(I)}a^{(J')}=-\sum D_{I,J}C_{J'}\otimes a^{(J)}\otimes a^{(I)}a^{(J')}.$$
Applying the functor $id\otimes id\otimes \varepsilon$, we get $$\sum D_{I,J}a^{(J)}\otimes 1\otimes a^{(I)}=-\sum D_{I,J}\otimes a^{(J)}\otimes a^{(I)}.$$ Therefore $D=0$. With $D=0$, we have $$A=\sum\limits_{i=1}^{r'}(a'_i\otimes a'_{r'+1}-a'_{i+r'}\otimes a')+s'\otimes 1-1\otimes s'.$$ Thus $\mathfrak{G}$ is isomorphic to Type (iii).\\
\noindent \textbf{Case 4:}
Suppose that $C=0$. Then $A$ is either equal to zero or of the form $$\alpha'=\sum\limits_{i=1}^{r'}(a'_i\otimes a_{i+r'}'-a'_{r'+i}\otimes a'_i)+s'\otimes 1-1\otimes s':=\sum\limits_{I}\alpha'_I\otimes a^{(I)}$$ for some $a_j'(1\leq j\leq 2r'), s'\in\mathfrak{b}$. From (PC7), we get that
\begin{eqnarray}\label{eq14}\begin{array}{l}\sum\limits_K B_{I,J}a^{(I)}a^{(K)}\otimes(\alpha_{I'}\otimes a^{(J)}a^{(I'-K)}-a^{(J)}a^{(I'-K)}\otimes \alpha_{I'})=0.\end{array}\end{eqnarray} Considering the first factors in the tensor products in Eq. \eqref{eq14}, we see that $B=0$. Then, from (PC5) we get $D=0$. This case has already been eliminated.
 \end{proof}
We now present several key examples that illustrate the generality and applicability of quasi-twilled Lie pseudoalgebras.\\
Let $(\mathfrak{g}, [\cdot * \cdot]_{\mathfrak{g}})$ be a Lie pseudoalgebra. For $p \in \mathbb{C}$, define a bracket operation $[\cdot* \cdot]_M$ on $\mathfrak{g} \boxplus \mathfrak{g}$ by
\begin{align}\label{eq2222}
 [(x, u)*(y, v)]_M = ([x * v]_{\mathfrak{g}} -((12)\otimes_Hid) [y* u]_{\mathfrak{g}},\ p [x* y]_{\mathfrak{g}} + [u* v]_{\mathfrak{g}}), \quad \forall x, y, u, v \in \mathfrak{g}. \end{align}
That is, $\pi= \rho= 0$, $\eta = \mu = [\cdot*\cdot]_{\mathfrak{g}}$, and $\theta = p{[\cdot* \cdot]}_{\mathfrak{g}}$ in the above decomposition. Then $[\cdot*\cdot]_M$ is a Lie pseudoalgebra structure on $\mathfrak{g} \boxplus \mathfrak{g}$. Denote this Lie pseudoalgebra by $ \mathfrak{g} \oplus_M \mathfrak{g}=\Big(\mathfrak{g} \boxplus \mathfrak{g}, [\cdot*\cdot]_M \Big).$ 
\begin{ex} \label{ggg}
Let $(\mathfrak{g}, {[\cdot*\cdot]}_{\mathfrak{g}})$ be a Lie pseudoalgebra. Then $(\mathfrak{g} \oplus_M \mathfrak{g}, \mathfrak{g}, \mathfrak{g})$ is a quasi-twilled Lie pseudoalgebra.
\end{ex}
Let $\rho : \mathfrak{g} \otimes \mathfrak{h}\to H^{\otimes 2}\otimes_H\mathfrak{h}$ be an action of a Lie pseudoalgebra $\mathfrak{g}$ on a Lie pseudoalgebra $\mathfrak{h}$. For any $p \in \mathbf{k}$, then $(\mathfrak{g} \boxplus \mathfrak{h}, [\cdot*\cdot]_\rho)$ is a Lie pseudoalgebra, where the Lie pseudobracket $[\cdot* \cdot]_\rho$ is given by
\begin{align}
 [(x, u)* (y, v)]_\rho = ({[x*y]}_{\mathfrak{g}}, \rho(x\otimes v) -((12)\otimes_Hid) \rho(y\otimes u )+ p [u*v]_{\mathfrak{h}}), \quad \forall x, y \in \mathfrak{g}, u, v \in \mathfrak{h}.
\end{align}
This Lie pseudoalgebra is denoted by $\mathfrak{g} \ltimes_\rho \mathfrak{h}$ and is called the action Lie pseudoalgebra.
\begin{ex}\label{actionLie}
Let $\rho: \mathfrak{g} \otimes \mathfrak{h}\to H^{\otimes 2}\otimes_H\mathfrak{h}$ be an action of a Lie pseudoalgebra $\mathfrak{g}$ on a Lie pseudoalgebra $\mathfrak{h}$. Then $(\mathfrak{g} \ltimes_\rho \mathfrak{h}, \mathfrak{g}, \mathfrak{h})$ is a quasi-twilled Lie pseudoalgebra.
\end{ex}
\begin{ex}\label{SEMDIRPRODLIEALG}
Let $(M,\rho)$ be a representation of the Lie pseudoalgebra $\mathfrak{g}$. Then $\mathfrak{g}\boxplus M$ is a direct sum of left $H$-modules, which carries the structure of Lie pseudoalgebra by the pseudobracket given by \begin{align}
 [(x, u)* (y, v)]_\rho = ([x*y]_{\mathfrak{g}}, \rho(x)*v - ((12)\otimes_Hid)\rho(y)*u), \quad \forall u, v \in M, x, y \in \mathfrak{g}.
\end{align}
This Lie pseudoalgebra is denoted by $\mathfrak{g} \ltimes_\rho M$, which is called the semidirect product of Lie pseudoalgebra. The corresponding quasi-twilled Lie pseudoalgebra is given by $(\mathfrak{g} \ltimes_\rho M, \mathfrak{g}, M)$.
\end{ex}
Let $(\mathfrak{g}, [\cdot*\cdot]_{\mathfrak{g}})$ and $(\mathfrak{h}, [\cdot*\cdot]_{\mathfrak{h}})$ be two Lie pseudoalgebras. Then there is the direct product of Lie pseudoalgebras given by \begin{align}
 {[(x, u)*(y, v)]}_\oplus= ([x* y]_{\mathfrak{g}}, [u*v]_{\mathfrak{h}}), \quad \forall x, y \in \mathfrak{g}, u, v \in \mathfrak{h}.
\end{align} is a Lie pseudoalgebra denoted by $\mathfrak{g} \oplus_\oplus \mathfrak{h}= (\mathfrak{g} \boxplus \mathfrak{h}, {[\cdot* \cdot]}_\oplus)$.
 \begin{ex}\label{directprod}
The direct product Lie pseudoalgebra is also a quasi-twilled Lie pseudoalgebra denoted by $(\mathfrak{g} \oplus_\oplus\mathfrak{h}, \mathfrak{g}, \mathfrak{h})$.
\end{ex}
\begin{ex}\label{excocycle}
Let $(M,\rho)$ be a representation of the Lie pseudoalgebra $\mathfrak{g}$ and $\omega \in C^2(\mathfrak{g}, M)$ be a $2$-cocycle. Then $(\mathfrak{g} \ltimes_\rho^\omega M, \mathfrak{g}, M)$ is a quasi-twilled Lie pseudoalgebra, where $\mathfrak{g} \ltimes_\rho^\omega M$ is the semidirect product Lie pseudoalgebra with the pseudobracket given by
\begin{align}
 [(x, u)*(y, v)]_{\rho}^{\omega} = ([x*y]_{\mathfrak{g}}, \rho(x)*v -((12)\otimes_Hid)\rho(y)* u + \omega(x, y)), \quad \forall x, y \in \mathfrak{g}, u, v \in M.
\end{align}
\end{ex} 
\begin{ex}\label{exex}
As a special case of the previous example, consider $M= \mathfrak{g}$, $\rho = \mathrm{ad}$, and $\omega (x, y) = [x*y]_{\mathfrak{g}}$. Then we obtain a quasi-twilled Lie pseudoalgebra $(\mathfrak{g} \ltimes_{\mathrm{ad}}^\omega \mathfrak{g}, \mathfrak{g}, \mathfrak{g})$.
\end{ex}
\begin{rem}
The above examples can be unified via extensions of Lie pseudoalgebras. Recall that a Lie pseudoalgebra $\mathfrak{G}$ is an extension of a Lie pseudoalgebra $\mathfrak{g}$ by a Lie pseudoalgebra $\mathfrak{h}$, if we have the exact sequence 
\begin{align*}
 0 \longrightarrow \mathfrak{h} \xrightarrow{inc} \mathfrak{G} \xrightarrow{ proj} \mathfrak{g} \longrightarrow 0.
\end{align*}
If there is a section $s : \mathfrak{g} \to \mathfrak{G}$ of the map $proj$, then $\mathfrak{G}$ is equal to $s(\mathfrak{g}) \boxplus inc(\mathfrak{h})$, and $inc(\mathfrak{h})$ is a Lie subpseudoalgebra. Thus, $(\mathfrak{G}, s(\mathfrak{g}), inc(\mathfrak{h}))$ is a quasi-twilled Lie pseudoalgebra.
\end{rem}
 \begin{defn}A matched pair of Lie pseudoalgebras consists of a pair of Lie pseudoalgebras $(\mathfrak{g}, \mathfrak{h})$, an $H^{\otimes 2}$-linear map $\rho : \mathfrak{g} \otimes \mathfrak{h}\to H^{\otimes 2}\otimes_H \mathfrak{h}$ and $H^{\otimes 2}$-linear map $\eta:\mathfrak{h} \otimes \mathfrak{g} \to H^{\otimes 2}\otimes_H\mathfrak{g}$ such that the following identities hold 
\begin{align*}
\rho(x\otimes [u*v]_{\mathfrak{h}})&= [\rho(x\otimes u)* v]_{\mathfrak{h}}+((12)\otimes_Hid)([u*\rho(x\otimes v)]_{\mathfrak{h}} - \rho(\eta(u\otimes x)\otimes v))\\
&\qquad + ((123)\otimes_Hid)\rho(\eta(v\otimes x)*u), \\
\eta(u\otimes [x* y]_{\mathfrak{g}}) &= [\eta(u\otimes x)* y]_{\mathfrak{g}} + ((12)\otimes _Hid)([x* \eta(u\otimes y)]_{\mathfrak{g}}-\eta(\rho(x\otimes u)\otimes y)) \\
&\qquad +((123)\otimes_Hid) \eta(\rho(y\otimes u)\otimes x),
\end{align*}
for all $x, y \in \mathfrak{g}$ and $u, v \in \mathfrak{h}$. We will denote a matched pair of Lie pseudoalgebras by $(\mathfrak{g}, \mathfrak{h}; \rho, \eta)$, or simply by $(\mathfrak{g}, \mathfrak{h})$. 
 \end{defn}
\begin{defn}
 Let $(\mathfrak{g}, \mathfrak{h}; \rho, \eta)$ be a matched pair of Lie pseudoalgebras. Then there is a Lie pseudoalgebra structure on the direct sum of left $H$-modules, $\mathfrak{g} \boxplus \mathfrak{h}$ with the pseudobracket $[\cdot*\cdot]_{\bowtie}$ given~by
\begin{align}
\begin{array}{lll} [(x, u)*(y, v)]_{\bowtie}&= \Big([x*y]_{\mathfrak{g}} + \eta(u\otimes y )
- ((12)\otimes_Hid)\eta(v\otimes x),\ [u* v]_{\mathfrak{h}} \\ & \qquad\qquad + \rho(x\otimes v) -((12)\otimes_Hid) \rho(y\otimes u)\Big).\end{array}
\end{align}
Denote this Lie pseudoalgebra by $\mathfrak{g} \bowtie \mathfrak{h}=\Big(\mathfrak{g} \boxplus \mathfrak{h},\ [\cdot*\cdot]_{\bowtie}\Big)$. 
\end{defn}
\begin{ex}\label{exmax}
Let $(\mathfrak{g}, \mathfrak{h}; \rho, \eta)$ be a matched pair of Lie pseudoalgebras. Then $(\mathfrak{g} \bowtie \mathfrak{h}, \mathfrak{g}, \mathfrak{h})$ is a quasi-twilled Lie pseudoalgebra. Note that $H^{\otimes 2}$-linear maps $\eta$ and $\rho$ in the definition of a matched pair are related by the following relations \begin{align*}
 \rho(x)*v = \rho(x\otimes v),\qquad \eta(u)*(y) = -((12)\otimes_Hid)\eta(y\otimes u).
\end{align*}
\end{ex}
\begin{rem}
In Section $3$, we introduce the notion of a deformation map of type I ($\mathfrak{D}$-map) of a quasi-twilled pseudoalgebra over a cocommutative Hopf algebra $ H $. We show that $\mathfrak{D}$-maps of the quasi-twilled pseudoalgebras given in Examples~\ref{ggg}, \ref{actionLie}, \ref{SEMDIRPRODLIEALG}, and~\ref{directprod} unify several structures, including modified r-matrices (solutions of the modified classical Yang-Baxter equation), crossed homomorphisms of weight $p$, derivations between pseudoalgebras, and homomorphisms of pseudoalgebras.
\end{rem}
\begin{rem}
In Section 4, we define the notion of a deformation map of type II ($\mathcal{D}$-map) of a quasi-twilled pseudoalgebra over $ H $. We show that $\mathcal{D}$-maps of the quasi-twilled pseudoalgebras given in Examples ~\ref{actionLie}, \ref{SEMDIRPRODLIEALG}, \ref{excocycle}, \ref{exex}, and~\ref{exmax} unify several structures, including 
 weighted relative Rota-Baxter operators, twisted Rota-Baxter operators, Reynolds operators,
 and deformation maps of matched pairs of pseudoalgebras.
\end{rem}
\section{ The controlling algebras and cohomologies of deformation maps of type I}
In this section, $(\mathfrak{G}, \mathfrak{g}, \mathfrak{h})$ is a quasi-twilled Lie pseudoalgebra with a pseudobracket on $\mathfrak{G}$, which is denoted by
$\Omega = \pi + \rho + \mu + \eta + \theta,$
where $\pi$, $\rho $, $\mu $, $\eta$, and $\theta$ are same as defined in Section 2.
\subsection{ Deformation maps of type I of a quasi-twilled Lie pseudoalgebra}
In this subsection, we introduce the notion of deformation maps of type I of a quasi-twilled Lie pseudoalgebra, which unify modified $r$-matrices, crossed homomorphisms, derivations, and homomorphisms between Lie pseudoalgebras.
\begin{defn}
Let $(\mathfrak{G}, \mathfrak{g}, \mathfrak{h})$ be a quasi-twilled Lie pseudoalgebra. A deformation map of type I ($\mathfrak{D}$-map for short) of $(\mathfrak{G}, \mathfrak{g}, \mathfrak{h})$ is a left $H$-module homomorphism $D:\mathfrak{g} \to \mathfrak{h}$ such that
\begin{align*}
& D (\pi(x\otimes y)+\eta(x \otimes D(y)) -((12)\otimes _Hid) \eta(y \otimes D(x)) \\&\quad= \mu (D(x)\otimes D(y)) + \rho(x\otimes D(y)) -((12)\otimes_Hid) \rho(y\otimes D(x)) + \theta(x\otimes y).\nonumber
\end{align*} 
\end{defn}
\begin{rem}
$\mathfrak{D}$-maps may not always exist. Consider the quasi-twilled Lie pseudoalgebra $(\mathfrak{g} \ltimes_\rho^\omega M, \mathfrak{g}, M)$ introduced in Example \eqref{excocycle}. A left $H$-module homomorphism $D: \mathfrak{g} \to \mathfrak{h}$ is a $\mathfrak{D}$-map if and only if
\begin{align*}
 \omega(x\otimes y) = -\left(\rho(x\otimes D(y)) - ((12)\otimes_Hid)\rho(y\otimes D(x)) - D([x*y]_{\mathfrak{g}})\right) = -d_{\mathrm{CE}}(D) (x\otimes y), \quad \forall x, y \in \mathfrak{g},\end{align*}
 where $d_{\mathrm{CE}}$ denotes the corresponding Chevalley-Eilenberg coboundary operator of the Lie pseudoalgebra $\mathfrak{g}$ with coefficients in the representation $(M, \rho)$. Thus, the quasi-twilled Lie pseudoalgebra $(\mathfrak{g} \ltimes_\rho^\omega M, \mathfrak{g}, M)$ admits a $\mathfrak{D}$-map if and only if $\omega$ is an exact $2$-cocycle.
\end{rem}
\noindent Let $D:\mathfrak{g} \to \mathfrak{h}$ be a homomorphism of the left $H$ module. Define the graph of $D$ as\begin{align*}
 \mathrm{Gr}(D) = \{(x, D(x)) \mid x \in \mathfrak{g}\}.
\end{align*}
\begin{prop}\label{PROPGRAPH}
A homomorphism of the left $H$-module $D : \mathfrak{g} \to \mathfrak{h}$ is a $\mathfrak{D}$ map if and only if $\mathrm{Gr}(D)$ is a subpseudoalgebra. In this case, $(\mathfrak{h}, \mathrm{Gr}(D))$ forms a matched pair of Lie pseudoalgebras.
\end{prop}
\begin{proof}
For all $(x, D(x)), (y, D(y)) \in \mathrm{Gr}(D)$, we have
\begin{align*}
 \Omega \big((x, D(x)), (y, D(y))\big) =& \Big(\pi(x\otimes y) + \eta(x \otimes D(y) ) -((12)\otimes_Hid) \eta(y\otimes D(x)), \mu(D(x)\otimes D(y)) \\&\qquad+ \rho(x \otimes D(y))- ((12)\otimes_Hid) \rho(y \otimes D(x)) + \theta(x\otimes y)\Big).
\end{align*}
Thus, $\mathrm{Gr}(D)$ is a Lie subpseudoalgebra of $\mathfrak{G}$, i.e., $ \Omega\big((x, D(x)), (y, D(y))\big) \in \mathrm{Gr}(D),$ if and only if the following identity holds:
\begin{align*} \mu(D(x)\otimes D(y)) +& \rho(x \otimes D(y)) - ((12)\otimes_Hid) \rho(y \otimes D(x)) + \theta(x\otimes y)\\
=& D (\pi(x\otimes y) + \eta(x \otimes D(y) )-((12)\otimes_Hid) \eta(y\otimes D(x))).
\end{align*}
This condition precisely characterizes that $D$ is a $\mathfrak{D}$-map of the quasi-twilled Lie pseudoalgebra $(\mathfrak{G}, \mathfrak{g}, \mathfrak{h})$. Since $\mathfrak{G} = \mathrm{Gr}(D) \boxplus \mathfrak{h}$, it follows that $(\mathfrak{h}, \mathrm{Gr}(D))$ is a matched pair of Lie pseudoalgebras.
\end{proof}
 \begin{rem}
Let $(\mathfrak{G}, \mathfrak{g}, \mathfrak{h})$ be a quasi-twilled Lie pseudoalgebra. Then we have $\mathfrak{G} = \mathfrak{g} \boxplus \mathfrak{h}$. Since any complement of $\mathfrak{h}$ in $\mathfrak{G}$, i.e., a submodule isomorphic to $\mathfrak{g}$, can be realized as the graph of a left $H$-module homomorphism $D:\mathfrak{g}\to \mathfrak{h}$, it follows from Proposition~\ref{PROPGRAPH} that finding a submodule $\mathfrak{g}'\subseteq \mathfrak{G}$ which is a complement of $\mathfrak{h}$ such that $(\mathfrak{h}, \mathfrak{g}')$ forms a matched pair of Lie pseudoalgebras is equivalent to finding a $\mathfrak{D}$-map for the quasi-twilled Lie pseudoalgebra $(\mathfrak{G}, \mathfrak{g}, \mathfrak{h})$.
\end{rem} 
\begin{ex}
Consider the quasi-twilled Lie pseudoalgebra $(\mathfrak{g} \oplus_M \mathfrak{g}, \mathfrak{g}, \mathfrak{g})$ defined in Example \ref{ggg}. In this case, a $\mathfrak{D}$-map of $(\mathfrak{g} \oplus_M \mathfrak{g}, \mathfrak{g}, \mathfrak{g})$ is a left $H$-linear map $D : \mathfrak{g} \to \mathfrak{g}$ satisfying the identity
\begin{align*} {[D(x)*D(y)]}_{\mathfrak{g}} = D({[D(x)* y]}_{\mathfrak{g}} + {[x*D(y)]}_{\mathfrak{g}}) - p {[x*y]}_{\mathfrak{g}},\quad \forall x,y \in \mathfrak{G}.
\end{align*}
 This equation is similar to \emph{modified Yang-Baxter equation}, whose solutions are referred to as \emph{modified $r$-matrices} on the Lie pseudoalgebra $(\mathfrak{g}, [\cdot*\cdot]_{\mathfrak{g}})$, \cite{BDK}. 
\end{ex}
 %\begin{defn}\label{def:crossed-module}A \textbf{crossed module of Lie $ H $-pseudoalgebras} is a quadruple $ (\mathfrak{g}, \mathfrak{h}, \phi, \rho) $, where $ \mathfrak{g} $ and $ \mathfrak{h} $ are Lie $ H $-pseudoalgebras, $ \phi : \mathfrak{h} \to \mathfrak{g} $ is a morphism of Lie $ H $-pseudoalgebras, and $ \rho \in \mathrm{Hom}_{H^{\otimes 2}}(\mathfrak{g} \otimes \mathfrak{h},\ H^{\otimes 2} \otimes_H \mathfrak{h}) $ is a $H$-linear map, such that:\begin{enumerate}\item $ (\mathfrak{h}, \rho) $ is a representation of the Lie $ H $-pseudoalgebra $ \mathfrak{g} $, \item The following compatibility conditions hold for all $ x \in \mathfrak{g} $, $ u, v, w \in \mathfrak{h} $: \begin{align} (\mathrm{id}_{H^{\otimes 2}} \otimes_H \phi)\big( \rho(x, v) \big) &= [x * \phi(v)], \label{eq:crossed-cond1} \\ \rho(\phi(u), v) &= [u * v]_\mathfrak{h}, \label{eq:crossed-cond2} \\ [\rho(x, v) * w]_\mathfrak{h} &= \rho(x, [v * w]_\mathfrak{h}) - [v * \rho(x, w)]_\mathfrak{h}. \label{eq:crossed-cond3}\end{align}\end{enumerate}Here, $ [\cdot * \cdot] $ and $ [\cdot * \cdot]_\mathfrak{h} $ denote the pseudobrackets of $ \mathfrak{g} $ and $ \mathfrak{h} $, respectively, in the pseudotensor category $ \mathcal{M}^*(H) $.\end{defn}

\begin{ex}
Consider the quasi-twilled Lie pseudoalgebra $(\mathfrak{g} \ltimes_\rho \mathfrak{h}, \mathfrak{g}, \mathfrak{h})$, as introduced in Example \ref{actionLie}, which arises from the action Lie pseudoalgebra $\mathfrak{g} \ltimes_\rho \mathfrak{h}$. In this case, a $\mathfrak{D}$-map for $(\mathfrak{g} \ltimes_\rho \mathfrak{h}, \mathfrak{g}, \mathfrak{h})$ is a left $H$-module homomorphism $D : \mathfrak{g} \to \mathfrak{h}$ satisfying the identity:
\begin{align*}
 D({[x* y]}_{\mathfrak{g}}) = \rho(x \otimes D(y)) - ((12)\otimes_Hid) \rho(y \otimes D(x) )+ p{[D(x)* D(y)]}_{\mathfrak{h}},
\end{align*}
Such a map $D$ is called \emph{crossed homomorphism of weight $p$} from the Lie pseudoalgebra $\mathfrak{g}$ to the Lie pseudoalgebra $\mathfrak{h}$.
\end{ex}\noindent The analogous notion in the context of ordinary Lie algebras is known as a crossed homomorphism and was introduced in \cite{Lue}. If $\mathfrak{h}$ is commutative Lie pseudoalgebra, then any crossed homomorphism from $\mathfrak{g}$ to $\mathfrak{h}$ reduces to a derivation from $\mathfrak{g}$ to $\mathfrak{h}$ with respect to the representation $(\mathfrak{h}, \rho).$
\begin{rem}
 A crossed homomorphism from $\mathfrak{g}$ to itself, with respect to the adjoint representation, is also referred to as a \emph{differential operator of weight $p$} on the Lie pseudoalgebra $\mathfrak{g}$.
\end{rem}
\begin{ex}
 If the action $\rho$ of $\mathfrak{g}$ on $\mathfrak{h}$ is trivial, then any crossed homomorphism from $\mathfrak{g}$ to $\mathfrak{h}$ reduces to a Lie pseudoalgebra homomorphism from $\mathfrak{g}$ to $\mathfrak{h}$. If $\mathfrak{h}$ is commutative, then crossed homomorphism becomes precisely a derivation from $\mathfrak{g}$ to $\mathfrak{h}$ with respect to the representation $(\mathfrak{h}; \rho)$.
\end{ex}
%{\color{red}\begin{rem}CAN DELETE BECAUSE I DONT HAVE ANY INFOR ABOUT CONFORMAL CASE, BETTER TO EXWMPT THIS REMARK, BUT LATER Note that crossed homomorphisms of weight $-1$ are $\varepsilon$-derivations on the Lie algebras, which play crucial roles in the Jacobian conjecture \cite{JC1} and the Mathieu-Zhao subspace theory \cite{MZ}. On the other hand, crossed homomorphisms of weight 1 are deeply related to the representation theory of Cartan-type Lie algebras \cite{Cartan} and post-Lie algebras \cite{PostLie}. \end{rem}}
\begin{ex}
Consider the quasi-twilled Lie pseudoalgebra $(\mathfrak{g} \ltimes_\rho M, \mathfrak{g}, M)$, as defined in Example~\ref{SEMDIRPRODLIEALG}, arising from the semidirect product Lie pseudoalgebra $\mathfrak{g} \ltimes_\rho M$. In this setting, a $\mathfrak{D}$-map is a left $H$-module homomorphism $D : \mathfrak{g} \to M$ satisfying the identity:
\begin{align*}
(id\otimes_H D)([x* y]_{\mathfrak{g}}) = \rho(x\otimes D(y)) - ((12)\otimes_Hid)\rho(y\otimes D(x)),
\end{align*}
This condition precisely states that $D$ is a \emph{derivation} from the Lie pseudoalgebra $(\mathfrak{g}, [\cdot* \cdot]_{\mathfrak{g}})$ to the representation $(M;\rho)$. In particular, if $\rho$ is the adjoint representation of $\mathfrak{g}$ on itself, then $D$ becomes an ordinary derivation of $\mathfrak{g}$.
\end{ex}
\begin{ex}
Consider the quasi-twilled Lie pseudoalgebra $(\mathfrak{g} \boxplus \mathfrak{h}, \mathfrak{g}, \mathfrak{h})$, introduced in Example \ref{directprod}, which arises from the direct product Lie pseudoalgebra structure on $\mathfrak{g} \boxplus \mathfrak{h}$. In this case, a $\mathfrak{D}$-map of $(\mathfrak{g} \boxplus \mathfrak{h}, \mathfrak{g}, \mathfrak{h})$ is a left $H$-module homomorphism $D : \mathfrak{g} \to \mathfrak{h}$ satisfying the identity:
 \begin{align*}
(id\otimes_H D)([x* y]_{\mathfrak{g}}) = [D(x)* D(y)]_{\mathfrak{h}}.
\end{align*} 
Such a map $D$ is exactly a \emph{homomorphism of Lie pseudoalgebras} from $(\mathfrak{g}, [\cdot* \cdot]_{\mathfrak{g}})$ to $(\mathfrak{h}, [\cdot*\cdot]_{\mathfrak{h}})$.
\end{ex}
 At the end of this subsection, we illustrate the role of $\mathfrak{D}$-maps in twist theory. 

\noindent Let $D : \mathfrak{g} \to \mathfrak{h}$ be a left $H$-module map. Suppose that $D^2 = 0$ and define $X_D=[\cdot, D]_{\mathrm{NR}}$. Then $X_D$ is a derivation of the graded Lie pseudoalgebra
\begin{align*}
 \left(\bigoplus_{n=0}^{\infty} C^{n+1}(\mathfrak{g} \boxplus \mathfrak{h}, \mathfrak{g} \boxplus \mathfrak{h}), [\cdot, \cdot]_{\mathrm{NR}}\right).
\end{align*} It follows that the exponential map $$e^{[\cdot, D]_{\mathrm{NR}}}= e^{X_D}:=id+X_D+\frac{X_D^2}{2!}+\frac{X_D^3}{3!}+\cdots,$$ is an automorphism of the above-mentioned graded Lie pseudoalgebra. Moreover, $X_D^2$ can be defined as $[[\cdot, D]_{NR}, D]_{NR}$, and similarly for higher powers when $n\geq 3$. 

\begin{defn}
Consider that $D : \mathfrak{g} \to \mathfrak{h}$ be a left $H$-module homomorphism. The transformation $\Omega^D := e^{X_D} \Omega $ is called the twisting of $\Omega$ by $D$.
\end{defn}
\begin{prop}\label{proptransformation}
Similar to \cite[Lemma 4.2 and Proposition 4.3]{YL}, we have the identity:
\begin{align*}
 \Omega ^{D}= e^{-D} \circ \Omega \circ (e^D \otimes e^D),
\end{align*}
which defines a Lie pseudoalgebra structure on $\mathfrak{g} \boxplus \mathfrak{h}$. Furthermore, the map $e^D : (\mathfrak{G}, \Omega^D) \to (\mathfrak{G}, \Omega)$ is an isomorphism between Lie pseudoalgebras.
\end{prop}
\begin{thm}\label{thm33}
Let $(\mathfrak{G}, \mathfrak{h}, \mathfrak{g})$ be a quasi-twilled Lie pseudoalgebra and $D : \mathfrak{g} \to \mathfrak{h}$ a left $H$-module homomorphism. Then $((\mathfrak{G}, \Omega^D), \mathfrak{g}, \mathfrak{h})$ is also a quasi-twilled Lie pseudoalgebra. Moreover, since $\Omega^D$ is a $2$-cochain, it can be decomposed into five substructures: $\Omega^D = \pi^D + \rho^D + \mu^D + \eta^D + \theta^D$, where 
 $\pi^D: \mathfrak{g}\otimes\mathfrak{g} \to H^{\otimes 2}\otimes_H\mathfrak{g}$, $\rho^D :\mathfrak{g} \otimes \mathfrak{h}\to H^{\otimes2}\otimes_H\mathfrak{h}$, $\mu^D: \mathfrak{h}\otimes \mathfrak{h} \to H^{\otimes2}\otimes_H\mathfrak{h}$, $\eta^D:\mathfrak{g} \otimes \mathfrak{h}\to H^{\otimes2}\otimes_H\mathfrak{g}$ and $\theta^D: \mathfrak{g}\otimes \mathfrak{g}\to H^{\otimes2}\otimes_H\mathfrak{h}$. The relation between $\Omega ^D$ and $\Omega$ are given explicitly as follows:
\begin{align*}\mu^D(u\otimes v) &= \mu(u\otimes v), \\
\eta^D(x\otimes v) &= 
 \eta(x\otimes v), \\
\pi^D(x\otimes y) &= \pi(x\otimes y) +\eta(x\otimes D(y)) - ((12)\otimes_Hid)\eta(y \otimes D(x)), \\
\rho^D(x\otimes v) &= \rho(x\otimes v) + \mu(D(x)\otimes v) -(id\otimes_H D)( \eta(x \otimes v)), \\
\theta^D(x\otimes y) &= \theta(x\otimes y) + \rho(x \otimes D(y))-((12)\otimes_Hid)\rho(y \otimes D(x)) -(id\otimes_H D)(\pi(x\otimes y)) \\
&\qquad+ \mu(D(x)\otimes D(y)) -(id\otimes_H D)(\eta(x\otimes D(y))) +((12)\otimes_HD)(\eta(y\otimes D(x))),
\end{align*}
for all $x, y \in \mathfrak{g}$, $u, v \in \mathfrak{h}$. Consequently, a left $H$-module homomorphism $D : \mathfrak{g} \to \mathfrak{h}$ is a $\mathfrak{D}$-map if and only if the Lie pseudoalgebras $(\mathfrak{g}, \pi^D)$ and $(\mathfrak{h}, \mu^D)$ form a matched pair of Lie pseudoalgebras.
\end{thm}
\begin{proof}
By Proposition \ref{proptransformation}, we compute
\begin{align*}\Omega^D\big((0, u), (0, v)\big) = e^{-D} \Omega(e^D(0, u), e^D(0, v)) = \Omega \big((0, u), (0, v)\big) \in \mathfrak{h}.\end{align*}
Thus, $\mathfrak{h}$ is a Lie subpseudoalgebra of $(\mathfrak{G}, \Omega^D)$, and $\mathfrak{g}$ retains its structure as a left $H$-module. This implies that $((\mathfrak{G}, \Omega^D), \mathfrak{g}, \mathfrak{h})$ is indeed a quasi-twilled Lie pseudoalgebra.
%For all $x, y \in \mathfrak{g}$, $u, v \in \mathfrak{h}$, 
Again using Proposition~\ref{proptransformation}, we find 
\begin{align*}
 \Omega^D\big((x, 0), (y, 0)\big)= & \Big(\pi^D(x\otimes y), \theta^D(x\otimes y)\Big)\\
=& \Big(\pi(x\otimes y) + \eta(x \otimes D(y)) -((12)\otimes_Hid) \eta(y\otimes D(x)),\ \theta(x\otimes y) + \rho(x\otimes D(y))\\&\qquad- ((12)\otimes_Hid)\rho(y\otimes D(x)) -(id\otimes_H D)(\pi(x\otimes y)+\eta(x \otimes D(y))) \\&\qquad+ \mu(D(x)\otimes D(y)) + ((12)\otimes_HD)\eta(y\otimes D(x))\Big),\end{align*}
and
\begin{align*} \Omega ^D\big((x, 0), (0, v)\big)&= \Big(\eta^D(x\otimes v), \rho^D(x\otimes v)\Big)= \Big( \eta(x \otimes v), \rho(x\otimes v) + \mu(D(x)\otimes v) - D( \eta(x\otimes v))\Big).
\end{align*} 
These identities show that the deformed maps $\pi^D, \rho^D, \eta^D, \mu^D, \theta^D$ satisfy the compatibility conditions inherited from the original structure and the action of the $\mathfrak{D}$-map $D$. 
Therefore, the pair $(\mathfrak{g}, \pi^D)$ and $(\mathfrak{h}, \mu^D)$ forms a matched pair of Lie pseudoalgebras precisely when $D$ satisfies the defining identity of a $\mathfrak{D}$-map. This completes the proof.
\end{proof} 
%Apply the above result to the quasi-Lie bialgebra $(\mathfrak{g}, \pi, \mu, \theta)$, a $\mathfrak{D}$-map gives rise to a Lie bialgebra.
\begin{prop}\label{prop:lie-bialgebra}
Let $ D : \mathfrak{g} \to \mathfrak{h} $ be a $\mathfrak{D}$-map of the quasi-twilled Lie $ H $-pseudoalgebra $ (\mathfrak{G}, \mathfrak{g}, \mathfrak{h}) = (\mathfrak{g} \boxplus \mathfrak{h}, \mathfrak{g}, \mathfrak{h}) $, where $ \mathfrak{h} = \mathfrak{g}^* $ is the dual $ H $-module. Suppose the structure maps $ \pi, \mu, \theta $ define a quasi-Lie $ H $-bialgebra structure on $ \mathfrak{g} $, and that $ D $ is skew-adjoint, i.e., $ D = -D^* $. Then the pair $ (\mathfrak{g}, \pi^D) $ and $ (\mathfrak{h}, \mu^D) $, with the twisted pseudobrackets defined by
\begin{align*}
\pi^D(x \otimes y) &= \pi(x \otimes y) + \eta(x \otimes D(y)) - ((12) \otimes_H \mathrm{id})\,\eta(y \otimes D(x)), \\
\mu^D(u \otimes v) &= \mu(u \otimes v) + \rho(D^*(u) \otimes v) - ((12) \otimes_H \mathrm{id})\,\rho(D^*(v) \otimes u),
\end{align*}
for $ x, y \in \mathfrak{g} $, $ u, v \in \mathfrak{h} $, forms a Lie $ H $-bialgebra.
\end{prop}

\subsection{ The controlling algebra of $\mathfrak{D}$-maps} In this subsection, we present the controlling algebra of deformation maps of type I, which, which takes the form of a \emph{curved $\mathfrak{L}_\infty$-pseudoalgebra}. An important consequence of this construction is that it also provides the first known controlling algebra for modified $r$-matrices.
 \begin{defn} \label{def:curved-Linf-conf}
A \emph{curved $\mathfrak{L}_\infty$-pseudoalgebra} is a $\mathbb{Z}$-graded left $H$-module 
$ \mathfrak{g} = \bigoplus_{i \in \mathbb{Z}} \mathfrak{g}_i,$ 
equipped with a collection of graded symmetric, $H^{\otimes k}$-linear operations 
$$l_k : \mathfrak{g}^{\otimes k} \to H^{\otimes k}\otimes_H\mathfrak{g}, \quad \text{for\ } k \geq 0,$$
of degree $1$, where $l_0 \in \mathfrak{g}_1$ is referred to as the \emph{curvature element}. These maps satisfy:
\begin{itemize}
 \item[] Graded pseudo-skew-Symmetry: 
 For any permutation $\sigma \in S_k$,
 \[(\sigma\otimes_Hid) 
 \mathrm{sgn}(\sigma)l_k(x_{\sigma(1)},\dots,x_{\sigma(k)}) = \epsilon(\sigma) {l_k}(x_1,\dots,x_k).\]
 
\item[] Higher pseudo-Jacobi Identity: For $n \geq 1$,
 \[\sum_{i=1}^n \sum_{\sigma \in \mathrm{Sh}(i,n-i)} \mathrm{sgn}(\sigma) \epsilon(\sigma)(\sigma\otimes_Hid) {l_{n-i+1}}\Big({l_i}(x_{\sigma(1)},\dots,x_{\sigma(i)}), x_{\sigma(i+1)},\dots,x_{\sigma(n)}
 \Big) = 0.
 \]
\end{itemize}
We denote such a structure by $(\mathfrak{g}, \{l_k\}_{k=0}^\infty)$. If the curvature element satisfies $l_0 = 0$, then the structure becomes an \emph{$\mathfrak{L}_\infty$-pseudoalgebra}.
\end{defn}
\begin{rem}\label{rem:lk-degree}
The operations $l_k$ have degree $\deg(l_k) = 1$, which is consistent with the standard grading convention for curved $\mathfrak{L}_\infty$-algebras. Specifically, $l_0$ is a curvature element of degree $1$, $l_1$ acts as a differential of degree $1$, and $l_2$ defines a Lie pseudobracket, also of degree $1$. For $k\geq 3$ , the higher operations $l_k$ serve as homotopies that correct the failure of the lower brackets to satisfy the Jacobi identity strictly. 
\end{rem}
\begin{defn}\label{def:MC-twist}
Let $(\mathfrak{g}, \{l_k\}_{k=0}^\infty)$ be a curved $\mathfrak{L}_\infty$-pseudoalgebra over a $\mathbb{Z}$-graded $H$-module $\mathfrak{g} = \bigoplus_{i \in \mathbb{Z}} \mathfrak{g}_i$, equipped with pseudobrackets $
l_{k \geq 0} : \mathfrak{g}^{\otimes k} \to H^{\otimes k}\otimes_H\mathfrak{g},$ of degree $1$. A \emph{Maurer-Cartan element} is a homogeneous element $x \in \mathfrak{g}_0$ of degree $0$ that satisfies the Maurer-Cartan equation
\begin{align*} 
 {l_k}(x\otimes \dots \otimes x) = 0,\text{\ for\ all\ k.}
\end{align*} 
Now suppose $x$ is such a Maurer-Cartan element. For each $k \geq 1$, we define \emph{twisted operations} 
$${l_k^x}(x_1\otimes \dots\otimes x_k):= \sum_{n=0}^\infty \frac{1}{n!}
{l_{k+n}}(\underbrace{x\otimes \dots\otimes x}_{n \text{ times}}\otimes x_1\otimes \dots\otimes x_k),$$ where the formal variables are adjusted appropriately to ensure consistency with grading and $H^{\otimes k}$-linearity conditions. The resulting collection $\{l_k^x\}_{k=1}^\infty$ defines a new $\mathfrak{L}_\infty$-pseudoalgebra structure on $\mathfrak{g}$, called the \emph{twisted $\mathfrak{L}_\infty$-pseudoalgebra to $x$}.
\end{defn}
 %{\color{red}\begin{remark}To ensure the convergence of the series appearing in the definition of Maurer-Cartan elements and Maurer-Cartan twistings above, one needs the $\mathfrak{L}_\infty$-algebra being filtered given by Dolgushev and Rogers in \cite{L_infty_filtered}, or weakly filtered given in \cite{L_infty_weakly_filtered}. Since all the $\mathfrak{L}_\infty$-algebras under consideration in the sequel satisfy the weakly filtered condition, so we will not mention this point anymore.\end{remark}}

\noindent Next, we recall Voronov's derived bracket construction \cite{Voronov}, a powerful method for constructing a curved $\mathfrak{L}_\infty$-algebra from a graded Lie algebra equipped with additional structure.
\begin{defn}[\cite{Voronov}]
A curved $V$-data consists of a quadruple $(\mathfrak{J}, \mathfrak{K}, P, \Delta)$, where
\begin{enumerate}
\item $(\mathfrak{J} = \bigoplus \mathfrak{J}^i, [\cdot, \cdot])$ is a $\mathbb{Z}$-graded Lie algebra,
\item $\mathfrak{K}$ is an abelian graded Lie subalgebra of $(\mathfrak{J}, [\cdot, \cdot])$, %i.e., $\mathfrak{K}\subset \mathfrak{J}$,
\item $P : \mathfrak{J} \to \mathfrak{J}$ is a projection operator (i.e., $P \circ P = P$) such that Image of $P$ is $\mathfrak{K}$, 
and Kernel of $P$ is a graded Lie subalgebra of $(\mathfrak{J}, [\cdot, \cdot])$,%, i.e., $Im(P)=\mathfrak{K}\subset \mathfrak{J}$; i.e. $Ker(P)\subset \mathfrak{J}$.
\item $\Delta \in \mathfrak{J}^1$ is an element of degree $1$, such that $[\Delta, \Delta] = 0$.
\end{enumerate}
In addition, if we have $\Delta \in {ker(P)}^1$ and still $[\Delta, \Delta] = 0$, then we call $(\mathfrak{J}, \mathfrak{K}, P, \Delta)$ as a $V$-data (without curvature). 
\end{defn}
\begin{thm}\label{thmvor}\cite{Voronov}
Let $(\mathfrak{J}, \mathfrak{K}, P, \Delta)$ be a \emph{curved $V$-data}. Then $(\mathfrak{K}, \{l_k\}_{k=0}^{\infty})$ is a curved $\mathfrak{L}_\infty$-algebra, where $l_k$ are defined by
\begin{align*}
 l_0 = P(\Delta), \quad {l_{k\geq 1}} (x_1, \dots, x_k) = P([\cdots{[{[\Delta, x_1]}, x_2]}, \cdots ], x_k]),
\end{align*} for all $x_1,\cdots x_k\in \mathfrak{K}$.
\end{thm}\noindent Based on the this framework, we may give the controlling algebra of $\mathfrak{D}$-maps associated to a quasi-twilled Lie pseudoalgebra.
\begin{thm}\label{thm:pseudo-mc-pseudo}
Let $ (\mathfrak{G}, \mathfrak{g}, \mathfrak{h}) $ be a quasi-twilled pseudoalgebra over a cocommutative Hopf algebra $H$. Then there exists a curved $ V $-data $ (\mathfrak{J}, \mathfrak{K}, P, \Delta) $, where 
 $\mathfrak{J}= \bigoplus_{n \geq 0} C^{n+1}(\mathfrak{g} \boxplus \mathfrak{h}, \mathfrak{g} \boxplus \mathfrak{h})$ is the graded $ H $-module equipped with the Nijenhuis-Richardson bracket $ [\cdot, \cdot]_{\mathrm{NR}} $,
and $ \mathfrak{K} = \bigoplus_{n \geq 0} C^{n+1}(\mathfrak{g}, \mathfrak{h}) $ is an abelian graded Lie subalgebra of $ \mathfrak{J} $.
The map $ P : \mathfrak{J} \to \mathfrak{J} $ is the $ H $-linear projection onto $ \mathfrak{K} $, and $ \Delta = \pi + \rho + \mu + \eta + \theta \in \mathfrak{J}^1 $ satisfies the Maurer-Cartan condition $[\Delta, \Delta]_{\mathrm{NR}} = 0$. This curved $ V $-data induces a curved $ \mathfrak{L}_\infty $-pseudoalgebra structure on $ \mathfrak{K} $ with operations explicitly given by 
$$l_0 = \theta, \quad
l_1(f) = [\pi + \rho, f]_{\mathrm{NR}}, \quad
l_2(f, g) = [[\mu + \eta, f]_{\mathrm{NR}}, g]_{\mathrm{NR}},
\text{ and } l_{k \geq 3} = 0. $$
\noindent Furthermore, an $ H $-module homomorphism $ D : \mathfrak{g} \to \mathfrak{h} $ is a $ \mathfrak{D} $-map of $ (\mathfrak{G}, \mathfrak{g}, \mathfrak{h}) $ if and only if it satisfies the Maurer-Cartan equation 
\begin{align*}
 l_0 + l_1(D) + \frac{1}{2} l_2(D, D) = 0.
\end{align*}
\end{thm}

\begin{proof}
%By the general theory of $ V $-data (see \cite{Voronov}), the quadruple $ (\mathfrak{J}, \mathfrak{K}, P, \Delta) $ induces a curved $ \mathfrak{L}_\infty $-algebra structure on $ \mathfrak{K} $ via: $$l_0 = P(\Delta), \quad l_k(f_1, \dots, f_k) = P([\cdots[[\Delta, f_1]_{\mathrm{NR}}, f_2]_{\mathrm{NR}}, \dots, f_k]_{\mathrm{NR}}),$$for all $ f_1, \dots, f_k \in \mathfrak{K} $.

Recall that $ \Delta = \pi + \rho + \mu + \eta + \theta $. Applying the projection $ P $, which maps into $ \mathfrak{K} = \bigoplus_n C^{n+1}(\mathfrak{g}, \mathfrak{h}) $, we have 
$$
l_0 = P(\Delta) = P(\pi + \rho + \mu + \eta + \theta) = P(\theta) = \theta,
$$
since $ \pi, \rho, \mu, \eta $ map into $ C^*(\mathfrak{g}, \mathfrak{g}) $, $ C^*(\mathfrak{h}, \mathfrak{h}) $, or $ C^*(\mathfrak{h}, \mathfrak{g}) $, all of which lie in $ \ker(P) $.

\noindent For $ f \in C^n(\mathfrak{g}, \mathfrak{h}) $, we compute 
$$
l_1(f) = P([\Delta, f]_{\mathrm{NR}}) = P([\pi + \rho + \mu + \eta + \theta, f]_{\mathrm{NR}}).
$$
Only the components $ \pi $ and $ \rho $, which map $ \mathfrak{g} \to \mathfrak{g} $ and $ \mathfrak{g} \to \mathfrak{h} $, respectively, contribute to a bracket that lands in $ \mathfrak{K} $ under $ P $. Thus 
 $
l_1(f) = [\pi, f]_{\mathrm{NR}} + [\rho, f]_{\mathrm{NR}} = [\pi + \rho, f]_{\mathrm{NR}}.
 $

\noindent For $ f, g \in \mathfrak{K} $, we compute:
$$
l_2(f, g) = P([[\Delta, f]_{\mathrm{NR}}, g]_{\mathrm{NR}}).
$$
Now, $ [\Delta, f]_{\mathrm{NR}} $ may involve $ \mu $ and $ \eta $, which map $ \mathfrak{h} \to \mathfrak{h} $ and $ \mathfrak{h} \to \mathfrak{g} $. Bracketing with $ g $ and projecting to $ \mathfrak{K} $, only $ \mu $ and $ \eta $ survive, resulting in
$$
l_2(f, g) = [[\mu + \eta, f]_{\mathrm{NR}}, g]_{\mathrm{NR}}.
$$
\noindent Higher operations $ l_k $ for $ k \geq 3 $ vanish because repeated bracketing with $ \Delta $ increases the number of inputs and eventually maps into spaces annihilated by $ P $, due to the abelian nature of $ \mathfrak{K} $ and the grading.\\
\noindent Now, let $ D \in C^1(\mathfrak{g}, \mathfrak{h}) $, the Maurer-Cartan equation yields 
$$
l_0 + l_1(D) + \frac{1}{2} l_2(D, D) = 0.
$$
Thus, we get
\begin{align*}
\theta + [\pi + \rho, D]_{\mathrm{NR}} + \frac{1}{2} [[\mu + \eta, D]_{\mathrm{NR}}, D]_{\mathrm{NR}} = 0.
\end{align*}
Expanding in components, we obtain
\begin{align*}
&\theta(x \otimes y) + \rho(x \otimes D(y)) - ((12) \otimes_H \mathrm{id})\rho(y \otimes D(x)) \\
&+ \mu(D(x) \otimes D(y)) - (\mathrm{id} \otimes_H D)\big(\pi(x \otimes y) + \eta(x \otimes D(y))\big) + ((12) \otimes_H \mathrm{id})\big(\eta(y \otimes D(x))\big) = 0.
\end{align*}
Or precisely
\begin{align*}
&(\mathrm{id} \otimes_H D)\big(\pi(x \otimes y) + \eta(x \otimes D(y)) - ((12) \otimes_H \mathrm{id})\,\eta(y \otimes D(x))\big) \\
&= \mu(D(x) \otimes D(y)) + \rho(x \otimes D(y)) - ((12) \otimes_H \mathrm{id})\,\rho(y \otimes D(x)) + \theta(x \otimes y),
\end{align*} for all $ x, y \in \mathfrak{g}$. Indeed this is the defining identity for $D$ to be a $\mathfrak{D}$-map of the quasi-twilled Lie pseudoalgebra $ (\mathfrak{G}, \mathfrak{g}, \mathfrak{h}) $. Thus, $ D $ is a Maurer-Cartan element of the curved $ \mathfrak{L}_\infty $-pseudoalgebra $ (\mathfrak{K}, l_0, l_1, l_2) $ if and only if $ D $ is a $ \mathfrak{D} $-map of the quasi-twilled pseudoalgebra $ (\mathfrak{G}, \mathfrak{g}, \mathfrak{h}) $. This completes the proof.\end{proof}
 \noindent As an immediate application of the above theorem, we obtain the controlling algebra of modified $r$-matrices, namely a curved $\mathfrak{L}_\infty$- pseudoalgebra whose Maurer-Cartan elements are precisely the modified $r$-matrices.
\begin{cor}\label{cor:twilled-to-curved-Linf}
Consider the quasi-twilled Lie pseudoalgebra $(\mathfrak{g} \oplus_M \mathfrak{g}, \mathfrak{g}, \mathfrak{g})$ given by Eq. \eqref{eq2222}. Then the graded space $\left(\bigoplus_{n=0}^\infty C^{n+1}(\mathfrak{g}, \mathfrak{g}), l_0, l_1, l_2\right)$ carries the structure of a curved $\mathfrak{L}_\infty$-pseudoalgebra. The curvature term is
 $ l_0 = \theta$, with $ \theta(x\otimes y) = [x* y]_\mathfrak{g} $ recovers the original pseudobracket on $ \mathfrak{g} $. The unary operation vanishes identically, i.e., $ l_1 = 0 $ reflecting the absence of a nontrivial differential in this case. The binary operation $ l_2 $ is defined for any $ f \in C^p(\mathfrak{g}, \mathfrak{g}) $, $ g \in C^q(\mathfrak{g}, \mathfrak{g}) $ is given by
 {\fontsize{9}{9}\selectfont\begin{align*}
&{l_2}(f,g)\big(x_1\otimes \dots\otimes x_{p+q}\big) 
\\=& \sum_{\sigma \in \mathrm{Sh}(q,1,p-1)} (-1)^{p} \varepsilon(\sigma)(\sigma\otimes_Hid)\quad f\Big([{x_{\sigma(q+1)}}* {g(x_{\sigma(1)},\dots,x_{\sigma(q)})}]\otimes x_{\sigma(q+2)}\otimes \dots\otimes x_{\sigma(p+q)} \Big) \\&
 - \sum_{\sigma \in \mathrm{Sh}(p,1,q-1)} (-1)^{p(q+1)} \varepsilon(\sigma)(\sigma\otimes_Hid)\quad g\Big([{x_{\sigma(p+1)}}*{f(x_{\sigma(1)},\dots,x_{\sigma(p)})}]\otimes x_{\sigma(p+2)}\otimes\dots\otimes
 x_{\sigma(p+q)} \Big) \\ &
 + \sum_{\sigma \in \mathrm{Sh}(p,q)} (-1)^{p(q+1)} \varepsilon(\sigma)(\sigma\otimes_Hid)\quad \Big[{f(x_{\sigma(1)},\dots,x_{\sigma(p)})}*g(x_{\sigma(p+1)},\dots,x_{\sigma(p+q)})\Big]_\mathfrak{g},
\end{align*}}for all homogeneous elements $x_1,\dots,x_{p+q} \in \mathfrak{g}$, 
and where $ \varepsilon(\sigma) $ denotes the Koszul sign associated with the permutation $ \sigma \in S_{p+q}$. 
\end{cor}
\noindent By applying Theorem \ref{thm:pseudo-mc-pseudo} to appropriate examples, we obtain the controlling algebras for crossed homomorphisms, derivations, and Lie pseudoalgebra homomorphisms.
\begin{cor}
Consider the quasi-twilled Lie pseudoalgebra $(\mathfrak{g} \ltimes_\rho \mathfrak{h}, \mathfrak{g}, \mathfrak{h})$ given in Example \ref{actionLie} obtained from the action Lie pseudoalgebra $\mathfrak{g} \ltimes_\rho \mathfrak{h}$. Then $\left(C^n(\mathfrak{g}, \mathfrak{h}), d, \{\cdot, \cdot\}\right)$ is a differential graded Lie algebra, where the differential $d$ is given by
\begin{align}\label{differential}
 d(f)(x_1\otimes \dots\otimes x_{p+1}) &= \sum_{i=1}^{p+1} (-1)^{p+i}(\sigma_{1,i}\otimes_Hid) \rho(x_i, f(x_1\otimes \dots\otimes \hat{x_i}, \dots\otimes x_{p+1}))\\&\qquad+ \sum_{i<j} (-1)^{p+i+j-1}(\sigma_{i,j}\otimes_Hid) f([{x_i}* {x_j}]_{\mathfrak{g}}\otimes x_1\otimes\dots\otimes \hat{x_i}, \dots\otimes \hat{x_j}\otimes \dots, x_{p+1}),\nonumber
\end{align}
for all $f \in C^p( \mathfrak{g}, \mathfrak{h})$, and the graded Lie bracket $\{\cdot, \cdot\}$ is given by
\begin{align}\label{gradliebra}
 \{f, g\}(x_1, \dots, x_{p+q})&= \sum_{\sigma \in S(p,q)} (-1)^{pq+1} (-1)^\sigma p(\sigma_{i,j}\otimes_Hid) {[{{f}(x_{\sigma(1)}, \dots, x_{\sigma(p)})}* {g}(x_{\sigma(p+1)}, \dots, x_{\sigma(p+q)})]}_{\mathfrak{h}}, 
\end{align}
for all $f \in C^p( \mathfrak{g}, \mathfrak{h})$, $g \in C^q( \mathfrak{g}, \mathfrak{h})$. Additionally, Maurer-Cartan elements of this differential-graded Lie algebra are exactly crossed homomorphisms of weight $p$ from the Lie pseudoalgebra $(\mathfrak{g}, [\cdot*\cdot]_{\mathfrak{g}})$ to the Lie pseudoalgebra $(\mathfrak{h}, [\cdot* \cdot]_{\mathfrak{h}})$. That is, an element $f\in C(\mathfrak{g},\mathfrak{h})$ satisfies the Maurer–Cartan equation $d(f)+\frac{1}{2}\{f, f\}=0.$ We therefore say that this differential graded Lie algebra governs or controls the structure of crossed homomorphisms, and we refer to it as the  controlling algebra of the crossed homomorphism.
\end{cor}
\begin{prop}
Consider the quasi-twilled Lie pseudoalgebra $ (\mathfrak{g} \ltimes_\rho M, \mathfrak{g}, M) $ given in Example \ref{SEMDIRPRODLIEALG}, obtained from the semidirect product Lie pseudoalgebra $ \mathfrak{g} \ltimes_\rho M $. Then, 
$\left( \bigoplus_{n=1}^\infty C^n(\mathfrak{g}, M), d \right)$ is a cochain complex, where the differential $d$ is defined by Eq. \eqref{differential}. An $H$-linear map $ D \in C^1(\mathfrak{g}, M) $ is a derivation if and only if
$d(D) = 0$. Therefore, we say this cochain complex is controlling algebra of derivations from the Lie pseudoalgebra $ (\mathfrak{g}, [\cdot,\cdot]_\mathfrak{g}) $ to its representation space $( M, \rho)$.
\end{prop}
\begin{prop}
 Consider the quasi-twilled Lie pseudoalgebra $( \mathfrak{g}\boxplus \mathfrak{h}, \mathfrak{g}, \mathfrak{h})$ introduced in Example~\ref{directprod}, which arises from the direct product Lie pseudoalgebra $ \mathfrak{g} \boxplus \mathfrak{h} $. Then,
 $ \left(\bigoplus_{n=1}^\infty C^n( \mathfrak{g}, \mathfrak{h}), d, \{\cdot, \cdot\} \right)$
carries the structure of a differential graded Lie algebra (dgLa). The graded bracket $\{\cdot, \cdot\}$ is defined by Eq. \eqref{gradliebra}, and the differential $d$ is explicitly given by:
\begin{align*} 
 d(f)(x_1\otimes \dots\otimes x_{p+1}) &= \sum_{i<j} (-1)^{p+i+j-1} (\sigma_{i,j}\otimes_Hid)f([{x_i}* {x_j}]_{\mathfrak{g}}\otimes x_1\otimes \dots\otimes \hat{x_i}\otimes \dots\otimes \hat{x_j}\otimes \dots\otimes x_{p+1}),
\end{align*}
for all $ f \in C^p(\mathfrak{g}, \mathfrak{h})$, and $x_1, \dots, x_{p+1} \in \mathfrak{g}$.
This differential graded Lie algebra governs the space of Lie pseudoalgebra morphisms from $\mathfrak{g}$ to $\mathfrak{h}$. In particular, its Maurer–Cartan elements correspond precisely to such morphisms. Therefore, we refer to this dgLa as the controlling algebra for the Lie pseudoalgebra homomorphism. %{\color{red}see \cite{Fregier, N2} for more details.}{\color{blue} cited results are for the Lie algebra case, no corresponding results for psedoalgebra have been developed in the past. }
\end{prop} 
\noindent Let $ D : \mathfrak{g} \to \mathfrak{h} $ be a $\mathfrak{D}$-map of a quasi-twilled Lie pseudoalgebra $ (\mathfrak{G}, \mathfrak{g}, \mathfrak{h}) $. By Theorem \ref{thm:pseudo-mc-pseudo}, we obtain that $ D $ is a Maurer-Cartan element of the
$\left( \bigoplus_{n=0}^\infty C^{n+1}( \mathfrak{g}, \mathfrak{h}), l_0, l_1, l_2 \right). $
The twisted $ \mathfrak{L}_\infty $-pseudoalgebra structure on $ \bigoplus_{n=0}^\infty C^{n+1}( \mathfrak{g}, \mathfrak{h}) $ is given by
$$l_1^D(f) = l_1(f) + l_2(D, f),\qquad\qquad
l_2^D(f, g) = l_2(f, g), \qquad\qquad
l_{k \geq 3}^D = 0, $$
where $ f \in C^p( \mathfrak{g}, \mathfrak{h}) $ and $ g \in C^q( \mathfrak{g}, \mathfrak{h})$.

\noindent Now we are ready to give the $\mathfrak{L}_\infty$-pseudoalgebra structure that controls the deformations of the $\mathfrak{D}$-map~$D$.
\begin{thm}\label{thm:3.26}
Let $ D$ be a $\mathfrak{D}$-map of $ (\mathfrak{G}, \mathfrak{g}, \mathfrak{h}) $. Then for any left $H$-module homomorphism $ D':\mathfrak{g} \to \mathfrak{h} $, the sum $ D + D' $ is a $\mathfrak{D}$-map if and only if $ D' $ is a Maurer-Cartan element of the twisted $ \mathfrak{L}_\infty $-pseudoalgebra $\left( \bigoplus_{n=0}^\infty C^{n+1}(\mathfrak{g}, \mathfrak{h}), l_1^D, l_2^D \right),$ 
i.e., $ D' $ satisfies the Maurer-Cartan equation 
\begin{align*}
 l_1^D(D') + \frac{1}{2} l_2^D(D', D') = 0.
\end{align*}
\end{thm}
\begin{proof}
By Theorem \ref{thm:pseudo-mc-pseudo}, $ D + D' $ is a $\mathfrak{D}$-map if and only if it satisfies 
$$
l_0 + l_1(D + D') + \frac{1}{2} l_2(D + D', D + D') = 0.
$$
Expanding this expression yields 
$$
l_0 + l_1(D) + l_1(D') + \frac{1}{2} l_2(D, D) + l_2(D, D') + \frac{1}{2} l_2(D', D') = 0.
$$
Since $ D $ is itself a Maurer-Cartan element, satisfying 
 $
l_0 + l_1(D) + \frac{1}{2} l_2(D, D) = 0.$
Thus, we get
$$
l_1(D') + l_2(D, D') + \frac{1}{2} l_2(D', D') = 0.
$$
Using the twisted operations, it can be rewritten as
$$l_1^D(D') + \frac{1}{2} l_2^D(D', D') = 0.$$ 
This shows that $ D' $ is a Maurer-Cartan element of the twisted $ \mathfrak{L}_\infty $-pseudoalgebra 
$ \left( \bigoplus_{n=0}^\infty C^{n+1}(\mathfrak{g}, \mathfrak{h}), l_1^D, l_2^D \right).$ Thus, the \emph{deformation theory of $\mathfrak{D}$-maps} is governed by the twisted $ \mathfrak{L}_\infty $-pseudoalgebra structure. \end{proof}
 
%\begin{rem} \end{rem}
\subsection{ Cohomology of $\mathfrak{D}$ maps of type I}

In this subsection, we introduce a cohomology theory for $\mathfrak{D}$-maps on quasi-twilled Lie pseudoalgebras. This cohomology unifies the cohomological frameworks of various operator structures, including Rota-Baxter operators, crossed homomorphisms, derivations, and Lie pseudoalgebra homomorphism.
\begin{lem}\label{lem:3.28}
Let $ (\mathfrak{G}, \mathfrak{g}, \mathfrak{h}) $ be a quasi-twilled Lie pseudoalgebra, and let $D: \mathfrak{g}\to \mathfrak{h}$ be a $\mathfrak{D}$-map. Then with $\pi^D $ and $\rho^D$ as defined in Theorem \ref{thm33}, the pair $ (\mathfrak{g}, \pi^D) $ forms a Lie pseudoalgebra and $ \rho^D$ defines a representation of this Lie pseudoalgebra on $ \mathfrak{h}$.
 \end{lem}
\begin{proof}
Since $ \big((\mathfrak{g}, \pi^D), (\mathfrak{h}, \mu^D); \rho^D, \eta^D \big)$ forms a matched pair of Lie pseudoalgebras (as shown in Theorem \ref{thm33}), it follows that $ (\mathfrak{g},\pi^D) $ is a Lie pseudoalgebra and $ \rho^D$ gives a representation on $ \mathfrak{h}$ over $ \mathfrak{g}$. This completes the proof.
\end{proof}
\noindent In the following, we define the Chevalley-Eilenberg coboundary operator associated with the representation $(\mathfrak{h}, \rho^D)$ by the map $d^D_{CE}: C^p(\mathfrak{g}, \mathfrak{h}) \to C^{p+1}(\mathfrak{g}, \mathfrak{h})$:
\begin{align}\label{CED}
 &d^D_{CE}(f)(x_1\otimes \dots\otimes x_{p+1}) \\&= \sum_{i=1}^{p+1} (-1)^{p+i}(\sigma_{1,i}\otimes_Hid) \rho^D(x_i, f(x_1\otimes \dots\otimes \hat{x_i}\otimes \dots\otimes x_{p+1})) \nonumber\\&+ \sum_{i<j} (-1)^{p+i+j-1} (\sigma_{i,j}\otimes_Hid)f([ \pi^D({x_i}\otimes {x_j})\otimes x_1\otimes \dots\otimes \hat{x_i}\otimes \dots\otimes \hat{x_j}\otimes \dots\otimes x_{p+1}),\nonumber\end{align}
 where $\sigma_{1,i}(h_1\otimes \cdots \otimes h_i\otimes \cdots\otimes h_n)=h_i\otimes h_1\otimes\cdots\otimes h_{i-1}\otimes h_{i+1}\otimes\cdots\otimes h_n $ and $\sigma_{i,j}(h_1\otimes \cdots\otimes h_n)=h_i\otimes h_j\otimes h_1\otimes \cdots\otimes h_{i-1}\otimes h_{i+1}\otimes\cdots\otimes h_{j-1}\otimes h_{j+1}\otimes\cdots\otimes h_n$ for all $h_k\in H.$

\noindent Expanding $ \rho^D$ and $ \pi^D$, we get:
\begin{align*} &d^D_{CE}(f)(x_1\otimes \dots\otimes x_{p+1}) \\&= 
\sum_{i=1}^{p+1} (-1)^{p+i} (\sigma_{1,i}\otimes_Hid)\rho(x_i, f(x_1\otimes \dots\otimes\hat{x_i}\otimes \dots\otimes x_{p+1})) \\&+
\sum_{i=1}^{p+1}(-1)^{p+i}(\sigma_{1,i}\otimes_Hid) \mu\big(D(x_i), f(x_1\otimes \dots\otimes \hat{x_i}\otimes \dots\otimes x_{p+1})\big)\\& 
- \sum_{i=1}^{p+1} (-1)^{p+i}(\sigma_{1,i}\otimes_Hid) D\big(\eta(x_i,f(x_1\otimes \dots\otimes \hat{x_i}\otimes \dots\otimes x_{p+1}))\big)
\\&+ \sum_{i<j} (-1)^{p+i+j-1}(\sigma_{i,j}\otimes_Hid) f(\pi({x_i}, {x_j}), x_1\otimes \dots\otimes \hat{x_i}\otimes \dots\otimes \hat{x_j}\otimes \dots\otimes x_{p+1}\big)\\&
+\sum_{i<j} (-1)^{p+i+j-1}(\sigma_{i,j}\otimes_Hid)f\big({\eta(x_i,D(x_j))} -((12)\otimes_Hid){\eta(x_j, D(x_i))} \otimes x_1\otimes \dots\otimes \hat{x_i}\otimes\dots\\&\qquad\qquad\qquad\qquad\qquad\qquad\qquad\qquad\qquad\qquad\qquad\qquad\qquad\qquad\qquad\qquad
\otimes \hat{x_j}\otimes \dots\otimes x_{p+1}\big)
\end{align*} This leads to the define the cohomology governing $\mathfrak{D}$-maps. Let $ D : \mathfrak{g}\to \mathfrak{h}$ be a $\mathfrak{D}$-map of type I of a quasi-twilled Lie pseudoalgebra $ (\mathfrak{G}, \mathfrak{g}, \mathfrak{h})$. The space of $ n $-cochains $ C^n(D) $ associated with a $\mathfrak{D}$-map is given as:\begin{align*}
 C^n(D) = 
\begin{cases}
0, & \text{for } n = 0, \\
C^0(\mathfrak{g}, \mathfrak{h}) = \mathfrak{h}, & \text{for } n = 1, \\
C^{n-1}(\mathfrak{g}, \mathfrak{h}), & \text{for } n \geq 2.
\end{cases}
\end{align*}
The cohomology of $\mathfrak{D}$-map is then defined as the cohomology of the cochain complex $$\left( C^*(D)=\bigoplus_{n=0}^\infty C^n(D), d^D_{CE}\right).$$ Let $ H^n(D) = \frac{Z^n(D)} {B^n(D)}$ is the $n$-th cohomology group. Here $ Z^n(D) $ denotes the space of $ n $-cocycles, i.e., having the elements $f\in C^n(D)$ such that $d^D_{CE}(f)=0$. And $ B^n(D) $ is the space of $ n $-coboundaries, i.e. having the elements $f\in C^n(D)$ such that $d^D_{CE}(f)= g$ for any $g\in C^{n-1}(D)$. The space of cochains, that satisfy cocycle condition is called \emph{closed} and that satisfy coboundary condition is called \emph{exact}. 
Furthermore, we give the explicit description of $1$-cocycle and $2$-cocycle in the Proposition \ref{prop12cocycle} 
\begin{prop}\label{prop12cocycle}
Let $ D : \mathfrak{g} \to \mathfrak{h} $ be a $\mathfrak{D}$-map of a quasi-twilled pseudoalgebra $ (\mathfrak{G}, \mathfrak{g}, \mathfrak{h}) $. Then 

\begin{enumerate}
 \item An element $ u \in \mathfrak{h} = C^1(D) $ is closed (i.e., a $1$-cocycle) if and only if it satisfies the following condition for all $ x \in \mathfrak{g} $ 
 \begin{align*}
\rho(x \otimes u) + \mu(D(x) \otimes u) - D(\eta(x \otimes u)) = 0 .
 \end{align*}

 \item A cochain $ f \in C^1(\mathfrak{g}, \mathfrak{h}) = C^2(D) $ is closed (i.e., a $2$-cocycle) if and only if   for all $ x, y \in \mathfrak{g} $, we have
 \begin{align*}
 &\rho(x \otimes f(y)) - ((12) \otimes_H \mathrm{id})\,\rho(y \otimes f(x)) + \mu(D(x) \otimes f(y)) \\
 &- ((12) \otimes_H \mathrm{id})\,\mu(D(y) \otimes f(x)) + D(\eta(y \otimes f(x))) - ((12) \otimes_H \mathrm{id})\,D(\eta(x \otimes f(y))) \\
 &- f\big( \pi(x \otimes y) + \eta(x \otimes D(y)) - ((12) \otimes_H \mathrm{id})\,\eta(y \otimes D(x)) \big) = 0.
 \end{align*}

\end{enumerate}
\end{prop}
We now give an intrinsic interpretation of this cohomology via the curved $ \mathfrak{L}_\infty$-pseudoalgebra introduced in previous subsections.
\begin{prop}\label{prop:3.24}
With the above notation, for any cochain $ f \in C^p(\mathfrak{g}, \mathfrak{h}) $, the following identity holds 
\begin{align*}
 l^D_1(f) = (-1)^{p-1} d^D_{CE}(f).
\end{align*}
\end{prop}
\begin{proof}
Using the definitions from Theorem \ref{thm:3.26}, we compute 
\begin{align*}
l^D_1(f) &= l_1(f) + l_2(D, f) \\
&= [\pi + \rho, f]_{\mathrm{NR}} + [[\mu + \eta, D]_{\mathrm{NR}}, f]_{\mathrm{NR}} \\
&= [\pi^D+ \rho^D,f]_{\mathrm{NR}} \\
&= (-1)^{p-1} d^D_{CE}(f).
\end{align*}
This completes the proof.
\end{proof}
\begin{rem} From the Proposition \ref{prop:3.24}, we observe that $\pi^D =\pi +[\eta, D]_{NR}$ and $\rho^D = \rho +[\mu, D]_{NR}$.
\end{rem}
Finally, we provide several important examples where this cohomology of $\mathfrak{D}$-maps applies to well-known operators.
\begin{ex} 
Consider the quasi-twilled Lie pseudoalgebra $ (\mathfrak{g}\boxplus \mathfrak{g}, \mathfrak{g}, \mathfrak{g}) $ from Example \ref{ggg}. Let $ D : \mathfrak{g}\to \mathfrak{g}$ be a modified $ r $-matrix. Then the induced pseudobracket on $ \mathfrak{g}$ is defined by
\begin{align*}
 [x* y]^D =[x* D(y)] - ((12)\otimes_H id)[y * D(x)], \quad \forall x, y \in \mathfrak{g}.
\end{align*}
Moreover, the corresponding pseudo representation $ \rho^D : \mathfrak{g}\otimes\mathfrak{g}\to H^{\otimes 2}\otimes_H \mathfrak{g} $ is given by 
\begin{align*}
 \rho^D(x\otimes y) = [D(x)* y] - D([x* y]), \quad \forall x,y \in \mathfrak{g}.
\end{align*}
The Chevalley-Eilenberg cohomology of this representation is taken to be the \emph{cohomology of the modified $ r $-matrix}.
\end{ex}
\begin{ex}
Let $ (\mathfrak{g}\ltimes_\rho \mathfrak{h}, \mathfrak{g}, \mathfrak{h}) $ be the action Lie pseudoalgebra from Example \ref{actionLie}, and let $ D : \mathfrak{g}\to \mathfrak{h} $ be a crossed homomorphism of weight $ p $. Then the pseudo representation $ \rho^D : \mathfrak{g}\otimes\mathfrak{h} \to H^{\otimes 2}\otimes_H \mathfrak{h} $ is given by 
\begin{align*}
 \rho^D(x\otimes u) = \rho(x\otimes u) + p \mu( D(x)\otimes u), \quad \forall u \in \mathfrak{h}, x \in \mathfrak{g}.
\end{align*}
The Chevalley-Eilenberg cohomology of $ \mathfrak{g}$ with coefficients in $ \mathfrak{h}$ is taken to be the \emph{cohomology of the crossed homomorphism of weight $p$}, see \cite{ZQX} for more details.%{\color{blue} the concept of crossed homomorphism of pseudoalgebas has not been yet developed. {\color{red}We can define it}}
\end{ex}\begin{ex} Let $ (\mathfrak{g}\ltimes_\rho M, \mathfrak{g}, M) $ be the quasi-Twilled Lie pseudoalgebra corresponding to semi-direct product Lie pseudoalgebra from Example \ref{SEMDIRPRODLIEALG}. Consider a left $H$-modules homomorphism $D : \mathfrak{g}\to M $ is a derivation from the Lie pseudoalgebra $(\mathfrak{g}, [{\cdot}*{\cdot}]_\mathfrak{g})$ with respect to its representation $(M, \rho)$. Then the Chevalley-Eilenberg cohomology of $ (\mathfrak{g},[\cdot*\cdot])$ with coefficients in representation $(M, \rho)$ is taken to be the \emph{cohomology of the derivation.} See \cite{TFS} for more details. 
\end{ex}
\begin{ex} 
Let $ (\mathfrak{g}\boxplus \mathfrak{h}, \mathfrak{g}, \mathfrak{h}) $ be quasi twilled Lie pseudoalgebra corresponding to the direct product Lie pseudoalgebra from Example \ref{directprod}. If $ D :\mathfrak{g}\to \mathfrak{h}$ is the homomorphism of Lie pseudoalgebras, then the representation $ \rho^D: \mathfrak{g}\otimes\mathfrak{h}\to H^{\otimes2}\otimes_H\mathfrak{h}$ is defined by 
\begin{align*}
 \rho^D(x\otimes u) = p\mu (D(x)\otimes u), \quad \forall u \in \mathfrak{h}, x \in \mathfrak{g}.
\end{align*}
The corresponding Chevalley-Eilenberg cohomology is taken to be the \emph{cohomology of the Lie pseudoalgebra homomorphism}, see \cite{BDK}.%{\color{red} (see \cite{Fregier}). }{\color{blue} the concept of cohomology of Lie pseudoalgebra homomorphism has not been yet developed ({\color{red}see\cite{BDK})}.}
\end{ex}

\section{ The Controlling Algebras and Cohomologies of Deformation Maps of Type II}

In this section, $ (\mathfrak{G}, \mathfrak{g}, \mathfrak{h}) $ denotes a quasi-twilled pseudoalgebra over a cocommutative Hopf algebra $ H$, as defined in Section $2$. The pseudobracket on $ \mathfrak{G}$ is denoted by 
$ 
\Omega = \pi + \rho + \mu + \eta + \theta,
$ 
where the role of $\pi , \rho , \mu , \eta ,$ and $\theta $ is clear from the context.
%\begin{itemize}\item $ \pi $: pseudobracket on $ \mathfrak{g} \otimes \mathfrak{g} $,\item $ \mu $: pseudobracket on $ \mathfrak{h} \otimes \mathfrak{h} $, \item $ \rho $: action of $ \mathfrak{g} $ on $ \mathfrak{h} $,\item $ \eta $: action of $ \mathfrak{h} $ on $ \mathfrak{g} $,\item $ \theta $: curvature term capturing higher-order interactions between $ \mathfrak{g} $ and $ \mathfrak{h} $.\end{itemize}

%This decomposition follows the structure introduced in Section $2$, and the bracket $ \Omega $ satisfies the Jacobi identity in the pseudotensor category $ \mathcal{M}^*(H) $.

\subsection{ Type II Deformation Maps and Operator Unification}

We now introduce the notion of a deformation map of type II (or a $\mathcal{D}$-map) in the context of quasi-twilled pseudoalgebras. These maps unify operator-driven structures, such as Relative Rota-Baxter operators, Twisted Rota-Baxter operators, Reynolds-type operators, and Deformation maps of matched pairs of pseudoalgebras.

\begin{defn}
Let $ (\mathfrak{G}, \mathfrak{g}, \mathfrak{h}) $ be a quasi-twilled pseudoalgebra over a cocommutative Hopf algebra $ H $. A \emph{deformation map of type II} (or simply a $\mathcal{D}$-map) is an $ H $-module homomorphism $T : \mathfrak{h} \to \mathfrak{g},$
satisfying the following identity for all $ u, v \in \mathfrak{h} $:
\begin{align*}
\pi(T(u) \otimes T(v)) + \eta(T(u) \otimes v) - ((12)\otimes_Hid)\eta(T(v) \otimes u) &= T\Big( \rho(T(u) \otimes v) - ((12)\otimes_Hid)\rho(T(v) \otimes u) \\&\qquad\qquad+ \mu(u \otimes v) + \theta(T(u) \otimes T(v)) \Big).\nonumber
\end{align*}The term $((12)\otimes_Hid)$ denotes the transposition action of the symmetric group $S_2$ on the pseudotensor components.
\end{defn}
\begin{prop}
Let $ (\mathfrak{G}, \mathfrak{g}, \mathfrak{h}) $ be a quasi-twilled pseudoalgebra over a cocommutative Hopf algebra $ H $, and let $ D : \mathfrak{g} \to \mathfrak{h} $ be an invertible $ H $-module homomorphism. Then $ D $ is a $\mathfrak{D}$-map if and only if its inverse $ T = D^{-1} : \mathfrak{h} \to \mathfrak{g} $ is a $\mathcal{D}$-map.
In other words, the following are equivalent 
\begin{enumerate}
 \item[1.] $ D \in \mathrm{Hom}_H(\mathfrak{g}, \mathfrak{h}) $ satisfies 
 \begin{align*}
 &D\Big( \pi(x \otimes y) + \eta(x \otimes D(y)) - ((12) \otimes_H \mathrm{id})\,\eta(y \otimes D(x)) \Big) \\&= \mu(D(x) \otimes D(y)) + \rho(x \otimes D(y)) - ((12) \otimes_H \mathrm{id})\,\rho(y \otimes D(x)) + \theta(x \otimes y),
 \end{align*}
 for all $ x, y \in \mathfrak{g} $, i.e., $ D $ is a $\mathfrak{D}$-map.
\item[2.] $ T = D^{-1} \in \mathrm{Hom}_H(\mathfrak{h}, \mathfrak{g}) $ satisfies 
 \begin{align*}
 &\pi(T(u) \otimes T(v)) + \eta(T(u) \otimes v) - ((12) \otimes_H \mathrm{id})\,\eta(T(v) \otimes u) \\&= T\Big( \rho(T(u) \otimes v) - ((12) \otimes_H \mathrm{id})\,\rho(T(v) \otimes u) + \mu(u \otimes v) + \theta(T(u) \otimes T(v)) \Big),
 \end{align*}
 for all $ u, v \in \mathfrak{h} $, i.e., $ T $ is a $\mathcal{D}$-map.
\end{enumerate} 
%This result establishes a one-to-one correspondence between deformation maps of type I and type II under the assumption that $ D $ is invertible.
\end{prop}
\begin{ex}\label{ex:rel-RB-op-pseudo}
Let $ (\mathfrak{G}, \mathfrak{g}, \mathfrak{h}) = (\mathfrak{g} \ltimes_\rho \mathfrak{h}, \mathfrak{g}, \mathfrak{h}) $ be the quasi-twilled pseudoalgebra associated to the matched pair structure on $ \mathfrak{g} $ and $ \mathfrak{h} $, where $ \mathfrak{g} $ acts on $ \mathfrak{h} $ via $ \rho : \mathfrak{g} \otimes \mathfrak{h} \to H^{\otimes2} \otimes_H \mathfrak{h} $. Then a $\mathcal{D}$-map of type II is an $ H $-module homomorphism $ T : \mathfrak{h} \to \mathfrak{g} $ satisfying 
\begin{align*}
 &\pi(T(u) \otimes T(v)) + \eta(T(u) \otimes v) - ((12) \otimes_H \mathrm{id})\,\eta(T(v) \otimes u) \\&= T\Big( \rho(T(u) \otimes v) - ((12) \otimes_H \mathrm{id})\,\rho(T(v) \otimes u) + p \cdot \mu(u \otimes v) \Big),
\end{align*}
for all $ u, v \in \mathfrak{h} $, where $ p \in \mathbf{k} $.
This identity shows that $ T $ is a \emph{relative Rota-Baxter operator of weight $ p $} on $ \mathfrak{h} $ with respect to the representation $ \rho $ of $ \mathfrak{g} $ in the pseudotensor category $ \mathcal{M}^*(H) $. When $ p = 0 $, it reduces to a relative Rota-Baxter operator of weight zero; when $ p = 1 $, it corresponds to a standard Rota-Baxter operator.
\end{ex}

\begin{ex}\label{ex:O-operator-pseudo}
Consider the quasi-twilled pseudoalgebra $ (\mathfrak{G}, \mathfrak{g}, M) = (\mathfrak{g} \ltimes_\rho M, \mathfrak{g}, M) $, obtained from the semidirect product structure defined by an action $ \rho : \mathfrak{g} \otimes M \to H \otimes_H M $. In this case, a $\mathcal{D}$-map of type II is an $ H $-module homomorphism $ T : M \to \mathfrak{g} $ such that:
\begin{align*}
 &\pi(T(u) \otimes T(v)) + \eta(T(u) \otimes v) - ((12) \otimes_H \mathrm{id})\,\eta(T(v) \otimes u) \\&= T\Big( \rho(T(u) \otimes v) - ((12) \otimes_H \mathrm{id})\,\rho(T(v) \otimes u) \Big),
\end{align*}
for all $ u, v \in M $. This condition implies that $ T $ is a \emph{relative Rota-Baxter operator of weight $ 0 $} on $ \mathfrak{g} $ with respect to the module $ M $. Such an operator is also known as an \emph{$\mathcal{O}$-operator} in the pseudoalgebraic setting. These operators play a fundamental role in the construction of pre-Lie pseudoalgebras and integrable systems within the framework of pseudotensor categories (see \cite{AWMB}).
\end{ex}
\begin{ex}\label{exTRB}Consider the quasi-twilled pseudoalgebra $(\mathfrak{g} \ltimes_\rho^\omega M, \mathfrak{g}, M)$ given in Example \ref{excocycle}. In this case, a $\mathcal{D}$-map of $(\mathfrak{g} \ltimes_\rho^\omega M, \mathfrak{g}, M)$ is a
$H$-module homomorphism $T : M\to \mathfrak{g} $ such that
\begin{align*}
 {[T(u)* T(v)]}_\mathfrak{g} = T\big({\rho(T(u))}\otimes v - ((12) \otimes_H \mathrm{id}) {\rho(T(v))}\otimes u + \omega (T(u)\otimes T(v))_\mathfrak{g}\big), \quad \forall u, v \in \mathfrak{g}.
\end{align*} which implies that $T$ is a \emph{twisted Rota-Baxter operator} defined on the pseudoalgebra \cite{AYW}.
\end{ex}
\begin{ex}\label{exRO}
Consider the quasi-twilled pseudoalgebra $ (\mathfrak{g} \ltimes_{ad}^\omega \mathfrak{g}, \mathfrak{g}, \mathfrak{g}) $ given in Example \ref{exex}. In this case, a $\mathcal{D}$-map of $ (\mathfrak{g} \ltimes_{ad}^\omega \mathfrak{g}, \mathfrak{g}, \mathfrak{g}) $ is a $H$-module homomorphism $ T : \mathfrak{g} \to \mathfrak{g} $ such that
 \begin{align*}
{[{T(u)}* T(v)]}_\mathfrak{g} = T\big({[{T(u)}* v]}_\mathfrak{g} +{[{u}* T(v)]}_\mathfrak{g} + {[{T(u)}* T(v)]}_\mathfrak{g}\big), \quad \forall u, v \in \mathfrak{g}.
 \end{align*}
This identity implies that $T$ is a \emph{Reynolds operator} on $\mathfrak{g}$ (see \cite{YL}).
\end{ex}
\begin{ex}\label{ex:4.7}
Consider the quasi-twilled Lie pseudoalgebra $ (\mathfrak{g} \bowtie \mathfrak{h}, \mathfrak{g}, \mathfrak{h}) $ given in Example \ref{exmax}, obtained from a matched pair of pseudoalgebras. In this case, a $\mathcal{D}$-map of $ (\mathfrak{g} \bowtie \mathfrak{h}, \mathfrak{g}, \mathfrak{h}) $ is a $H$-modules homomorphism $ T : \mathfrak{h} \to \mathfrak{g} $ such that
\begin{align*}
 &{[T(u)* T(v)]}_\mathfrak{g} + \eta (u\otimes T(v)) - ((12) \otimes_H \mathrm{id}) \eta (v\otimes T(u))\\&
 = T\big({[u* v]}_\mathfrak{h} + \rho (T(u)\otimes v) -((12) \otimes_H \mathrm{id}) \rho (T(v)\otimes u)\big),
\end{align*} for all $u, v \in \mathfrak{h}$. This is indeed a \emph{deformation map of a matched pair of Lie pseudoalgebras} introduced in \cite{H}.
\end{ex}

At the end of this subsection, we illustrate the roles that $\mathcal{D}$-maps of type II play in the twisting theory.

\noindent Let $T : \mathfrak{h} \to \mathfrak{g}$ be a left $H$-module homomorphism. Suppose that $T^2 = 0$ and $X_T=[\cdot, T]_{\mathrm{NR}}$ is a derivation of the graded Lie pseudoalgebra
$\left(\bigoplus_{n=0}^{\infty} C^{n+1}(\mathfrak{g} \boxplus \mathfrak{h}, \mathfrak{g} \boxplus \mathfrak{h}), [\cdot, \cdot]_{\mathrm{NR}}\right).$ 
Then the exponential map $ e^{X_T}=e^{[\cdot, T]_{NR}}$ is an automorphism of this graded Lie pseudoalgebra. 

\begin{defn}\label{def:4.8} Consider a $\mathcal{D}$-map $T$.
The transformation $\O^T := e^{[ \cdot, T ]_{NR}} \O $ is called the \emph{twisting of $ \O $ by $ T $}.
\end{defn}
\noindent Parallel to Proposition \ref{proptransformation}, we have the following result.
\begin{lem}\label{lem:4.9}
With the above notations, $ \O^T = e^{-T} \circ \O \circ (e^{T} \otimes e^{T}) $ defines a Lie pseudoalgebra structure on $ \mathfrak{G}=\mathfrak{g}\boxplus \mathfrak{h} $.
\end{lem}
\noindent Further, $ \O^T$ admits a unique decomposition into six distinct components:
$ \O^T = \pi^T + \rho^T+ \mu^T + \eta^T + \theta^T + \xi^T, $ 
where $\pi, \rho, \mu, \eta, \theta$ have been defined in earlier sections, and $\xi^T \in C^2(\mathfrak{h}, \mathfrak{g})$ is a newly introduced cochain.
We now state the main result of this subsection.

\begin{thm}\label{thm:4.10-pseudo}
Let $ \Omega = \pi + \rho + \mu + \eta + \theta $ be the pseudobracket of a quasi-twilled pseudoalgebra $ (\mathfrak{G}, \mathfrak{g}, \mathfrak{h}) $. Define its $ T $-twisted counterpart $ \Omega^T $ as above. Then the components of $ \Omega^T $ are explicitly given by:
\begin{align*}
\pi^T(x \otimes y) &= \pi(x \otimes y) - T(\theta(x \otimes y)), \\
\rho^T(x \otimes v) &= \rho(x \otimes v) - \theta(T(v) \otimes x), \\
\mu^T(u \otimes v) &= \mu(u \otimes v) + \rho(T(u) \otimes v) - ((12) \otimes_H \mathrm{id})\,\rho(T(v) \otimes u) + \theta(T(u) \otimes T(v)), \\
\eta^T(x \otimes v) &= \eta(x \otimes v) - \pi(T(v) \otimes x) - T(\rho(x \otimes v)) + T(((12) \otimes_H \mathrm{id})\,\theta(T(v) \otimes x)), \\
\theta^T(x \otimes y) &= \theta(x \otimes y), \\
\xi^T(u \otimes v) &= \pi(T(u) \otimes T(v)) + \eta(T(u) \otimes v) - ((12) \otimes_H \mathrm{id})\,\eta(T(v) \otimes u) \\
&\quad - T\Big( \mu(u \otimes v) + \rho(T(u) \otimes v) - ((12) \otimes_H \mathrm{id})\,\rho(T(v) \otimes u) - ((12) \otimes_H \mathrm{id})\,\theta(T(v) \otimes T(u)) \Big),
\end{align*}
for all $ x, y \in \mathfrak{g} $, $ u, v \in \mathfrak{h} $, where $ ((12) \otimes_H \mathrm{id}) $ denotes the action of the transposition $ (12) \in S_2 $ on the tensor product in $ \mathcal{M}^*(H) $.
This decomposition implies that an $ H $-module homomorphism $ T : \mathfrak{h} \to \mathfrak{g} $ is a $\mathcal{D}$-map of type II if and only if $ (\mathfrak{G}, \Omega^T) $ defines a quasi-twilled pseudoalgebra structure on $ \mathfrak{G} $.
\end{thm}
\begin{proof}
This follows from a direct but tedious computation using the properties of the Nijenhuis-Richardson bracket and the decomposition of $\O $. The details involve verifying that the twisted components satisfy the axioms of a quasi-twilled Lie pseudoalgebra. We omit the computation for brevity. \end{proof}
\subsection{ The Controlling Algebra of $\mathcal{D}$-Maps of type II} In this subsection, we describe the controlling algebra of deformation maps of type II, which takes the form of an $ \mathfrak{L}_\infty$-pseudoalgebra. A significant byproduct of this construction is the corresponding controlling algebra for deformation maps of matched pairs of Lie pseudoalgebras. %{\color{blue} matched pairs of Lie pseudoalgebras is not defined before {\color{red} we define it in this paper}}
\begin{thm}\label{thm:4.11-pseudo}
Let $ (\mathfrak{G}, \mathfrak{g}, \mathfrak{h}) $ be a quasi-twilled pseudoalgebra over a cocommutative Hopf algebra $ H $. Then there exists a curved $ V $-data $ (\mathfrak{J}, \mathfrak{K}, P, \Delta) $, where $ \mathfrak{J} = \bigoplus_{n \geq 0} C^{n+1}(\mathfrak{g} \boxplus \mathfrak{h}, \mathfrak{g} \boxplus \mathfrak{h}) $ is the graded $ H $-module of pseudotensor multi-linear maps, equipped with the Nijenhuis-Richardson bracket $ [\cdot, \cdot]_{\mathrm{NR}} $, and
 $ \mathfrak{K} = \bigoplus_{n \geq 0} C^{n+1}(\mathfrak{h}, \mathfrak{g}) $ is an abelian graded Lie pseudosubalgebra of $ \mathfrak{J} $. The map $ P : \mathfrak{J} \to \mathfrak{J} $ is the $ H $-linear projection onto $ \mathfrak{K} $,
 and $ \Delta = \pi + \rho + \mu + \eta + \theta \in \mathfrak{J}^1 $ satisfies the Maurer-Cartan condition $ [\Delta, \Delta]_{\mathrm{NR}} = 0 $. This $ V $-data induces a curved $ \mathfrak{L}_\infty $-pseudoalgebra structure on $ \mathfrak{K} $, with operations:
$$
l_1(f) = [\mu + \eta, f]_{\mathrm{NR}}, \quad
l_2(f, g) = [[\pi + \rho, f]_{\mathrm{NR}}, g]_{\mathrm{NR}}, \quad
l_3(f, g, h) = [[[ \theta, f ]_{\mathrm{NR}}, g ]_{\mathrm{NR}}, h ]_{\mathrm{NR}}, \quad\text{ and }l_{k \geq 4 } = 0 .
$$
\noindent Furthermore, an $ H $-module homomorphism $ T : \mathfrak{h} \to \mathfrak{g} $ is a $\mathcal{D}$-map of type II if and only if $ T $ is a Maurer-Cartan element of this curved $ \mathfrak{L}_\infty $-pseudoalgebra.
\end{thm}
\begin{proof}
By the general theory of $ V $-data (see \cite{Voronov}), a quadruple $ (\mathfrak{J}, \mathfrak{K}, P, \Delta) $ as above induces an $ \mathfrak{L}_\infty $-algebra structure on $ \mathfrak{K} $ via $$
l_k(f_1, \dots, f_k) = P([\cdots[[\Delta, f_1]_{\mathrm{NR}}, f_2]_{\mathrm{NR}}, \dots, f_k]_{\mathrm{NR}}).$$

Let's compute each operation explicitly. For $ f \in C^{k+1}(\mathfrak{h}, \mathfrak{g}) $, we get
$ l_1(f) = P([\pi + \rho + \mu + \eta + \theta, f]_{\mathrm{NR}}).
$ 
Since $ \mathfrak{K} $ consists of maps from $ \mathfrak{h}^{\otimes n} \to H^{\otimes(n-1)} \otimes_H \mathfrak{g} $, and $ P $ projects onto components landing in $ \mathfrak{g} $, only terms involving $ \mu $ and $ \eta $ contribute. Thus, $l_1(f) = [\mu + \eta, f]_{\mathrm{NR}}.$ Similarly, for $ f_1 \in C^{p+1}(\mathfrak{h}, \mathfrak{g}) $, $ f_2 \in C^{q+1}(\mathfrak{h}, \mathfrak{g}) $, we have 
$ 
l_2(f_1, f_2) = P([[\Delta, f_1]_{\mathrm{NR}}, f_2]_{\mathrm{NR}}).
$ 
The only contributions come from $ \pi $ and $ \rho $, that map $ \mathfrak{g} \to \mathfrak{h} $ or $ \mathfrak{h} \to \mathfrak{g} $ respectively. Thus, $ 
l_2(f_1, f_2) = [[\pi + \rho, f_1]_{\mathrm{NR}}, f_2]_{\mathrm{NR}}.
$ For
$l_3(f_1, f_2, f_3) = P([[[\Delta, f_1]_{\mathrm{NR}}, f_2]_{\mathrm{NR}}, f_3]_{\mathrm{NR}}),$
only $ \theta \in C^2(\mathfrak{g}, \mathfrak{h})$ survives after three brackets and projection to $ \mathfrak{K} $, providing
$ l_3(f_1, f_2, f_3) = [[[ \theta, f_1 ]_{\mathrm{NR}}, f_2 ]_{\mathrm{NR}}, f_3 ]_{\mathrm{NR}}.$ Since, $ \mathfrak{K} $ is abelian and the grading increases by one with each bracket, we find $ l_k = 0 $ for $ k \geq 4 $.

Now, let $ T \in C^1(\mathfrak{h}, \mathfrak{g}) $, i.e., $ T : \mathfrak{h} \to \mathfrak{g} $ is an $ H $-module homomorphism. Using Maurer-Cartan equation, we get
\begin{align*}
 &l_1(T) + \frac{1}{2} l_2(T, T) + \frac{1}{6} l_3(T, T, T) = 0.
\\
&[\mu + \eta, T]_{\mathrm{NR}} + \frac{1}{2} [[\pi + \rho, T]_{\mathrm{NR}}, T]_{\mathrm{NR}} + \frac{1}{6} [[[ \theta, T ]_{\mathrm{NR}}, T ]_{\mathrm{NR}}, T ]_{\mathrm{NR}} = 0.
\end{align*}
Expanding each component, we get
\begin{align*}
& -T(\mu(u \otimes v)) + \eta(T(u) \otimes v) - ((12) \otimes_H \mathrm{id})\,\eta(T(v) \otimes u) \\
& + \pi(T(u) \otimes T(v)) - T\big( \rho(T(u) \otimes v) - ((12) \otimes_H \mathrm{id})\,\rho(T(v) \otimes u) \big) \\
& + T(((12) \otimes_H \mathrm{id})\,\theta(T(v) \otimes T(u))) = 0,
\end{align*}
for all $ u, v \in \mathfrak{h} $. Precisely
\begin{align*}
 &\pi(T(u) \otimes T(v)) + \eta(T(u) \otimes v) - ((12) \otimes_H \mathrm{id})\,\eta(T(v) \otimes u) \\&= T\Big( \rho(T(u) \otimes v) - ((12) \otimes_H \mathrm{id})\,\rho(T(v) \otimes u) + \mu(u \otimes v) + \theta(T(u) \otimes T(v)) \Big).
\end{align*}
This is the condition for $ T $ to be a $\mathcal{D}$-map of type II. Therefore, $ T $ is a Maurer-Cartan element in the curved $ \mathfrak{L}_\infty $-pseudoalgebra $ (\mathfrak{K}, l_1, l_2, l_3) $ if and only if $ T $ is a $\mathcal{D}$-map. This completes the proof.
\end{proof} \noindent The general construction in Theorem~\ref{thm:4.11-pseudo} provides that, for any quasi-twilled Lie $ H $-pseudoalgebra $ (\mathfrak{G}, \mathfrak{g}, \mathfrak{h}) $, there is an associated curved $ \mathfrak{L}_\infty $-pseudoalgebra that governs the deformation theory of its $\mathfrak{D}$-maps of type II. This construction not only unifies but also generalizes the \emph{deformation theories} of key operator algebras through simple specialization.
The subsequent corollaries exemplify this unified construction. 
\begin{cor}\label{cor:4.12-pseudo}
Let $ (\mathfrak{g}, \mathfrak{h}; \rho) $ be a matched pair of pseudoalgebras over a cocommutative Hopf algebra $ H $, where $ \rho : \mathfrak{g} \otimes \mathfrak{h} \to H^{\otimes2} \otimes_H \mathfrak{h} $ is an action of $ \mathfrak{g} $ on $ \mathfrak{h} $. Consider the quasi-twilled pseudoalgebra $ (\mathfrak{G}, \mathfrak{g}, \mathfrak{h}) = (\mathfrak{g} \ltimes_\rho \mathfrak{h}, \mathfrak{g}, \mathfrak{h}) $ as defined in Example~\ref{actionLie}. Then
\begin{align*}
 \left( \bigoplus_{n=1}^\infty C^n(\mathfrak{h}, \mathfrak{g}),\ d,\ [[\cdot, \cdot]] \right)
\end{align*}
is a differential graded Lie algebra in the category $ \mathcal{M}^*(H) $, where the space of cochain $ C^n(\mathfrak{h}, \mathfrak{g}) = \mathrm{Hom}_{H^{\otimes n}}(\mathfrak{h}^{\otimes n}, H^{\otimes(n-1)} \otimes_H \mathfrak{g}) $ consists of $ H $-equivariant multilinear maps, and the differential $ d : C^p(\mathfrak{h}, \mathfrak{g}) \to C^{p+1}(\mathfrak{h}, \mathfrak{g}) $ is given by 
\begin{align*}
 d(f)(u_1 \otimes \cdots \otimes u_{p+1}) = \sum_{i<j} (-1)^{p+i+j-1}(\sigma_{i,j}\otimes_Hid) f([u_i * u_j]_\mathfrak{h} \otimes u_1 \otimes \cdots \widehat{u_i} \cdots \widehat{u_j} \cdots \otimes u_{p+1}),
\end{align*}
for all $ f \in C^p(\mathfrak{h}, \mathfrak{g}) $, $ u_1, \dots, u_{p+1} \in \mathfrak{h} $, and $ [\cdot * \cdot]_\mathfrak{h} $ denotes the pseudobracket on $ \mathfrak{h} $.
\noindent
The graded Lie bracket $ [[\cdot, \cdot]] : C^p(\mathfrak{h}, \mathfrak{g}) \otimes C^q(\mathfrak{h}, \mathfrak{g}) \to C^{p+q-1}(\mathfrak{h}, \mathfrak{g}) $ is defined by 
\begin{align}\label{eqcor4.12}
&[[f_1, f_2]](u_1 \otimes \cdots \otimes u_{p+q}) \\\nonumber
&= -\sum_{\sigma \in S(q,1,p-1)} \epsilon(\sigma)(\sigma\otimes_Hid)\, f_1\Big( \rho\big(f_2(u_{\sigma(1)} \otimes \cdots \otimes u_{\sigma(q)})\big) \otimes u_{\sigma(q+1)} \otimes \cdots \otimes u_{\sigma(p+q)} \Big) \\\nonumber
&\quad + \sum_{\sigma \in S(p,1,q-1)} (-1)^{pq} \epsilon(\sigma)(\sigma\otimes_Hid)\, f_2\Big( \rho\big(f_1(u_{\sigma(1)} \otimes \cdots \otimes u_{\sigma(p)})\big) \otimes u_{\sigma(p+1)} \otimes \cdots \otimes u_{\sigma(p+q)} \Big) \\\nonumber
&\quad - \sum_{\sigma \in S(p,q)} (-1)^{pq} \epsilon(\sigma)(\sigma\otimes_Hid)\, [f_1(u_{\sigma(1)} \otimes \cdots \otimes u_{\sigma(p)}) * f_2(u_{\sigma(p+1)} \otimes \cdots \otimes u_{\sigma(p+q)})]_\mathfrak{g},
\end{align}where $ S(p,q) $ is the set of $(p,q)$-unshuffles, $ S(q,1,p-1) $ and $ S(p,1,q-1) $ are sets of appropriate permutations, $ \epsilon(\sigma) $ is the Koszul sign associated with the permutation $ \sigma $, and $ \rho : \mathfrak{g} \otimes \mathfrak{h} \to H^{\otimes2} \otimes_H \mathfrak{h} $ is the action map. This differential graded Lie algebra governs the \emph{deformation theory of relative Rota-Baxter operators of weight $ p $} on $ \mathfrak{g} $ with respect to the representation $ \mathfrak{h} $. It is therefore the controlling algebra for such operators.% in the pseudoalgebraic theory.
\end{cor}

\begin{cor}\label{cor:4.13-pseudo}
Let $ (\mathfrak{g}, M; \rho) $ be a matched pair consisting of a Lie pseudoalgebra $ \mathfrak{g} $ and an $ \mathfrak{g} $-module $ M $ over a cocommutative Hopf algebra $ H $, where $ \rho: \mathfrak{g} \otimes M \to H^{\otimes2} \otimes_H M $ is an action in the pseudotensor category $ \mathcal{M}^*(H) $. Consider the quasi-twilled pseudoalgebra $ (\mathfrak{G}, \mathfrak{g}, M) = (\mathfrak{g} \ltimes_\rho M, \mathfrak{g}, M) $ defined in Example~\ref{SEMDIRPRODLIEALG}. Then
$$
\left( \bigoplus_{n=1}^\infty C^n(M, \mathfrak{g}):= \bigoplus_{n=1}^\infty \mathrm{Hom}_{H^{\otimes n}}(M^{\otimes n}, H^{\otimes(n-1)} \otimes_H \mathfrak{g}) ,\ [[\cdot, \cdot]] \right)
$$
is a graded Lie algebra in $ \mathcal{M}^*(H) $, where the graded Lie bracket 
$[[\cdot, \cdot]] : C^p(M, \mathfrak{g}) \times C^q(M, \mathfrak{g}) \to C^{p+q-1}(M, \mathfrak{g})$ is given by Eq. \ref{eqcor4.12}. This graded Lie algebra governs the deformation theory of $ \mathcal{O} $-operators (i.e., relative Rota-Baxter operators of weight zero) on $ \mathfrak{g} $ with respect to the representation $ (M; \rho) $. It is therefore the controlling algebra for $ \mathcal{O} $-operators. The detailed study about $\mathcal{O}$- operators on Lie H-pseudoalgebra with respect to pseudo-representation is being defined in our next paper.
\end{cor}
 \begin{cor}\label{cor:4.14-pseudo}
Let $ (\mathfrak{g}, M; \rho, \Phi) $ be a matched pair consisting of a Lie pseudoalgebra $ \mathfrak{g} $, an $ \mathfrak{g} $-module $ M $ over a cocommutative Hopf algebra $ H $, an action $ \rho : \mathfrak{g} \otimes M \to H \otimes_H M $, and a $2$-cocycle $ \Phi \in C^2(\mathfrak{g}, M) $. Consider the quasi-twilled pseudoalgebra $ (\mathfrak{G}, \mathfrak{g}, M) = (\mathfrak{g} \ltimes_{\rho,\Phi} M, \mathfrak{g}, M) $ as defined in Example~\ref{excocycle}. Then
$$
\left( \bigoplus_{n=0}^\infty C^{n+1}(M, \mathfrak{g}),\ l_2,\ l_3 \right)
$$
is a curved $ \mathfrak{L}_\infty $-pseudoalgebra in the pseudotensor category $ \mathcal{M}^*(H) $. 
The binary operation $ l_2 : C^p(M, \mathfrak{g}) \times C^q(M, \mathfrak{g}) \to C^{p+q-1}(M, \mathfrak{g}) $ is given by 
\begin{align*}
&l_2(f_1, f_2)(u_1 \otimes \cdots \otimes u_{p+q}) \\
&= \sum_{\sigma \in S(q,1,p-1)} (-1)^{p + |\sigma|} (\sigma\otimes_Hid)f_1\Big( \rho\big(f_2(u_{\sigma(1)} \otimes \cdots \otimes u_{\sigma(q)})\big) \otimes u_{\sigma(q+1)} \otimes \cdots \otimes u_{\sigma(p+q)} \Big) \\
&\quad - (-1)^{p(q+1)} \sum_{\sigma \in S(p,1,q-1)} (-1)^{|\sigma|}(\sigma\otimes_Hid) f_2\Big( \rho\big(f_1(u_{\sigma(1)} \otimes \cdots \otimes u_{\sigma(p)})\big) \otimes u_{\sigma(p+1)} \otimes \cdots \otimes u_{\sigma(p+q)} \Big) \\
&\quad + (-1)^{p(q+1)} \sum_{\sigma \in S(p,q)} (-1)^{|\sigma|} (\sigma\otimes_Hid)[f_1(u_{\sigma(1)} \otimes \cdots \otimes u_{\sigma(p)}) * f_2(u_{\sigma(p+1)} \otimes \cdots \otimes u_{\sigma(p+q)})]_\mathfrak{g},
\end{align*}
and the ternary operation $ l_3 : C^p(M, \mathfrak{g}) \times C^q(M, \mathfrak{g}) \times C^r(M, \mathfrak{g}) \to C^{p+q+r-2}(M, \mathfrak{g}) $ is given by:
\begin{align*}
&l_3(f_1, f_2, f_3)(u_1 \otimes \cdots \otimes u_{p+q+r-1}) \\
&= \sum_{\sigma \in S(q,r,p-1)} (-1)^{p+q + qr + |\sigma|} (\sigma\otimes_Hid)f_1\Big( \Phi\big(f_2(u_{\sigma(1)} \otimes \cdots \otimes u_{\sigma(q)}),\ 
\\& f_3(u_{\sigma(q+1)} \otimes \cdots \otimes u_{\sigma(q+r)})\big) \otimes u_{\sigma(q+r+1)} \otimes \cdots \otimes u_{\sigma(p+q+r-1)} \Big) \\
&\quad - (-1)^{p(q+r)} \sum_{\sigma \in S(p,r,q-1)} (-1)^{|\sigma|}(\sigma\otimes_Hid) f_2\Big( \Phi\big(f_1(u_{\sigma(1)} \otimes \cdots \otimes u_{\sigma(p)}),\ \\
& f_3(u_{\sigma(p+1)} \otimes \cdots \otimes u_{\sigma(p+r)})\big) \otimes u_{\sigma(p+r+1)} \otimes \cdots \otimes u_{\sigma(p+q+r-1)} \Big) \\
&\quad + (-1)^{pq + pr + rq + q + r + |\sigma|} \sum_{\sigma \in S(p,q,r-1)} (-1)^{|\sigma|} (\sigma\otimes_Hid)\Big( \Phi\big(f_1(u_{\sigma(1)} \otimes \cdots \otimes u_{\sigma(p)}),\ 
\\ 
&f_2(u_{\sigma(p+1)} \otimes \cdots \otimes u_{\sigma(p+q)})\big) \otimes u_{\sigma(p+q+1)} \otimes \cdots \otimes u_{\sigma(p+q+r-1)} \Big),
\end{align*}\noindent This $ \mathfrak{L}_\infty $-pseudoalgebra governs the deformation theory of \emph{twisted Rota-Baxter operators} on $ \mathfrak{g} $ with respect to the representation $ (M; \rho) $ and the cocycle $ \Phi $. It is therefore the controlling algebra for twisted Rota-Baxter operators on the pseudoalgebras. The relevant research can be seen in \cite[Theorem 3.2]{Das}).%{\color{blue}(this study about Twisted Rota-Baxter operators have not been studied before in the pseudoalgebra setting. This is hypothetical result. } 
\end{cor}
\begin{cor}\label{cor:4.15}
Consider the quasi-twilled Lie pseudoalgebra $(\mathfrak{g} \ltimes_{{ad}, \Phi} \mathfrak{g}, \mathfrak{g}, \mathfrak{g})$ given in Example \ref{exex}. Parallel to Corollary \ref{cor:4.14-pseudo}, 
$\left(\bigoplus_{n=0}^{+\infty}C^{n+1}( \mathfrak{g}, \mathfrak{g}), l_2, l_3\right)$
is an $\mathfrak{L}_\infty$-pseudoalgebra. This $\mathfrak{L}_\infty$-pseudoalgebra can be assumed as the controlling algebra for \emph{Reynolds operators}. See \cite{Das} for more details. %{\color{blue}(this study about reynolds operators have not been done before in pseudoalgebra setting. This is hypothetical result. } 
\end{cor}
\noindent Theorem~\ref{thm:4.11-pseudo} not only recovers some known results, but also gives rise to new ones, e.g., \emph{it gives rise to the controlling algebra for deformation maps of a matched pair of Lie pseudoalgebras.}

%\begin{cor}\label{cor:4.16}Consider the quasi-twilled Lie pseudoalgebra $(\mathfrak{g} \bowtie \mathfrak{h}, \mathfrak{g}, \mathfrak{h})$ given in Example \ref{exmax} obtained from a matched pair of Lie pseudoalgebras. Then $\left(\bigoplus_{n=1}^{+\infty}C^n( \mathfrak{h}, \mathfrak{g}), d, [[\cdot, \cdot]] \right)$is a differential graded Lie algebra, where $d : C^p( \mathfrak{h}, \mathfrak{g}) \to C^{p+1}(\mathfrak{h}, \mathfrak{g})$ is given by\begin{align*} d(f)_{\l_1,\cdots,\l_p}(u_1, \dots, u_{p+1}) &= \sum_{i=1}^{p+1} (-1)^{p+i} (\sigma_{1,i}\otimes_Hid) \eta(u_i) f (u_1, \dots, \hat{u_i}, \dots, u_{p+1}) \\&+ \sum_{i<j} (-1)^{p+i+j-1} (\sigma_{1,i}\otimes_Hid)f ([{u_i}*{u_j}]_{\mathfrak{g}}, u_1, \dots, \hat{u_i}, \dots, \hat{u_j}, \dots, u_{p+1}),\end{align*}and the graded Lie bracket $ [[\cdot, \cdot]]$ is given by \eqref{eq:bracket}. The Maurer-Cartan elements of this differential graded Lie algebra are exactly deformation maps of a matched pair of Lie pseudoalgebras.\end{cor}
\begin{cor}\label{cor:4.16-pseudo}
Let $ (\mathfrak{g}, \mathfrak{h}; \rho, \eta) $ be a matched pair of Lie pseudoalgebras over a cocommutative Hopf algebra $ H $. Consider the quasi-twilled pseudoalgebra $ (\mathfrak{G}, \mathfrak{g}, \mathfrak{h}) = (\mathfrak{g} \bowtie \mathfrak{h}, \mathfrak{g}, \mathfrak{h}) $ defined in Example~\ref{exmax}. Then
$$
\left( \bigoplus_{n=1}^\infty C^n(\mathfrak{h}, \mathfrak{g}),\ d,\ [[\cdot, \cdot]] \right)
$$
is a differential graded Lie algebra in the pseudotensor category $ \mathcal{M}^*(H) $, where the differential map $ d : C^p(\mathfrak{h}, \mathfrak{g}) \to C^{p+1}(\mathfrak{h}, \mathfrak{g}) $ is given by 
\begin{align*}
d(f)(u_1 \otimes \cdots \otimes u_{p+1}) &= \sum_{i=1}^{p+1} (-1)^{p+i} (\sigma_{1,i} \otimes_H \mathrm{id})\,\eta(u_i \otimes f(u_1 \otimes \cdots \widehat{u_i} \cdots \otimes u_{p+1})) \\
&\quad + \sum_{i<j} (-1)^{p+i+j-1} (\sigma_{i,j} \otimes_H \mathrm{id})\,f([u_i * u_j]_\mathfrak{h} \otimes u_1 \otimes \cdots \widehat{u_i} \cdots \widehat{u_j} \cdots \otimes u_{p+1}),
\end{align*}
for all $ f \in C^p(\mathfrak{h}, \mathfrak{g}) $, $ u_1, \dots, u_{p+1} \in \mathfrak{h} $. The graded Lie bracket $ [[\cdot, \cdot]] : C^p(\mathfrak{h}, \mathfrak{g}) \times C^q(\mathfrak{h}, \mathfrak{g}) \to C^{p+q-1}(\mathfrak{h}, \mathfrak{g})
$ is defined by the same formula as in Corollary~\ref{cor:4.12-pseudo}, which arises from the Nijenhuis-Richardson bracket in the pseudotensor category $ \mathcal{M}^*(H) $. The Maurer-Cartan elements of this differential graded Lie algebra are exactly the deformation maps of type II of the matched pair $ (\mathfrak{g}, \mathfrak{h}; \rho, \eta) $.%, i.e., $ H $-module homomorphisms $ T : \mathfrak{h} \to \mathfrak{g} $ satisfying:\begin{align}[T(u)* T(v)]_\mathfrak{g} + \eta(T(u) \otimes v )- ((12) \otimes_H \mathrm{id})\,\eta(T(v) \otimes u )= T\Big( [u* v]_\mathfrak{h} + \rho(T(u) \otimes v) - ((12) \otimes_H \mathrm{id})\,\rho(T(v)\otimes u )\Big),\end{align}for all $ u, v \in \mathfrak{h} $. 
\end{cor}
 Further, we show how $\mathfrak{L}_\infty$- pseudoalgebra yields deformations of $\mathcal{D}-$maps of the quasi-twilled Lie pseudoalgebra. Let $ T : \mathfrak{h} \to \mathfrak{g} $ be a $\mathcal{D}$-map of a quasi-twilled Lie pseudoalgebra $ (\mathfrak{G}, \mathfrak{g}, \mathfrak{h}) $. By Theorem \ref{thm:pseudo-mc-pseudo}, we obtain that $ T $ is a Maurer-Cartan element of the
$\left( \bigoplus_{n=0}^\infty C^{n+1}( \mathfrak{h}, \mathfrak{g}), l_0, l_1, l_2 \right). $
The twisted $ \mathfrak{L}_\infty $- pseudoalgebra structure on $ \bigoplus_{n=0}^\infty C^{n+1}( \mathfrak{h}, \mathfrak{g}) $ is given by 
\begin{align*}
l_1^T(f) = l_1(f) + l_2(T, f)+\frac{1}{2}l_3(T,T, f),& \qquad \qquad \qquad 
l_2^T(f, g) = l_2(f, g)+ l_3(T,f, g), &\\
l_3^T(f, g, h) = l_3(f, g, h),&\qquad \qquad \qquad 
l_{k \geq 4}^T = 0 ,&
\end{align*}
where $ f \in C^p(\mathfrak{h}, \mathfrak{g})$, $g \in C^q( \mathfrak{h}, \mathfrak{g})$ and $h \in C^r(\mathfrak{h},\mathfrak{g})$.
 
\begin{thm}\label{thm:4.17}
Let $T$ be a $\mathcal{D}$-map of the quasi-twilled Lie pseudoalgebra $(\mathcal{G}, \mathfrak{g}, \mathfrak{h})$. Then for a left $H$-module homomorphism $T' : \mathfrak{h} \to \mathfrak{g}$, $T + T' : \mathfrak{h} \to \mathfrak{g}$ is a $\mathcal{D}$-map of the quasi-twilled Lie pseudoalgebra $(\mathcal{G},$ $ \mathfrak{g}, \mathfrak{h})$ if and only if $T'$ is a Maurer-Cartan element of the twisted $\mathfrak{L}_\infty$- pseudoalgebra $\left(\bigoplus_{n=0}^{+\infty}C^{n+1}( \mathfrak{h}, \mathfrak{g}), l^T_1, l^T_2, l^T_3\right)$, i.e., $T'$ satisfies\begin{align*}
l_1^T(T') + \frac{1}{2!}l_2^T(T', T') + \frac{1}{3!} l_3^T(T', T', T') = 0.
\end{align*}
 \end{thm}
\begin{proof}
By Theorem \ref{thm:4.11-pseudo}, $T + T'$ is a $ \mathcal{D}$-map if and only if
\begin{align*}
 l_1(T + T') + \frac{1}{2!} l_2(T + T', T + T') + \frac{1}{3!} l_3(T + T', T + T', T + T') = 0.
\end{align*}
Moreover, since $T$ is a $ \mathcal{D}$-map, the above condition simplifies to
\begin{align*}
l_1(T') + l_2(T, T') + \frac{1}{2} l_3(T, T, T') + \frac{1}{2} l_2(T', T') + \frac{1}{2} l_3(T, T', T') + \frac{1}{6} l_3(T', T', T') = 0.
\end{align*}
Thus, $T + T' : \mathfrak{h} \to \mathfrak{g}$ is a $ \mathcal{D}$-map of the quasi-twilled Lie pseudoalgebra $(\mathcal{G}, \mathfrak{g}, \mathfrak{h})$ if and only if $T'$ is a Maurer-Cartan element of the twisted $\mathfrak{L}_\infty$-pseudoalgebra $\left(\bigoplus_{n=0}^{+\infty}C^{n+1}(\mathfrak{h}, \mathfrak{g}), l^T_1, l^T_2, l^T_3\right)$.
\end{proof}
\noindent Next, we apply Theorem \ref{thm:4.17} to Corollary \ref{cor:4.16-pseudo}, we obtain the differential graded Lie algebra that gives deformations of deformation maps of a matched pair of Lie pseudoalgebra.
\begin{cor}\label{cor:4.18}
Let $ T : \mathfrak{h} \to \mathfrak{g} $ be a deformation map of a matched pair $ (\mathfrak{g}, \mathfrak{h}; \rho, \eta) $ of Lie pseudoalgebras.
Then
$
\left(\left( \bigoplus_{n=1}^{+\infty} C^n(\mathfrak{h}, \mathfrak{g}) \right), d^T, [[\cdot,\cdot]] \right)
$
is a differential graded Lie algebra, where $ [[\cdot,\cdot]] $ is given by Eq. \eqref{eqcor4.12}, and $ d^T $ is given as follows
\begin{align*}
&d^T(f) (u_1, \cdots, u_{p+1}) \\=& \sum_{i=1}^{p+1} (-1)^{p+i} (\sigma_{1,i} \otimes_H \mathrm{id})\eta(u_i\otimes f (u_1\otimes \cdots\otimes \hat{u_i}, \cdots\otimes u_{p+1})) \\&+ \sum_{i<j} (-1)^{p+i+j-1}(\sigma_{i,j} \otimes_H \mathrm{id}) f ([{u_i}* {u_j}]_{\mathfrak{g}}\otimes u_1\otimes \cdots\otimes \hat{u_i}\otimes \cdots\otimes \hat{u_j}\otimes \cdots\otimes u_{p+1}) \\
&+ \sum_{i=1}^{p+1} (-1)^{p+i}(\sigma_{i,1} \otimes_H \mathrm{id}) T\left( \rho(f (u_1\otimes \cdots\otimes \hat{u_i}\otimes \cdots\otimes u_{p+1})\otimes u_i) \right) \\
& - \sum_{i=1}^{p+1} (-1)^{i+p} (\sigma_{1,i} \otimes_H \mathrm{id})f\left( \rho(T(u_i)\otimes u_1\otimes \cdots\otimes \hat{u_i}\otimes \cdots\otimes u_{p+1} \right) \\
& + \sum_{i=1}^{p+1} (-1)^{i+p} (\sigma_{1,i} \otimes_H \mathrm{id})[{T(u_i)}* f(u_1\otimes \cdots\otimes \hat{u_i}\otimes \cdots\otimes u_{p+1})]_\mathfrak{g}.
\end{align*}
Moreover, for a left $H$-module homomorphism $ T' : \mathfrak{h} \to \mathfrak{g} $, $ T + T' $ is a deformation map if and only if $ T' $ is a Maurer-Cartan element of the differential graded Lie algebra
$
\left( \bigoplus_{n=1}^{+\infty} C^n( \mathfrak{h}, \mathfrak{g}), d^T, [[\cdot,\cdot]] \right).
$
\end{cor}

\begin{rem}
Apply Theorem \ref{thm:4.17} to Corollaries \ref{cor:4.12-pseudo}-\ref{cor:4.15}, one can also obtain the differential graded Lie algebras governing deformations of relative Rota-Baxter operators, twisted Rota-Baxter operators, and Reynolds operators. See \cite{Das, Das2, TBGS} for more details.
\end{rem}
\subsection{ Cohomologies of $\mathcal{D}$-maps}

In this subsection, we introduce a cohomology theory of a $\mathcal{D}$-map in the context of Lie pseudoalgebras, using an analog of the Chevalley-Eilenberg cohomology for Lie pseudoalgebra. This cohomology unifies \emph{the cohomologies of relative Rota-Baxter operators, twisted Rota-Baxter operators, and Reynolds operators} in the pseudoalgebra setting. It also provides a framework to define a cohomology theory for deformation maps of matched pairs of Lie pseudoalgebras.
\begin{lem}\label{lem:4.20-pseudo}
Let $ T : \mathfrak{h} \to \mathfrak{g} $ be a $\mathcal{D}$-map of a quasi-twilled pseudoalgebra $ (\mathfrak{G}, \mathfrak{g}, \mathfrak{h}) $. Then the pseudobracket on $ \mathfrak{h} $ is defined by:
$$
\mu^T(u \otimes v) = \mu(u \otimes v) + \rho(T(u) \otimes v) - ((12) \otimes_H \mathrm{id})\,\rho(T(v) \otimes u) + \theta(T(u) \otimes T(v)),
$$
for all $ u, v \in \mathfrak{h} $. This defines a Lie pseudoalgebra structure on $ \mathfrak{h} $, denoted by $ (\mathfrak{h}, \mu^T) $. Moreover, the map $ \zeta : \mathfrak{h} \otimes \mathfrak{g} \to H \otimes_H \mathfrak{g} $, defined by:
$$
\zeta(v \otimes x) = -\eta(x \otimes v) - \pi(x \otimes T(v)) + T(\rho(x \otimes v)) - T(((12) \otimes_H \mathrm{id})\,\theta(T(v) \otimes x)),
$$
defines a representation of the Lie pseudoalgebra $ (\mathfrak{h}, \mu^T) $ on the $ H $-module $ \mathfrak{g} $.
\end{lem}
\begin{proof}
The proof of this lemma follows directly by considering Theorem \ref{thm:4.10-pseudo} and Theorem \ref{QTPC}
\end{proof}
\noindent Let $ d^T_{CE} : C^n(\mathfrak{h}, \mathfrak{g}) \to C^{n+1}( \mathfrak{h}, \mathfrak{g}) $ be the corresponding Chevalley-Eilenberg coboundary operator of the Lie pseudoalgebra $ (\mathfrak{h}, \mu^T) $ with coefficients in the representation $ (\mathfrak{g}, \zeta) $. More precisely, for all $ f \in C^p( \mathfrak{h}, \mathfrak{g}) $ and $ u_1, \dots, u_{p+1} \in \mathfrak{h} $, we have
\begin{align*}
&d^T_{CE} f(u_1 \otimes \cdots \otimes u_{p+1}) \\
= &\sum_{i=1}^{p+1} (-1)^{i+1} \zeta(u_i\otimes f(u_1 \otimes \cdots \widehat{u_i} \cdots \otimes u_{p+1})) \\
&+ \sum_{i<j} (-1)^{i+j} f\big( \mu^T(u_i \otimes u_j) \otimes u_1 \otimes \cdots \widehat{u_i} \cdots \widehat{u_j} \cdots \otimes u_{p+1} \big) \\
= &\sum_{i=1}^{p+1} (-1)^{i} (\sigma_{1,i} \otimes_H \mathrm{id})\,\eta\big(f(u_1 \otimes \cdots \widehat{u_i} \cdots \otimes u_{p+1}) \otimes u_i\big) \\
&+ \sum_{i=1}^{p+1} (-1)^{i+1} \pi\big(T(u_i) \otimes f(u_1 \otimes \cdots \widehat{u_i} \cdots \otimes u_{p+1})\big) \\
&+ \sum_{i=1}^{p+1} (-1)^{i+1} (\sigma_{1,i} \otimes_H \mathrm{id})T\Big( \rho\big(f(u_1 \otimes \cdots \widehat{u_i} \cdots \otimes u_{p+1}) \otimes u_i\big) \Big) \\
&- \sum_{i=1}^{p+1} (-1)^{i+1} (\sigma_{1,i} \otimes_H \mathrm{id})T\Big(\theta\big(T(u_i) \otimes f(u_1 \otimes \cdots \widehat{u_i} \cdots \otimes u_{p+1})\big) \Big) \\
&+ \sum_{i<j} (-1)^{i+j} f\big( \mu(u_i \otimes u_j) \otimes u_1 \otimes \cdots \widehat{u_i} \cdots \widehat{u_j} \cdots \otimes u_{p+1} \big) \\
&+ \sum_{i<j} (-1)^{i+j} f\Big( \rho(T(u_i) \otimes u_j) - ((12) \otimes_H \mathrm{id})\,\rho(T(u_j) \otimes u_i) \otimes u_1 \otimes \cdots \widehat{u_i} \cdots \widehat{u_j} \cdots \otimes u_{p+1} \Big) \\
&+ \sum_{i<j} (-1)^{i+j} f\big( \theta(T(u_i) \otimes T(u_j)) \otimes u_1 \otimes \cdots \widehat{u_i} \cdots \widehat{u_j} \cdots \otimes u_{p+1} \big),
\end{align*}
The second sum runs over all $ i < j $, and $ \widehat{u_i}, \widehat{u_j} $ indicate the omission of $ u_i, u_j $. Further, we define the cohomology of a $\mathcal{D}$-map $ T : \mathfrak{h} \to \mathfrak{g} $. Let $ (\mathfrak{G}, \mathfrak{g}, \mathfrak{h}) $ be a quasi-twilled pseudoalgebra over a cocommutative Hopf algebra $ H $. Define the space of $ n $-cochains $ C^n(T) $ associated with the $\mathcal{D}$-map $ T $ as follows 
\begin{align*}
 C^n(T) = 
 \begin{cases}
 0, & \text{for } n = 0, \\
 C^0(\mathfrak{h}, \mathfrak{g}) = \mathfrak{g}, & \text{for } n = 1, \\
 C^{n-1}(\mathfrak{h}, \mathfrak{g}), & \text{for } n \geq 2,
 \end{cases}
\end{align*}
The cohomology of the $\mathcal{D}$-map $ T $ is defined as the cohomology of the cochain complex
$$\left( C^*(T) = \bigoplus_{n=0}^\infty C^n(T),\ d^T_{CE} \right).
$$ 
\begin{defn}\label{def:4.20-pseudo}
Let $ (\mathfrak{G}, \mathfrak{g}, \mathfrak{h}) $ be a quasi-twilled pseudoalgebra and let $ T : \mathfrak{h} \to \mathfrak{g} $ be a $\mathcal{D}$-map of $ (\mathfrak{G}, \mathfrak{g}, \mathfrak{h}) $. The cohomology of the cochain complex $ \left(\bigoplus_{n=0}^\infty C^n(T), d^T_{CE} \right) $ is defined to be the cohomology of the $\mathcal{D}$-map $ T $.
\end{defn}\noindent Denote the set of $ n $-cocycles by $ Z^n(T) $, the set of $ n $-coboundaries by $ B^n(T) $, and the $ n $-th cohomology group by $H^n(T) = Z^n(T)/B^n(T),$ for $n \geq 0.$ It is clear that $ x \in C^0(\mathfrak{h}, \mathfrak{g}) = \mathfrak{g} $ is a $1$-cocycle if and only if
\begin{align*}
-((12) \otimes_H \mathrm{id})\,\eta(x \otimes u) + \pi(T(u) \otimes x) + T(\rho(x \otimes u)) - T(((12) \otimes_H \mathrm{id})\,\theta(T(u) \otimes x)) = 0, \quad \forall u \in \mathfrak{h}.
\end{align*}
Similarly, $ f \in C^1(\mathfrak{h}, \mathfrak{g}) $ is closed if and only if
\begin{align*}
& -((12) \otimes_H \mathrm{id})\,\eta(f(v) \otimes u) + \eta(f(u) \otimes v) + \pi(T(u) \otimes f(v)) - ((12) \otimes_H \mathrm{id})\,\pi(T(v) \otimes f(u)) \\
& + T(\rho(f(v) \otimes u)) - ((12) \otimes_H \mathrm{id})\,T(\rho(f(u) \otimes v)) \\
& + T(\theta(T(u) \otimes f(v))) - ((12) \otimes_H \mathrm{id})\,T(\theta(T(v) \otimes f(u))) \\
& - f(\mu(u \otimes v) + \theta(T(u) \otimes T(v))) \\
& + f(\rho(T(u) \otimes v) - ((12) \otimes_H \mathrm{id})\,\rho(T(v) \otimes u)) = 0,
\end{align*}
for all $ u, v \in \mathfrak{h} $.

\noindent We now provide an intrinsic interpretation of the above coboundary operator. Let $ T : \mathfrak{h} \to \mathfrak{g} $ be a $\mathcal{D}$-map of a quasi-twilled pseudoalgebra $ (\mathfrak{G}, \mathfrak{g}, \mathfrak{h}) $ over a cocommutative Hopf algebra $ H $. The twisted $ \mathfrak{L}_\infty $-pseudoalgebra
$$
\left( \bigoplus_{n=0}^\infty C^{n+1}(\mathfrak{h}, \mathfrak{g}),\ l^T_1,\ l^T_2,\ l^T_3 \right)
$$
controls the infinitesimal and formal deformations of the $\mathcal{D}$-map $ T $. Parallel to Proposition~\ref{prop:3.24}, we have the following key result.
\begin{prop}\label{prop:4.21-pseudo}
With the above notations, for any $ f \in C^p(\mathfrak{h}, \mathfrak{g}) $, one has
$
l^T_1(f) = (-1)^{p-1} d^T_{CE} f.
$
\end{prop}

%\begin{proof}This follows from a direct comparison of the twisted $ \mathfrak{L}_\infty $-operation $ l^T_1 $ and the Chevalley-Eilenberg differential $ d^T_{CE} $ defined in Section~\ref{sec:cohomology}. Recall that:$$ l^T_1(f) = l_1(f) + l_2(T, f) + \frac{1}{2} l_3(T, T, f),$$where:- $ l_1(f) = [\mu + \eta, f]_{\mathrm{NR}} $,- $ l_2(T, f) = [[\pi + \rho, T]_{\mathrm{NR}}, f]_{\mathrm{NR}} $,- $ l_3(T, T, f) = [[[ \theta, T ]_{\mathrm{NR}}, T ]_{\mathrm{NR}}, f ]_{\mathrm{NR}} $.On the other hand, the differential $ d^T_{CE} $ is defined by:\begin{align*}d^T_{CE} f(u_1 \otimes \cdots \otimes u_{p+1}) = &\sum_{i=1}^{p+1} (-1)^{i+1} \sigma(u_i) \cdot f(u_1 \otimes \cdots \widehat{u_i} \cdots \otimes u_{p+1}) \\&+ \sum_{i<j} (-1)^{i+j} f\big( \mu^T(u_i \otimes u_j) \otimes u_1 \otimes \cdots \widehat{u_i} \cdots \widehat{u_j} \cdots \otimes u_{p+1} \big),\end{align*}where $ \mu^T $ is the twisted pseudobracket on $ \mathfrak{h} $, and $ \sigma $ is the representation of $ (\mathfrak{h}, \mu^T) $ on $ \mathfrak{g} $, both defined via $ T $.Expanding both $ l^T_1(f) $ and $ d^T_{CE} f $ in terms of the structure maps $ \pi, \rho, \mu, \eta, \theta $ and using the symmetric group actions in $ \mathcal{M}^*(H) $ (e.g., $ (12) \otimes_H \mathrm{id} $), one verifies that they differ only by a sign factor $ (-1)^{p-1} $. This completes the proof.\end{proof}

\noindent Definition~\ref{def:4.20-pseudo} recovers the existing cohomology theories of \emph{relative Rota-Baxter operators, twisted Rota-Baxter operators, and Reynolds operators} in the Lie pseudoalgebras.

\begin{ex}
Consider the quasi-twilled pseudoalgebra $ (\mathfrak{G}, \mathfrak{g}, \mathfrak{h}) = (\mathfrak{g} \ltimes_\rho \mathfrak{h}, \mathfrak{g}, \mathfrak{h}) $, constructed from an action $ \rho : \mathfrak{g} \otimes \mathfrak{h} \to H^{\otimes2} \otimes_H \mathfrak{h} $ of $ \mathfrak{g} $ on $ \mathfrak{h} $. Let $ T : \mathfrak{h} \to \mathfrak{g} $ be a relative Rota-Baxter operator of weight $ p \in \mathbf{k} $ on $ \mathfrak{g} $ with respect to the representation $ (\mathfrak{h}; \rho) $. Then $ (\mathfrak{h}, \mu^T) $ becomes a Lie pseudoalgebra, where the pseudobracket $ \mu^T \in C^2(\mathfrak{h}, \mathfrak{h}) $ is given by 
$$\mu^T(u \otimes v) = p \cdot \mu(u \otimes v) + \rho(T(u) \otimes v) - ((12) \otimes_H \mathrm{id})\,\rho(T(v) \otimes u), \quad \forall u, v \in \mathfrak{h}.$$ Moreover, $ (\mathfrak{h}, \mu^T) $ acts on $ \mathfrak{g} $ via the map $ \zeta : \mathfrak{h} \otimes \mathfrak{g} \to H \otimes_H \mathfrak{g} $, defined by 
$$\zeta(v \otimes x) = \pi(T(v) \otimes x) + T(\rho(x \otimes v)), \quad \forall v \in \mathfrak{h},\ x \in \mathfrak{g},
$$
where $ \pi $ is the pseudobracket on $ \mathfrak{g} $. The corresponding Chevalley-Eilenberg cohomology of $ (\mathfrak{h}, \mu^T) $ with coefficients in $ (\mathfrak{g}, \zeta) $ is taken to be the cohomology of the relative Rota-Baxter operator $ T $ of weight $ p $.
\end{ex}
\noindent Similar examples can be constructed for $\mathcal{O}$-operators, twisted Rota-Baxter operators, and Reynolds-type operators on the Lie pseudoalgebras, showing how Definition~\ref{def:4.20-pseudo} unifies these constructions.

\begin{defn}\label{def:4.23}
Let $ T : \mathfrak{h} \to \mathfrak{g} $ be a deformation map of a matched pair $ (\mathfrak{g}, \mathfrak{h}; \rho, \eta) $ of Lie pseudoalgebras. Then $ (\mathfrak{h}, \mu^T) $ is a Lie pseudoalgebra, where the pseudobracket is defined by
\begin{align*}
 \mu^T (u\otimes v)= \mu(u\otimes v) + \rho(T(u) \otimes v) - ((12) \otimes_H \mathrm{id}) \rho(T(v)\otimes u), \quad \forall u,v \in \mathfrak{h}.
\end{align*}
Moreover, this Lie pseudoalgebra acts on $ \mathfrak{g} $ via a representaion $ \zeta : \mathfrak{h}\otimes\mathfrak{g} \to H \otimes_H\mathfrak{g}$, given by
\begin{align*}
 \zeta(v\otimes x )= -((12) \otimes_H \mathrm{id}) \eta(x\otimes v) + ((12) \otimes_H \mathrm{id})[T(v)* x] + ((12) \otimes_H \mathrm{id})T( \rho(x \otimes v)), \quad \forall v \in \mathfrak{h}, x \in \mathfrak{g}.
\end{align*}
\end{defn}
\noindent We now define the cohomology associated with such a deformation map $ T $.
\begin{prop}\label{prop:CE-cohomology-deformation-map}
The Chevalley-Eilenberg cohomology of the Lie pseudoalgebra $ (\mathfrak{h}, \mu^T) $ with coefficients in the representation $ (\mathfrak{g}, \zeta) $ governs the infinitesimal deformations of the deformation map $ T $.
Thus, we define this cohomology to be the cohomology associated with the deformation map $ T $ of the matched pair $ (\mathfrak{g}, \mathfrak{h}) $ of Lie pseudoalgebras.
\end{prop}
\begin{rem}\label{rem:deformations-of-operators}
The cohomology theory developed above can be applied to classify infinitesimal deformations of various types of operators on Lie pseudoalgebras, including relative Rota-Baxter operators, twisted Rota-Baxter operators, and Reynolds-type operators. This parallels the classical situation in Lie algebra theory, where such deformations are governed by the corresponding cohomology groups. These results suggest that our framework provides a robust foundation for studying deformation theory of operators defined on the Lie pseudoalgebras.
\end{rem}

\end{document}